\documentclass{amsart}
\usepackage[margin=1in]{geometry}

\usepackage[T1]{fontenc}
\usepackage{comment}
\newtheorem{theorem}{Theorem}[section]
\theoremstyle{definition}
\newtheorem{definition}[theorem]{Definition}
\theoremstyle{remark}
\newtheorem{remark}[theorem]{Remark}
\newtheorem{proposition}[theorem]{Proposition}
\newtheorem{corollary}[theorem]{Corollary}
\newtheorem{example}[theorem]{Example}
\newtheorem{exercise}[theorem]{Exercise}

\usepackage{amsthm,amsmath,amssymb,amsfonts,amscd,tikz-cd}
\usepackage{graphicx}
\usepackage{bbm}
\usepackage[colorinlistoftodos]{todonotes}
\usepackage{comment}
\usepackage{tikz}
\usetikzlibrary{arrows.meta,decorations.markings,arrows}
\usepackage{faktor}
\usepackage{forloop}
\usepackage{pgffor}
\usepackage{caption}
\usepackage{subcaption}
\usepackage{subfiles}
\usepackage{wrapfig}
\usepackage{bm}

\title{Lecture notes on Heegaard Floer homology}
\author{C.-M. Michael Wong and Sarah Zampa}
\date{}

\DeclareMathOperator{\rk}{rk}
\DeclareMathOperator{\Sym}{Sym}

\newcommand{\Z}{\mathbb{Z}}

\newcommand{\R}{\mathbb{R}}

\DeclareMathOperator{\Id}{Id}

\DeclareMathOperator{\grad}{grad}
\DeclareMathOperator{\Ind}{Ind}
\DeclareMathOperator{\HFhat}{\widehat{HF}}

\newcommand{\bG}{\mathbb{G}}
\newcommand{\bF}{\mathbb{F}}

\newcommand{\ie}{\textit{i.e.,}\ }
\newcommand{\eg}{\textit{e.g.,}\ }
\newcommand{\cf}{\textit{cf.}\ }
\newcommand{\etal}{\textit{et al}}
\newcommand{\Lint}{L_0\cap L_1}
\newcommand{\s}{\mathfrak{s}}

\DeclareMathOperator{\CM}{CM}
\DeclareMathOperator{\HM}{HM}

\DeclareMathOperator{\CF}{CF}
\DeclareMathOperator{\wHF}{HF} 
\DeclareMathOperator{\CFK}{CFK}
\DeclareMathOperator{\HFK}{HFK}
\newcommand*{\CFc}{\CF^{\mathord{\circ}}}
\newcommand*{\HFc}{\wHF^{\mathord{\circ}}}
\newcommand*{\CFh}{\widehat{\CF}}
\newcommand*{\CFt}{\widetilde{\CF}}
\newcommand*{\CFm}{\CF^{\mathord{-}}}
\newcommand*{\CFp}{\CF^{\mathord{+}}}
\newcommand*{\CFi}{\CF^{\infty}}
\newcommand*{\HFh}{\widehat{\wHF}}

\newcommand*{\wHFm}{\wHF^{\mathord{-}}}

\newcommand*{\CFKc}{\CFK^{\mathord{\circ}}}
\newcommand*{\HFKc}{\HFK^{\mathord{\circ}}}
\newcommand*{\CFKh}{\widehat{\CFK}}
\newcommand*{\HFKh}{\widehat{\HFK}}
\newcommand*{\CFKm}{\CFK^{\mathord{-}}}
\newcommand*{\HFKm}{\HFK^{\mathord{-}}}
\newcommand*{\CFKi}{\CFK^{\infty}}

\newcommand*{\gCFKc}{\mathrm{gCFK}^{\mathord{\circ}}}

\newcommand*{\heeg}{\mathcal{H}}
\newcommand*{\pathsp}{\mathcal{P}}
\newcommand*{\action}{\mathcal{A}}
\newcommand*{\wring}{\mathcal{R}} 

\newcommand*{\heegsurf}{\Sigma}
\newcommand*{\alphas}{\boldsymbol{\alpha}}
\newcommand*{\betas}{\boldsymbol{\beta}}

\DeclareMathOperator{\Spinc}{Spin^c}

\newcommand*{\OO}{\mathbb{O}}
\newcommand*{\XX}{\mathbb{X}}

\newcommand*{\JJ}{\mathcal{J}}

\DeclareMathOperator{\GC}{GC}
\DeclareMathOperator{\GH}{GH}
\newcommand*{\GCh}{\widehat{\GC}}
\newcommand*{\GCt}{\widetilde{\GC}}
\newcommand*{\GCm}{\GC^{\mathord{-}}}
\newcommand*{\GCp}{\GC^{\mathord{+}}}
\newcommand*{\GCi}{\GC^{\infty}}
\newcommand*{\GHh}{\widehat{\GH}}
\newcommand*{\GHt}{\widetilde{\GH}}
\newcommand*{\GHm}{\GH^{\mathord{-}}}

\newcommand*{\genset}{\mathcal{S}}

\newcommand*{\bdy}{\partial}
\DeclareMathOperator{\bdyt}{\widetilde{\bdy}}
\newcommand*{\bdyh}{\widehat{\bdy}}
\newcommand*{\bdym}{\bdy^{\mathord{-}}}

\newcommand*{\pentmap}{P}
\newcommand*{\hexmap}{H}

\DeclareMathOperator{\Rect}{Rect}
\newcommand*{\eRect}{\Rect^{\mathord{\circ}}}
\DeclareMathOperator{\Pent}{Pent}
\newcommand*{\ePent}{\Pent^{\mathord{\circ}}}

\DeclareMathOperator{\Tor}{Tor}
\newcommand*{\blank}{\cdot}

\newcommand*{\moduli}{\mathcal{M}}
\newcommand*{\acsfamily}{\mathcal{J}}

\newcommand*{\wfilt}{\mathcal{F}}
\newcommand{\G}{\mathbb{G}}
\DeclareMathOperator{\x}{\mathbf{x}}
\DeclareMathOperator{\y}{\mathbf{y}}
\DeclareMathOperator{\z}{\mathbf{z}}
\DeclareMathOperator{\w}{\mathbf{w}}

\begin{document}
\maketitle
\vspace{-7mm}

\begin{abstract}
These are the notes for a lecture series on Heegaard Floer homology, given by the first author at the R\'{e}nyi Institute in January 2023, as part of a special semester titled ``Singularities and Low Dimensional Topology’’. Familiarity with Heegaard diagrams and Morse theory is assumed. We first illustrate the relevant algebraic structures via grid homology, and then highlight the geometric rather than combinatorial nature of the general theory. We then define Heegaard Floer homology in the context of Lagrangian Floer homology, and describe some key properties.
\end{abstract}

\tableofcontents
\clearpage

% GRID DIAGRAM %
\begin{comment}
    \begin{figure}
\centering
\begin{tikzpicture}
\draw[step=1cm,black,very thin] (0,0) grid (5,5);
% a cross
\path (0.15,3.15) node(x) {} 
      (0.85,3.85) node(y) {};
\draw[thick] (x) -- (y);
\path (0.15,3.85) node(x) {} 
      (0.85,3.15) node(y) {};
\draw[thick] (x) -- (y);
% a circle
\draw[thick] (3.5,3.5) circle (0.3);
% the horizontal arrow (from circle to cross)
\draw[{To[length=4mm, width=3mm]}-, thick] (0.85,3.5) -- (3.2,3.5);

%  %  %  %  %  %  %  %

\draw[fill=white, white] (1.4,3.4) rectangle (1.6,3.6); 

%  %  %  %  %  %  %  %

\draw[thick] (1.5,4.5) circle (0.3);
\path (1.15,1.15) node(x) {} 
      (1.85,1.85) node(y) {};
\draw[thick] (x) -- (y);
\path (1.15,1.85) node(x) {} 
      (1.85,1.15) node(y) {};
\draw[thick] (x) -- (y);

\draw[{To[length=4mm, width=3mm]}-, thick] (1.5,1.85) -- (1.5,4.2);
\end{tikzpicture}
\caption{Grid diagram of the left-handed trefoil}
\end{figure}
\end{comment}

Heegaard Floer homology has grown to be a vast subject in low-dimensional 
topology since its inception in the early 2000s. It is thus impossible to cover 
in four lectures its numerous features, variants, applications, and connections 
to other theories.  These lecture notes are aimed at students and researchers 
who are not yet acquainted with the theory, and represent my attempt at 
providing the context of how various ideas fit in it, in a short amount of 
time.  As a result, we will inevitably be vague at times and omit details that 
the expert will find important. We will also have to omit or glide over many 
significant contributions by members of the Floer theory community, and the 
reader is encouraged to fill in these gaps by consulting other references.

%\section{First lecture}
\section*{Overview} %\addcontentsline{toc}{section}{Overview}

Floer homology is a kind of Morse homology on an
infinite-dimensional space. A particular example is Heegaard Floer homology, a $3$-manifold invariant defined by Ozsv\'ath and Szab\'o in the early 2000s.
The following is a simplified overview of one way to construct it, as well as some properties.

Given a $3$-manifold $Y$, we take a Heegaard diagram $\heeg$ for $Y$, and use $\heeg$ to construct an infinite-dimensional manifold $\pathsp (\heeg)$ with a functional $\action$. With some further choices, we then apply ``Morse homology'' to the pair $(\pathsp (\heeg), \action)$ to get a chain complex $\CFc (\heeg)$ over a nice ring $\wring$ (for example, the field $\bF = \bF_2$ of two elements), whose homotopy type turns out to depend only on $Y$---under further constraints on $\heeg$ and the choices along the way. Consequently, we obtain invariants $\CFc (Y)$ and $\HFc (Y)$, where $\HFc (Y)$ is the homology of $\CFc (Y)$.  
There are also ``relative'' invariants $\CFKc (Y, K)$ and $\HFKc (Y, K)$ for knots $K \subset Y$, defined similarly; note, however, that $\CFKc (Y, K)$ is often a filtered complex and $\HFKc (Y, K)$ is the \emph{associated   graded homology} of $\CFKc (Y, K)$ rather than its honest homology.

%We present a 3-manifold $Y$ (or a pair
%$(Y,K)$, where $K\subset Y$ is a knot)
%through a Heegaard diagram ${\mathcal {H}}$, and from this data 
%we construct an
%infinite-dimensional manifold $\mathcal{P}(\mathcal{H})$ 
%with a ``Morse'' function
%$\mathcal{A}$ (really, a functional).  We can now apply ``Morse
%homology'' to the infinite-dimensional space
%$(\mathcal{P}(\mathcal{H}), \mathcal{A})$ (with some further
%choices). From it, we get an $\mathcal{R}$-chain complex
%$\CF^\bullet(\mathcal{H})$ whose chain homotopy type turns out to
%depend only on $Y$ (or $(Y,K)$). Here, $\mathcal{R}$ is some nice
%ring (for example the field $\mathbb{F}=\mathbb{F}_2$ of two
%elements, or some polynomial ring).
%We can also consider the homology $\HF^\bullet(Y)$ (or
%$\HFK^\bullet(Y,K)$ in the knot case) of the chain complex, which is
%an $\mathcal{R}$-module. The isomorphism class of this
%homology module (sometimes under further constraints on the
%Heegaard diagram) turns out to be an invariant of $Y$ or $(Y,K)$.

%Before discussing the ideas that go into the definition, we present
Here are
some basic structural properties of the Heegaard Floer complex:
%Indeed,
%Heegaard Floer homology has some interesting structural properties:
%
\begin{itemize}
  \item It admits a direct sum decomposition by elements of $\Spinc(Y)$, the set
    of $\Spinc$-structures on $Y$;
  \item It has a homological $\Z / m$-grading, known as the \emph{Maslov 
      grading}.
    In many cases (\eg when $b_1 (Y) = 0$), we have $m=0$,
    and so we get a $\Z$-graded chain complex;
  \item There are many different flavors: $\CFh$, $\CFm$, $\CFp$, $\CFi$, 
    $\CFt$,\footnote{As we shall see later, the \emph{tilde} version is almost 
      but not quite an invariant.} which are defined over different rings 
    $\wring$. (The notation $\CFc$ denotes any of these flavors.)
\end{itemize}

% algebraic foundations first, geometry/analytic details later

In learning Heegaard Floer theory, it can be somewhat 
overwhelming to try to absorb the large amount of geometry and analysis 
required in its definition while keeping track of the relevant algebraic 
structures. Thus, in most of the first two lectures, we will separate these two 
aspects, focusing on the algebra first via a combinatorial model of $\CFKc 
(S^3, K)$ known as \emph{grid homology}.

\section{Grid homology}
\subsection{Definition and $\bdy^2 = 0$}
Grid homology was first defined in \cite{MOS09}, and the first combinatorial 
proof of its invariance appeared in \cite{MOST07}. A comprehensive treatise can 
be found in \cite{OSS15:GH-book}, which we will mostly be following. (However, we will 
be omitting a lot of details and proofs, and the reader is encouraged to 
consult \cite{OSS15:GH-book} for a fuller understanding.)

\begin{definition}%[\cite{MOS09, MOST07, OSS15}]
  \label{def:grid-diagram-cpx}
  A (toroidal) \emph{grid diagram} $\G=(n, \OO, \XX)$  of an oriented knot 
  $K\subset S^3$ is a diagram of an $n \times n$ grid, with a set $\OO$ of 
  $O$'s and a set $\XX$ of $X$'s, such that there is exactly one $O$ and one 
  $X$ in each row and in each column. (There is often an additional condition 
  that any square is not occupied by an $O$ and an $X$ at the same time.) The 
  diagram is to be thought of as lying on a torus, with the leftmost and 
  rightmost (resp.\ topmost and bottommost) lines identified into circles.
See Figure~\ref{fig:grid-trefoil} for an example.
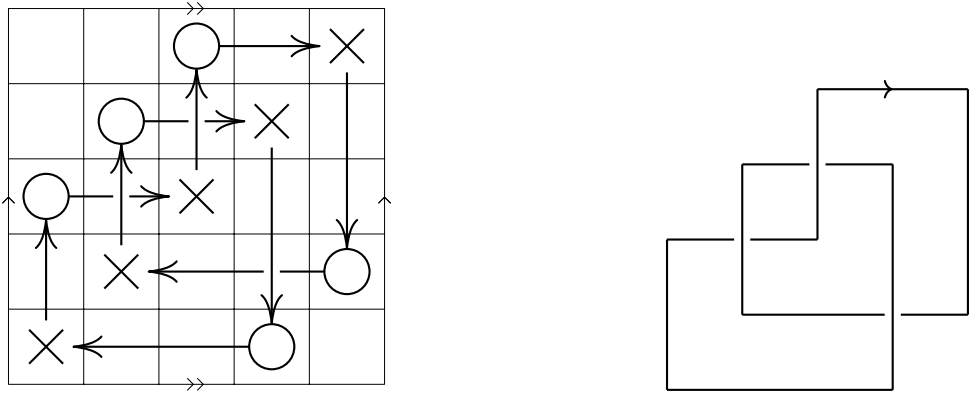
\begin{figure}[!htbp]
\centering
\begin{subfigure}{.5\textwidth}
  \centering
  \begin{tikzpicture}
\draw[step=1cm,black,very thin] (0,0) grid (5,5);

% the O's
\draw[thick] (0.5,2.5) circle (0.3);
\draw[thick] (1.5,3.5) circle (0.3);
\draw[thick] (2.5,4.5) circle (0.3);
\draw[thick] (3.5,0.5) circle (0.3);
\draw[thick] (4.5,1.5) circle (0.3);

%  %  %  %  %  %  %  %

% the X's
\foreach \n in {0,...,4}{
    \path (\n+0.15,\n+0.15) node(x) {} 
        (\n+0.85,\n+0.85) node(y) {};
    \draw[thick] (x) -- (y);
    \path (\n+0.15,\n+0.85) node(x) {} 
        (\n+0.85,\n+0.15) node(y) {};
    \draw[thick] (x) -- (y);
}

% the horizontal arrow (from circle to cross)
% (y-coordinate increasing) [radius circle=0.3// size X=0.7x0.7]
% x<-O %
\draw[{To[length=4mm, width=3mm]}-, thick] (0.85,0.5) -- (3.2,0.5);
\draw[{To[length=4mm, width=3mm]}-, thick] (1.85,1.5) -- (4.2,1.5);
% O->X %
\draw[-{To[length=4mm, width=3mm]}, thick] (0.8,2.5) -- (2.15,2.5);
\draw[-{To[length=4mm, width=3mm]}, thick] (1.8,3.5) -- (3.15,3.5);
\draw[-{To[length=4mm, width=3mm]}, thick] (2.8,4.5) -- (4.15,4.5);

%  %  %  %  %  %  %  %

% makes white underline
% \draw[fill=white, white] (x+0.4,y+0.4) rectangle (x+0.6,y+0.6); 
\draw[fill=white, white] (3.4,1.4) rectangle (3.6,1.6);
\draw[fill=white, white] (1.4,2.4) rectangle (1.6,2.6);
\draw[fill=white, white] (2.4,3.4) rectangle (2.6,3.6);

%  %  %  %  %  %  %  %

% the vertical arrows
\draw[-{To[length=4mm, width=3mm]}, thick] (0.5,0.85) -- (0.5,2.2);
\draw[-{To[length=4mm, width=3mm]}, thick] (1.5,1.85) -- (1.5,3.2);
\draw[-{To[length=4mm, width=3mm]}, thick] (2.5,2.85) -- (2.5,4.2);
\draw[{To[length=4mm, width=3mm]}-, thick] (3.5,0.8) -- (3.5,3.15);
\draw[{To[length=4mm, width=3mm]}-, thick] (4.5,1.8) -- (4.5,4.15);

%  %  %  %  %  %  %  %

%the arrows denoting the torus
\draw[Straight Barb-Straight Barb] (0,2.5);
\draw[Straight Barb-Straight Barb] (5,2.5);
\draw[-{Straight Barb[sep=1pt]. Straight Barb[]}] (2.445,5)--(2.6,5);
\draw[-{Straight Barb[sep=1pt]. Straight Barb[]}] (2.445,0)--(2.6,0);

\end{tikzpicture}
\end{subfigure}%
\begin{subfigure}{.5\textwidth}
  \centering
  \begin{tikzpicture}

% the horizontal lines
% x<-O %
    \draw[thick] (0.5,0.5) -- (3.5,0.5);
\draw[thick] (1.5,1.5) -- (4.5,1.5);
% O->X %
\draw[thick] (0.5,2.5) -- (2.5,2.5);
\draw[thick] (1.5,3.5) -- (3.5,3.5);

%  %  %  %  %  %  %  %  %
% the horizontal line with arrow
\begin{scope}[thick,decoration={
    markings,
    mark=at position 0.5 with {\arrow{>}}}
    ] 
    \draw[postaction={decorate}] (2.5,4.5) -- (4.5,4.5);
\end{scope}

%  %  %  %  %  %  %  %

% makes white underline
% \draw[fill=white, white] (x+0.4,y+0.4) rectangle (x+0.6,y+0.6); 
\draw[fill=white, white] (3.4,1.4) rectangle (3.6,1.6);
\draw[fill=white, white] (1.4,2.4) rectangle (1.6,2.6);
\draw[fill=white, white] (2.4,3.4) rectangle (2.6,3.6);

%  %  %  %  %  %  %  %

% the vertical lines
\draw[thick] (0.5,0.5) -- (0.5,2.5);
\draw[thick] (1.5,1.5) -- (1.5,3.5);
\draw[thick] (2.5,2.5) -- (2.5,4.5);
\draw[thick] (3.5,0.5) -- (3.5,3.5);
\draw[thick] (4.5,1.5) -- (4.5,4.5);

\end{tikzpicture}
\end{subfigure}
\caption{A grid diagram of the left-handed trefoil}
\label{fig:grid-trefoil}
\end{figure}
\end{definition}

The diagram $\G$ encodes a knot $K$ as follows: Connect the $X$'s to the $O$'s vertically, and the $O$'s to the $X$'s
horizontally, with the vertical strands crossing over the horizontal
strands; this gives a rectilinear projection of an oriented knot $K$.
%
% should i make 
%
We say that $\G$ is a \emph{grid diagram of $K$}.
It is not hard to see that any knot $K\subset S^3$ (and indeed any
link in $S^3$) can be presented this way.

We now define a chain complex associated to a grid diagram.

\begin{definition}
Given a grid diagram $\G$, we define the generating set
\[
  \genset (\G)=\{\text{matchings between vertical and horizontal circles in } \G\},
\]
where a matching is a one-to-one correspondence between the vertical
and horizontal circles.  A matching is denoted on the grid diagram by
dots on the intersections between vertical and horizontal circles.
Figure~\ref{fig:grid-trefoil-gen} illustrates two different matchings, in red and in green respectively.

\begin{comment}
    \begin{remark}
We can forget about the 6th vertical (or horizontal) line in our grid diagram of the trefoil, because those are identified with the first vertical (or horizontal) line.
\end{remark}
\end{comment}

The \emph{tilde grid chain complex} is now defined as the $\bF$-module
\[
  \widetilde{\GC}(\G)=\bF\langle \genset(\G) \rangle,
\]
with boundary homomorphism $\bdyt \colon \GCt (\G) \to \GCt (\G)$ given by
\[
  \bdyt (\x) = \sum_{\y \in \genset (\G)} \sum_{\substack{r \in \eRect (\x, \y) 
      \\ r \cap (\OO \cup \XX) = \emptyset}} \y,
\]
where $\eRect (\x, \y)$ denotes the space of empty rectangles from $\x$ to $\y$.
A \emph{rectangle $r$ from $\x$ to $\y$} is a $2$-chain on the (toroidal) grid 
diagram, such that $\bdy r$ consists exactly of two arcs on horizontal circles 
from a point in $\x$ to a point in $\y$ and two arcs on vertical circles from a 
point in $\y$ to a point in $\x$, in its induced orientation.
A succinct way to say the above is that $\bdy r$ consists of 4 arcs, and
\[
  \bdy ((\bdy r) \cap (\text{horizontal circles})) = \y - \x.
\]
Such a rectangle exists only if $\x$ and $\y$ coincide everywhere except at two 
points each. A rectangle is \emph{empty} if it does not contain any points in 
$\x$ (or equivalently $\y$) in its interior.
\end{definition}

In other words, the differential counts rectangles $r$ whose northeast and 
southwest corners are in $\x$, whose northwest and southeast corners are in 
$y$, and which do not contain any other points in $\x$ or $\y$. In this 
\emph{tilde} flavor, the contributing rectangles $r$ are further required to 
not contain any $O$'s or $X$'s in its interior, expressed in the requirement
$r\cap (\OO \cup \XX)=\emptyset$.

\begin{example}
Consider the diagram of the left-handed trefoil again.
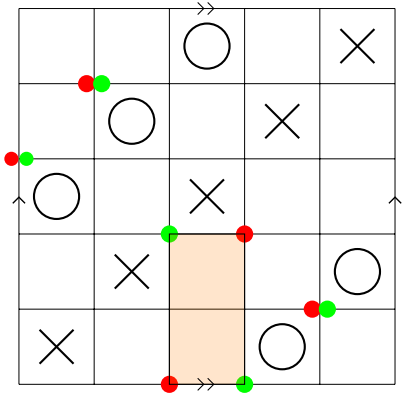
\begin{figure}[!htbp]
\centering
\begin{tikzpicture}
\draw[step=1cm,black,very thin] (0,0) grid (5,5);

% the O's
\draw[thick] (0.5,2.5) circle (0.3);
\draw[thick] (1.5,3.5) circle (0.3);
\draw[thick] (2.5,4.5) circle (0.3);
\draw[thick] (3.5,0.5) circle (0.3);
\draw[thick] (4.5,1.5) circle (0.3);

%  %  %  %  %  %  %  %

% the X's
\foreach \n in {0,...,4}{
    \path (\n+0.15,\n+0.15) node(x) {} 
        (\n+0.85,\n+0.85) node(y) {};
    \draw[thick] (x) -- (y);
    \path (\n+0.15,\n+0.85) node(x) {} 
        (\n+0.85,\n+0.15) node(y) {};
    \draw[thick] (x) -- (y);
}

\begin{comment}
    % the horizontal arrow (from circle to cross)
% (y-coordinate increasing) [radius circle=0.3// size X=0.7x0.7]
% x<-O %
\draw[{To[length=4mm, width=3mm]}-, thick] (0.85,0.5) -- (3.2,0.5);
\draw[{To[length=4mm, width=3mm]}-, thick] (1.85,1.5) -- (4.2,1.5);
% O->X %
\draw[-{To[length=4mm, width=3mm]}, thick] (0.8,2.5) -- (2.15,2.5);
\draw[-{To[length=4mm, width=3mm]}, thick] (1.8,3.5) -- (3.15,3.5);
\draw[-{To[length=4mm, width=3mm]}, thick] (2.8,4.5) -- (4.15,4.5);

%  %  %  %  %  %  %  %

% makes white underline
% \draw[fill=white, white] (x+0.4,y+0.4) rectangle (x+0.6,y+0.6); 
\draw[fill=white, white] (3.4,1.4) rectangle (3.6,1.6);
\draw[fill=white, white] (1.4,2.4) rectangle (1.6,2.6);
\draw[fill=white, white] (2.4,3.4) rectangle (2.6,3.6);

%  %  %  %  %  %  %  %

% the vertical arrows
\draw[-{To[length=4mm, width=3mm]}, thick] (0.5,0.85) -- (0.5,2.2);
\draw[-{To[length=4mm, width=3mm]}, thick] (1.5,1.85) -- (1.5,3.2);
\draw[-{To[length=4mm, width=3mm]}, thick] (2.5,2.85) -- (2.5,4.2);
%
\draw[{To[length=4mm, width=3mm]}-, thick] (3.5,0.8) -- (3.5,3.15);
\draw[{To[length=4mm, width=3mm]}-, thick] (4.5,1.8) -- (4.5,4.15);
\end{comment}

%     the generators of S(\G)     %
\draw[thick, red, fill] (-0.1,3) circle (0.08);
\draw[thick, red, fill] (0.9,4) circle (0.1);
\draw[thick, red, fill] (2,0) circle (0.1);
\draw[thick, red, fill] (3,2) circle (0.1);
\draw[thick, red, fill] (3.9,1) circle (0.1);
\draw[thick, green, fill] (0.1,3) circle (0.08);
\draw[thick, green, fill] (1.1,4) circle (0.1);
\draw[thick, green, fill] (2,2) circle (0.1);
\draw[thick, green, fill] (3,0) circle (0.1);
\draw[thick, green, fill] (4.1,1) circle (0.1);

%   shade the rectangle   %

\draw[fill=orange, fill opacity=0.2] (2,0) rectangle (3,2);

%  %  %  %  %  %  %  %

%the arrows denoting the torus
\draw[Straight Barb-Straight Barb] (0,2.5);
\draw[Straight Barb-Straight Barb] (5,2.5);
\draw[-{Straight Barb[sep=1pt]. Straight Barb[]}] (2.445,5)--(2.6,5);
\draw[-{Straight Barb[sep=1pt]. Straight Barb[]}] (2.445,0)--(2.6,0);

\end{tikzpicture}
\caption{Two generators of $S(\G)$ on a diagram $\G$ of the left-handed 
  trefoil}
\label{fig:grid-trefoil-gen}
\end{figure}

The generators $\x$ (in red) and $\y$ (in green)
coincide everywhere, except on the vertices of the
empty rectangle (shaded in orange) from $\x$ to $\y$.
\end{example}

Of course, one needs to show that $\GCt$ is indeed a chain complex.

\begin{proposition}
  $\bdyt$ is a boundary map; \ie
  $\bdyt\circ\bdyt=0$.
\end{proposition}

\begin{proof}
Suppose that $\z$ is a generator that appears as a term in $\bdyt \circ \bdyt
(\x)$. We will show that there are always an even number of ways to connect
$\x$ to $\z$ with two empty rectangles in order. Since $\GCt$ is
defined over $\bF=\bF_2$, this will complete the proof.

Marking $\x$ in red and $\z$ in green in the diagrams below,
there are three different possible configurations of the
rectangles. The first two configurations are where the rectangles do not share 
a vertex, as in Figure~\ref{fig:proof-tildemap-zeroA}.
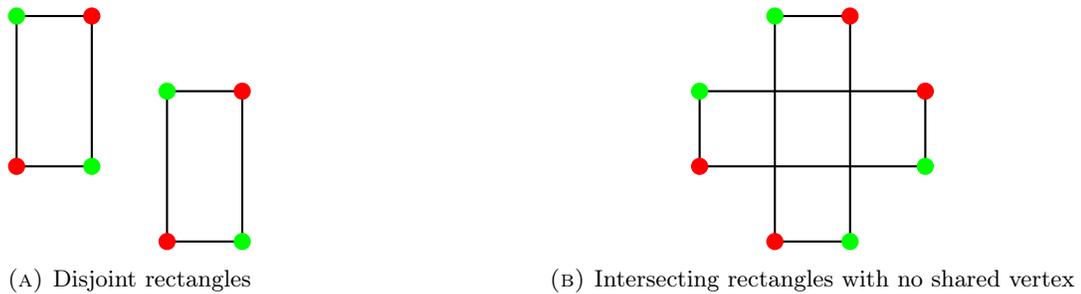
\begin{figure}[!htbp]
     \centering
     \begin{subfigure}[b]{0.4\textwidth}
         \centering
         \begin{tikzpicture}
% vertical lines %
\draw[thick] (0,1) -- (0,3);
\draw[thick] (1,1) -- (1,3);
\draw[thick] (2,0) -- (2,2);
\draw[thick] (3,0) -- (3,2);

% horizontal lines %
\draw[thick] (0,1) -- (1,1);
\draw[thick] (0,3) -- (1,3);
\draw[thick] (2,0) -- (3,0);
\draw[thick] (2,2) -- (3,2);

% generators %
\draw[thick, red, fill] (0,1) circle (0.1);
\draw[thick, red, fill] (1,3) circle (0.1);
\draw[thick, red, fill] (2,0) circle (0.1);
\draw[thick, red, fill] (3,2) circle (0.1);

\draw[thick, green, fill] (0,3) circle (0.1);
\draw[thick, green, fill] (1,1) circle (0.1);
\draw[thick, green, fill] (2,2) circle (0.1);
\draw[thick, green, fill] (3,0) circle (0.1);
    
\end{tikzpicture}
         \caption{Disjoint rectangles}
         \label{fig:proof-tildemap-zero-1}
     \end{subfigure}
     \hfill
     \begin{subfigure}[b]{0.5\textwidth}
         \centering
         \begin{tikzpicture}
% vertical lines %
\draw[thick] (0,1) -- (0,2);
\draw[thick] (1,0) -- (1,3);
\draw[thick] (2,0) -- (2,3);
\draw[thick] (3,1) -- (3,2);

% horizontal lines %
\draw[thick] (0,1) -- (3,1);
\draw[thick] (0,2) -- (3,2);
\draw[thick] (1,0) -- (2,0);
\draw[thick] (1,3) -- (2,3);

% generators %
\draw[thick, red, fill] (0,1) circle (0.1);
\draw[thick, red, fill] (1,0) circle (0.1);
\draw[thick, red, fill] (2,3) circle (0.1);
\draw[thick, red, fill] (3,2) circle (0.1);

\draw[thick, green, fill] (0,2) circle (0.1);
\draw[thick, green, fill] (1,3) circle (0.1);
\draw[thick, green, fill] (2,0) circle (0.1);
\draw[thick, green, fill] (3,1) circle (0.1);

\end{tikzpicture}
            \caption{Intersecting rectangles with no shared vertex}
            \label{fig:proof-tildemap-zero-2}
     \end{subfigure}
        %\caption{Different grid diagrams of the unknot}
        \caption{Two rectangles with no shared vertex}
        \label{fig:proof-tildemap-zeroA}
\end{figure}

In these two cases, there are two distinct ways to
  count the rectangles: either counting $r_1$ first then $r_2$---connecting 
  $\x$ to some $\y$ and then to $\z$---or
  counting $r_2$ first then $r_1$---connecting $\x$ to some $\w$ and then to 
  $\z$.\footnote{Technically, using $r_i$ to denote the rectangles in different 
    orders is a slight abuse of notation, since, \eg  the first $r_1$ belongs 
    to $\eRect (\x, \y)$, while the second $r_1$ belongs to $\eRect (\w, \z)$.}

The third case consists of L-shaped domains that arise from juxtaposing two
rectangles sharing a single vertex, as in 
Figure~\ref{fig:proof-tildemap-zeroB}, and its reflections and rotations.  In 
this case, with the same
configuration for $\x$ (in red) and $\z$ (in green), there are two ways of 
decomposing such
an L-shaped domain into rectangles.

\begin{figure}[!ht]
     \centering
     \begin{subfigure}[b]{0.22\textwidth}
         \centering

         \begin{tikzpicture}
    % vertical lines %
    \draw[thick] (0,0) -- (0,1);
\draw[thick] (1,0) -- (1,3);
\draw[thick] (2,0) -- (2,3);

% horizontal lines %
\draw[thick] (0,0) -- (2,0);
\draw[thick] (0,1) -- (2,1);
\draw[thick] (1,3) -- (2,3);

% generators %

% shaded rectangles %
\draw[fill=orange, fill opacity=0.2] (0,0) rectangle (1,1);
\draw[fill=purple, fill opacity=0.2] (1,0) rectangle (2,3);

% third generator %
\draw[thick, red, fill] (0,0) circle (0.1);
\draw[thick, red, fill] (1,1) circle (0.1);
\draw[thick, red, fill] (1.9,3) circle (0.1);

\draw[thick, green, fill] (-0.1,1) circle (0.1);
\draw[thick, green, fill] (1,3) circle (0.1);
\draw[thick, green, fill] (2,0) circle (0.1);

\draw[thick, blue, fill] (0.1,1) circle (0.1);
\draw[thick, blue, fill] (1,0) circle (0.1);
\draw[thick, blue, fill] (2.1,3) circle (0.1);

\end{tikzpicture}
            %\caption{Intersecting rectangles with shared vertex}
            %\label{fig:proof-tildemap-zero-3}
\end{subfigure}
    %\hfill
    \begin{subfigure}[b]{0.52\textwidth}
         \centering
         \begin{tikzpicture}
    % vertical lines %
    \draw[thick] (0,0) -- (0,1);
\draw[thick] (1,0) -- (1,3);
\draw[thick] (2,0) -- (2,3);

% horizontal lines %
\draw[thick] (0,0) -- (2,0);
\draw[thick] (0,1) -- (2,1);
\draw[thick] (1,3) -- (2,3);

% generators %
\draw[thick, red, fill] (0,0) circle (0.1);
\draw[thick, red, fill] (1,1) circle (0.1);
\draw[thick, red, fill] (2,3) circle (0.1);

\draw[thick, green, fill] (0,1) circle (0.1);
\draw[thick, green, fill] (1,3) circle (0.1);
\draw[thick, green, fill] (2,0) circle (0.1);

%  arrows pointing outwards  %
\draw [Straight Barb-](-2,1.5) -- (-0.5,1.5);
\draw [-Straight Barb](2.5,1.5) -- (4,1.5);

\end{tikzpicture}
            %\caption{Intersecting rectangles with shared vertex}
            %\label{fig:proof-tildemap-zero-3}
\end{subfigure}
    %\hfill
    \begin{subfigure}[b]{0.22\textwidth}
    \begin{tikzpicture}
%% SECOND RECTANGLE %%
% vertical lines %
    \draw[thick] (0,0) -- (0,1);
\draw[thick] (1,0) -- (1,3);
\draw[thick] (2,0) -- (2,3);

% horizontal lines %
\draw[thick] (0,0) -- (2,0);
\draw[thick] (0,1) -- (2,1);
\draw[thick] (1,3) -- (2,3);

% generators %

% shaded rectangles %
\draw[fill=purple, fill opacity=0.2] (0,0) rectangle (2,1);
\draw[fill=orange, fill opacity=0.2] (1,1) rectangle (2,3);

% third generator %
\draw[thick, red, fill] (-0.1,0) circle (0.1);
\draw[thick, red, fill] (1,1) circle (0.1);
\draw[thick, red, fill] (2,3) circle (0.1);

\draw[thick, green, fill] (0,1) circle (0.1);
\draw[thick, green, fill] (0.9,3) circle (0.1);
\draw[thick, green, fill] (2,0) circle (0.1);

\draw[thick, black, fill] (0.1,0) circle (0.1);
\draw[thick, black, fill] (1.1,3) circle (0.1);
\draw[thick, black, fill] (2,1) circle (0.1);

\end{tikzpicture}

         %\caption{$\mathcal{H}_3$}
         \label{fig:proof-tildemap-zero-3}
     \end{subfigure}
        \caption{Intersecting rectangles with a shared vertex}
        \label{fig:proof-tildemap-zeroB}
\end{figure}
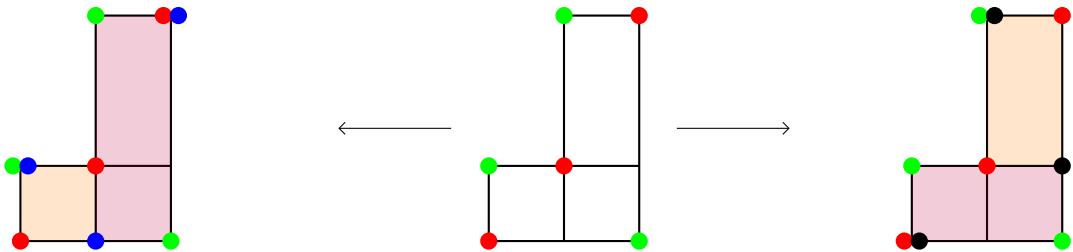
\begin{comment}
    \begin{wrapfigure}{r}{0.1\textwidth}
    \centering
    \includegraphics[width=0.4\textwidth]{images/proof-tilde-map.png}
\end{wrapfigure}
\end{comment}

As in the left of Figure~\ref{fig:proof-tildemap-zeroB}, we can first count the 
bottom left rectangle (shaded in orange), arriving at the intermediate 
generator $\y$ (in blue), before counting the right rectangle (in red), ending 
in $\z$.  Or, as in the right of the figure, we can first count the top right
rectangle (shaded in red), giving the intermediate generator $\w$ (in black), 
before counting the bottom rectangle (in orange), again ending in $\z$.
%Denoting the blue
%generator as $\bm{y}$ and the black one as $\bm{t}$, we find that there is a
%rectangle from $\x$ to $\bm{y}$ and one from $\bm{y}$ to $\z$,
%and another
%rectangle from $\x$ to $\bm{t}$ and one from $\bm{t}$ to $\z$.
This shows that such
decompositions always come in pairs, %so they cancel mod 2,
concluding the argument.
\end{proof}

\begin{remark}
  One can similarly define $\GCt (\G)$ over $\mathbb{Z}$, but it requires 
  adding signs to all definitions and proofs, which represents significantly 
  more work.
\end{remark}

In a similar fashion, we can define the \emph{minus} flavor of the grid chain 
complex.

\begin{definition}
  The \emph{minus grid chain complex} is the $\bF [U_1, \dotsc, U_n]$-module
  \[
    \GCm (\G)=\bF[U_1,\dotsc,U_n]\langle \genset(\G)\rangle,
  \]
  with boundary homomorphism
  \(
    \bdym \colon \GC^-(\G) \to \GC^-(\G)
  \)
  given by
  \[
    \bdym (\x) = \sum_{\y \in \genset (\G)} \sum_{\substack{r \in \eRect (\x, 
        \y) \\ r \cap \XX = \emptyset}} U_1^{O_1 (r)} \dotsm U_n^{O_n (r)} \y,
  \]
  where $O_i(r)$ is the number of times $O_i$ appears in $r$, which is either 
  $0$ or $1$.
  Here, we have fixed an (arbitrary) indexing of the elements of $\OO$, from 
  $O_1$ to $O_n$.
\end{definition}

In other words, in the \emph{minus} flavor, we allow $O$'s in the rectangles 
that are counted in the boundary homomorphism, and record them by formal 
variables $U_i$. We continue to block all $X$'s as before.

%   same type of O or not to use (function vs element of \OO

\begin{proposition}
  $\bdym$ is also a boundary map, \ie $\bdym \circ \bdym = 0$.
\end{proposition}
\begin{proof}
  The proof is identical to that for $\bdyt$, \emph{mutatis mutandis}.
\end{proof}

Since $\GCt (\G)$ and $\GCm (\G)$ are chain complexes, we can compute their 
homologies, denoted $\GHt (\G)$ and $\GHm (\G)$ respectively.

\subsection{Invariance}

The natural question to ask is if $\GHt (\G)$ and $\GHm (\G)$ are invariants of 
the knot $K$ that $\G$ encodes. The following example shows that $\GHt (\G)$ is 
not an invariant.

%Now, one can wonder if the homologies $\GHtilde(\G)$ and $\GH^-(\G)$
%of the previously defined chain complexes are invariants of knots; we
%will see with a direct computation that this is not necessarily the case.

\begin{example}
  Figure~\ref{fig:three-diagrams-unknot} shows three grid diagrams for the 
  unknot. (The leftmost $1 \times 1$ grid does not satisfy the additional 
  condition in Definition~\ref{def:grid-diagram-cpx} but is otherwise a fine 
  grid diagram.)
%Consider the case of the unknot. There are different possible diagrams for the 
%unknot, and we will compute $\GHtilde(\G)$ for each of those.
\begin{figure}[!htbp]
     \centering
     \begin{subfigure}[b]{0.15\textwidth}
         \centering
         \begin{tikzpicture}
\draw[step=1cm,black,very thin] (0,0) grid (1,1);

% the O's
\draw[thick] (0.25,0.5) circle (0.2);
%\draw[thick] (0.75,0.5) circle (0.2);

%  %  %  %  %  %  %  %

% the X's  !!!   REDOOOOOOOOOOO (NOT GOOD)   !!!
\path (0.45,0.2) node(x) {} 
        (1,0.8) node(y) {};
    \draw[thick] (x) -- (y);
    \path (0.45,0.8) node(x) {} 
        (1,0.2) node(y) {};
    \draw[thick] (x) -- (y);

% generators %
\draw[thick, red, fill] (0,0) circle (0.1);

%  %  %  %  %  %  %  %

%the arrows denoting the torus
\draw[Straight Barb-Straight Barb] (0,0.6);
\draw[Straight Barb-Straight Barb] (1,0.6);
\draw[-{Straight Barb[sep=1pt]. Straight Barb[]}] (0.445,1)--(0.6,1);
\draw[-{Straight Barb[sep=1pt]. Straight Barb[]}] (0.445,0)--(0.6,0);
    
\end{tikzpicture}
         \caption{$\G_1$}
         \label{fig:unknot-h1}
     \end{subfigure}
     %\hfill
     \begin{subfigure}[b]{0.3\textwidth}
         \centering
         \begin{tikzpicture}

\draw[step=1cm,black,very thin] (0,0) grid (2,2);

% the O's
\draw[thick] (0.5,1.5) circle (0.3);
\draw[thick] (1.5,0.5) circle (0.3);

%  %  %  %  %  %  %  %

% the X's
\foreach \n in {0,...,1}{
    \path (\n+0.15,\n+0.15) node(x) {} 
        (\n+0.85,\n+0.85) node(y) {};
    \draw[thick] (x) -- (y);
    \path (\n+0.15,\n+0.85) node(x) {} 
        (\n+0.85,\n+0.15) node(y) {};
    \draw[thick] (x) -- (y);
}

% generators %
\draw[thick, red, fill] (0,0) circle (0.1);
\draw[thick, red, fill] (1,1) circle (0.1);
\draw[thick, green, fill] (1,0) circle (0.1);
\draw[thick, green, fill] (0,1) circle (0.1);

%  %  %  %  %  %  %  %

%the arrows denoting the torus
\draw[Straight Barb-Straight Barb] (0,1.35);
\draw[Straight Barb-Straight Barb] (2,1.35);
\draw[-{Straight Barb[sep=1pt]. Straight Barb[]}] (1.245,2)--(1.4,2);
\draw[-{Straight Barb[sep=1pt]. Straight Barb[]}] (1.245,0)--(1.4,0);

\end{tikzpicture}
         \caption{$\G_2$}
         \label{fig:unknot-h2}
     \end{subfigure}
     %\hfill
     \begin{subfigure}[b]{0.45\textwidth}
         \centering
         \begin{tikzpicture}
\draw[step=1cm,black,very thin] (0,0) grid (3,3);

% the O's
\draw[thick] (0.5,2.5) circle (0.3);
\draw[thick] (1.5,0.5) circle (0.3);
\draw[thick] (2.5,1.5) circle (0.3);

%  %  %  %  %  %  %  %

% the X's
\foreach \n in {0,...,1}{
    \path (\n+0.15,\n+1.15) node(x) {} 
        (\n+0.85,\n+1.85) node(y) {};
    \draw[thick] (x) -- (y);
    \path (\n+0.15,\n+1.85) node(x) {} 
        (\n+0.85,\n+1.15) node(y) {};
    \draw[thick] (x) -- (y);
}
\path (2.15,0.15) node(x) {} 
        (2.85,0.85) node(y) {};
    \draw[thick] (x) -- (y);
    \path (2.15,0.85) node(x) {} 
        (2.85,0.15) node(y) {};
    \draw[thick] (x) -- (y);
    
%  %  %  %  %  %  %  %

%the arrows denoting the torus
\draw[Straight Barb-Straight Barb] (0,1.5);
\draw[Straight Barb-Straight Barb] (3,1.5);
\draw[-{Straight Barb[sep=1pt]. Straight Barb[]}] (1.445,3)--(1.6,3);
\draw[-{Straight Barb[sep=1pt]. Straight Barb[]}] (1.445,0)--(1.6,0);

\end{tikzpicture}
         \caption{$\G_3$}
         \label{fig:unknot-h3}
     \end{subfigure}
        \caption{Grid diagrams of the unknot of different sizes}
        \label{fig:three-diagrams-unknot}
\end{figure}
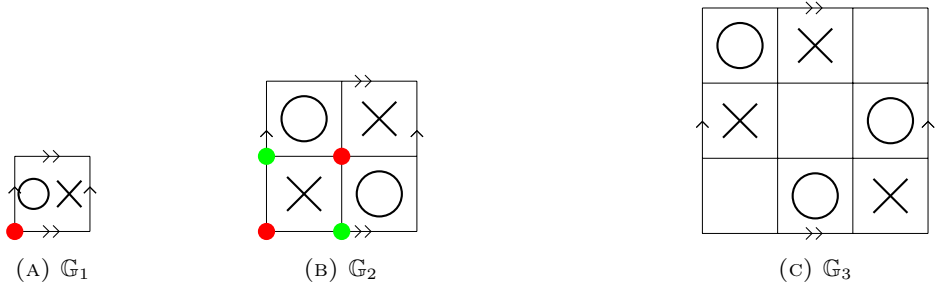

%(The $1\times 1$ grid is a special case which allows two markings in the same 
%box.)
We can compute:% the grid homologies using the different diagrams:
\begin{itemize}
    \item $\widetilde{\GH}(\G_1)\cong\widetilde{\GC}(\G_1)\cong\bF$,
      because there is only one generator (in red in Figure
      \ref{fig:unknot-h1}), and $\bdyt \equiv 0$ as there are no rectangles;
    \item $\widetilde{\GH}(\G_2) \cong \GCt (\G_2) \cong\bF\oplus\bF$, since 
      there are two
      matchings, and again $\bdyt \equiv 0$ as each of the four empty 
      rectangles contains an $O$ or an $X$;
    \item $\widetilde{\GH}(\G_3)\cong\bF^{\oplus 4}$, left as an exercise for 
      the reader. (There are $3!=6$ generators in $\GCt (\G_3)$.)
\end{itemize}
%
%We obtain different homologies depending on the chosen grid diagram representing the same knot, therefore $\widetilde{\GH}$ is not a knot invariant.
\end{example}

What about $\GHm (\G)$? It turns out that it is an invariant of $K$, but in a 
slightly subtle way. In particular, note that if $\G_1$ and $\G_2$ are grid 
diagrams of $K$ of different sizes, $\GCm (\G_1)$ and $\GCm (\G_2)$ are \emph{a 
  priori} modules over different rings.

How might one prove the invariance of grid homology? To do so, one can use the 
following theorem by Cromwell, that may be thought of as the ``Reidemeister 
Theorem'' for grid diagrams.

%To obtain knot invariants, we need to find relations between grid
%diagrams representing the same knot.

\begin{theorem}[Cromwell \cite{Crom95:embedding}]
Two grid diagrams $\G_1$ and $\G_2$ represent the same knot $K\subset
S^3$ if and only if they are related by a finite sequence of row and column 
commutations, and
(de)stabilizations.
%(These moves are shown by the next two figures.)
%Examples of these moves are shown in Figures~\ref{fig:grid-commutation} and 
%\ref{fig:grid-de-stabilization}.
\end{theorem}

\begin{figure}[!htbp]
    \centering
    \includegraphics[scale=0.28]{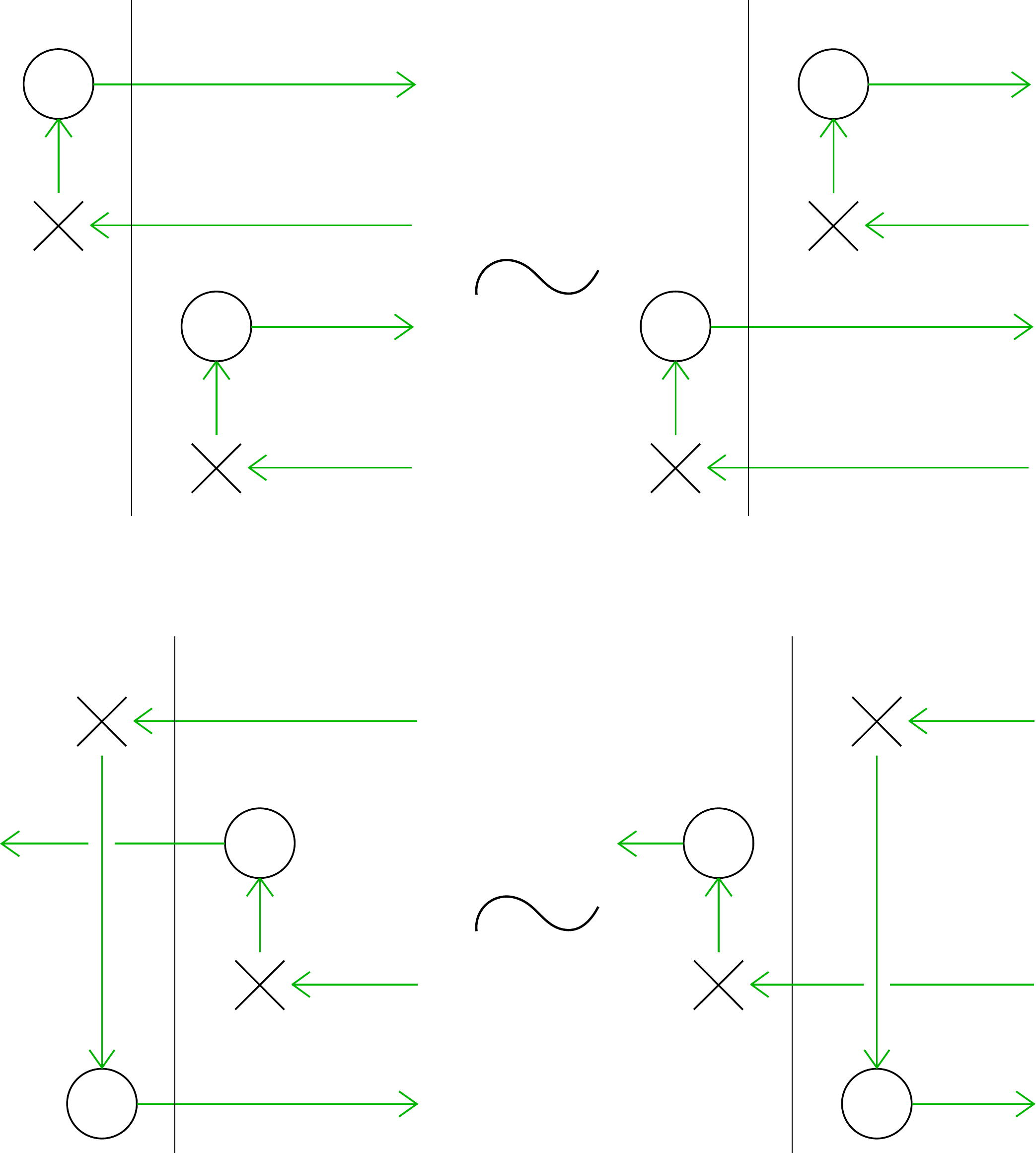}
    \caption{Two column commutations}
    \label{fig:grid-commutation}
\end{figure}

A \emph{column commutation} swaps two adjacent columns when
either the two corresponding vertical strands project to disjoint vertical 
intervals,
%(as in Figure~\ref{fig:grid-commutation}),
or one projected vertical interval
is contained in the interior of the other.
See Figure~\ref{fig:grid-commutation} for examples.
A \emph{row commutation} is similar but swaps adjacent rows.

%  ??? PUT OTHER ALLOWED COMMUTATION ????   %

Given a square in the grid diagram that is occupied by either an $X$ or an $O$, 
a \emph{stabilization} splits the row and column of the square the symbol is in into two rows 
and two columns, replacing the $1 \times 1$ marked square with a $2 \times 2$ 
grid of which three squares are marked. For a given marker type (say, an $X$), 
there are four distinct ways to achieve this, corresponding to which square in 
the $2 \times 2$ grid is unmarked. The rest of the diagram is left
unchanged. See Figure~\ref{fig:grid-de-stabilization} for examples.  
\emph{Destabilization} is the process inverse to stabilization.

\begin{figure}[!htbp]
    \centering
    \includegraphics[scale=0.26]{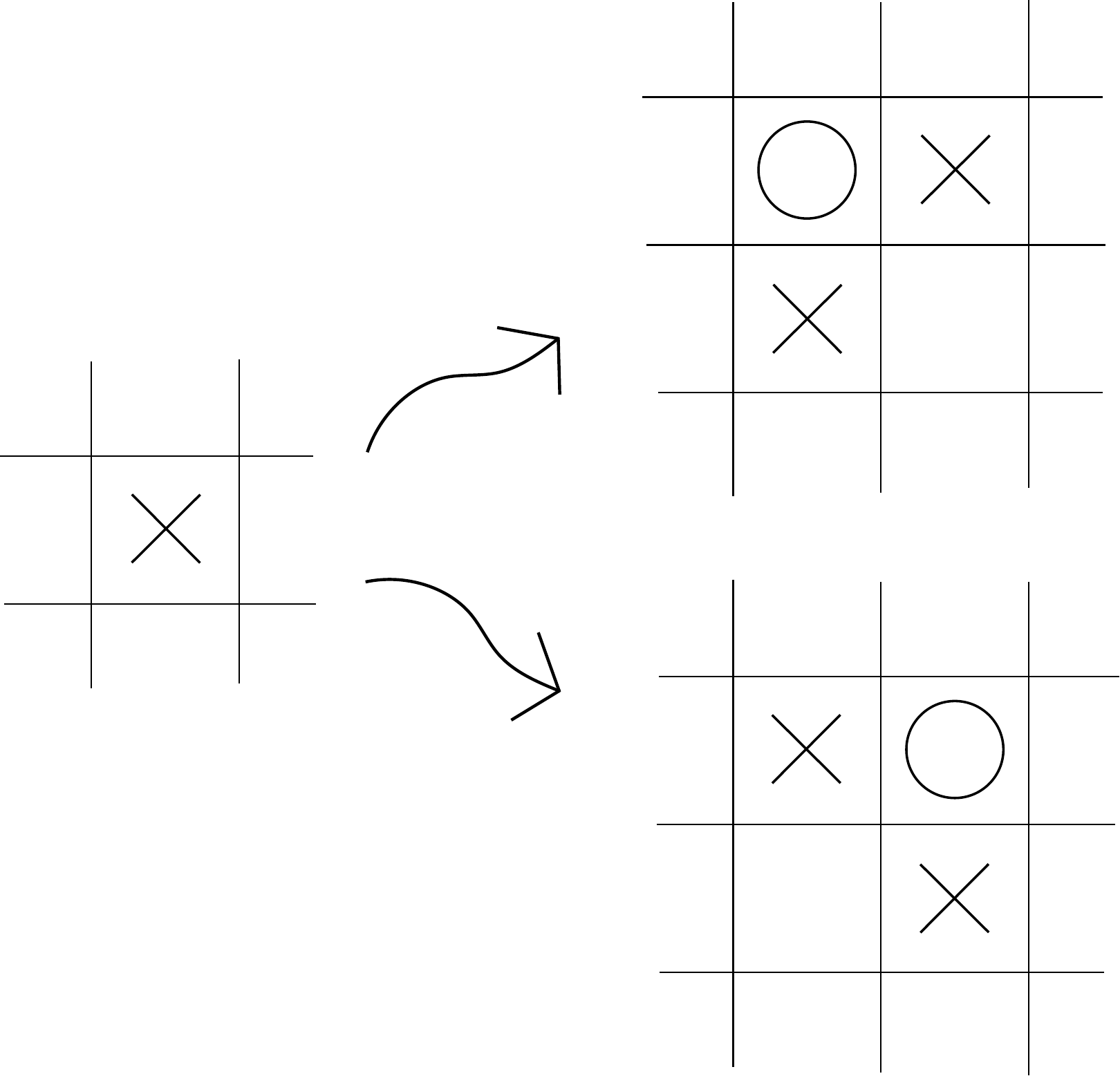}
    \caption{Stabilization; the diagram shows two possibilities out of the
    four}
    \label{fig:grid-de-stabilization}
\end{figure}

\begin{proposition}
  %The homologies $\widetilde{\GH}$
  %and $\GH^-$ are invariant under
  %  commutation.
  The isomorphism types of $\GHt$ and $\GHm$ are invariant under a row or 
  column commutation.
\end{proposition}

\begin{proof}
  Suppose that $\G_1$ and $\G_2$ are related by a column commutation. We can 
  combine $\G_1$ and $\G_2$ into a single diagram, as in the right of 
  Figure~\ref{fig:grid-proof-invar-comm}. Here, if we focus on the blue circle 
  and erase the green circle, we obain (a diagram isotopic to) $\G_1$; 
  similarly, erasing the blue circle, we obtain $\G_2$. Note that we position 
  the blue and green curves so as to avoid triple intersections with 
  horizontal circles. Label the intersection points between and blue and green 
  circles by $a$ and $b$ as shown.
  
%We can draw the commutation move on the same diagram. Indeed,
%we draw the two pairs of O's and X's on the same column (resp. row), and
%consider either the blue curve for the grid diagram $\mathbb{G}_1$
%before the move, or the green curve for the diagram $\mathbb{G}_2$
%after the commutation move. These two curves can be positioned so
%that they intersect each other in two points.
\begin{figure}[!htbp]
	\def\svgwidth{0.95\columnwidth}
	\centering
	%
	%% Creator: Inkscape 1.3.2 (091e20e, 2023-11-25, custom), www.inkscape.org
%% PDF/EPS/PS + LaTeX output extension by Johan Engelen, 2010
%% Accompanies image file 'grid-proof-invar-comm.pdf' (pdf, eps, ps)
%%
%% To include the image in your LaTeX document, write
%%   \input{<filename>.pdf_tex}
%%  instead of
%%   \includegraphics{<filename>.pdf}
%% To scale the image, write
%%   \def\svgwidth{<desired width>}
%%   \input{<filename>.pdf_tex}
%%  instead of
%%   \includegraphics[width=<desired width>]{<filename>.pdf}
%%
%% Images with a different path to the parent latex file can
%% be accessed with the `import' package (which may need to be
%% installed) using
%%   \usepackage{import}
%% in the preamble, and then including the image with
%%   \import{<path to file>}{<filename>.pdf_tex}
%% Alternatively, one can specify
%%   \graphicspath{{<path to file>/}}
%% 
%% For more information, please see info/svg-inkscape on CTAN:
%%   http://tug.ctan.org/tex-archive/info/svg-inkscape
%%
\begingroup%
  \makeatletter%
  \providecommand\color[2][]{%
    \errmessage{(Inkscape) Color is used for the text in Inkscape, but the package 'color.sty' is not loaded}%
    \renewcommand\color[2][]{}%
  }%
  \providecommand\transparent[1]{%
    \errmessage{(Inkscape) Transparency is used (non-zero) for the text in Inkscape, but the package 'transparent.sty' is not loaded}%
    \renewcommand\transparent[1]{}%
  }%
  \providecommand\rotatebox[2]{#2}%
  \newcommand*\fsize{\dimexpr\f@size pt\relax}%
  \newcommand*\lineheight[1]{\fontsize{\fsize}{#1\fsize}\selectfont}%
  \ifx\svgwidth\undefined%
    \setlength{\unitlength}{1159.97113037bp}%
    \ifx\svgscale\undefined%
      \relax%
    \else%
      \setlength{\unitlength}{\unitlength * \real{\svgscale}}%
    \fi%
  \else%
    \setlength{\unitlength}{\svgwidth}%
  \fi%
  \global\let\svgwidth\undefined%
  \global\let\svgscale\undefined%
  \makeatother%
  \begin{picture}(1,0.8601835)%
    \lineheight{1}%
    \setlength\tabcolsep{0pt}%
    \put(0,0){\includegraphics[width=\unitlength,page=1]{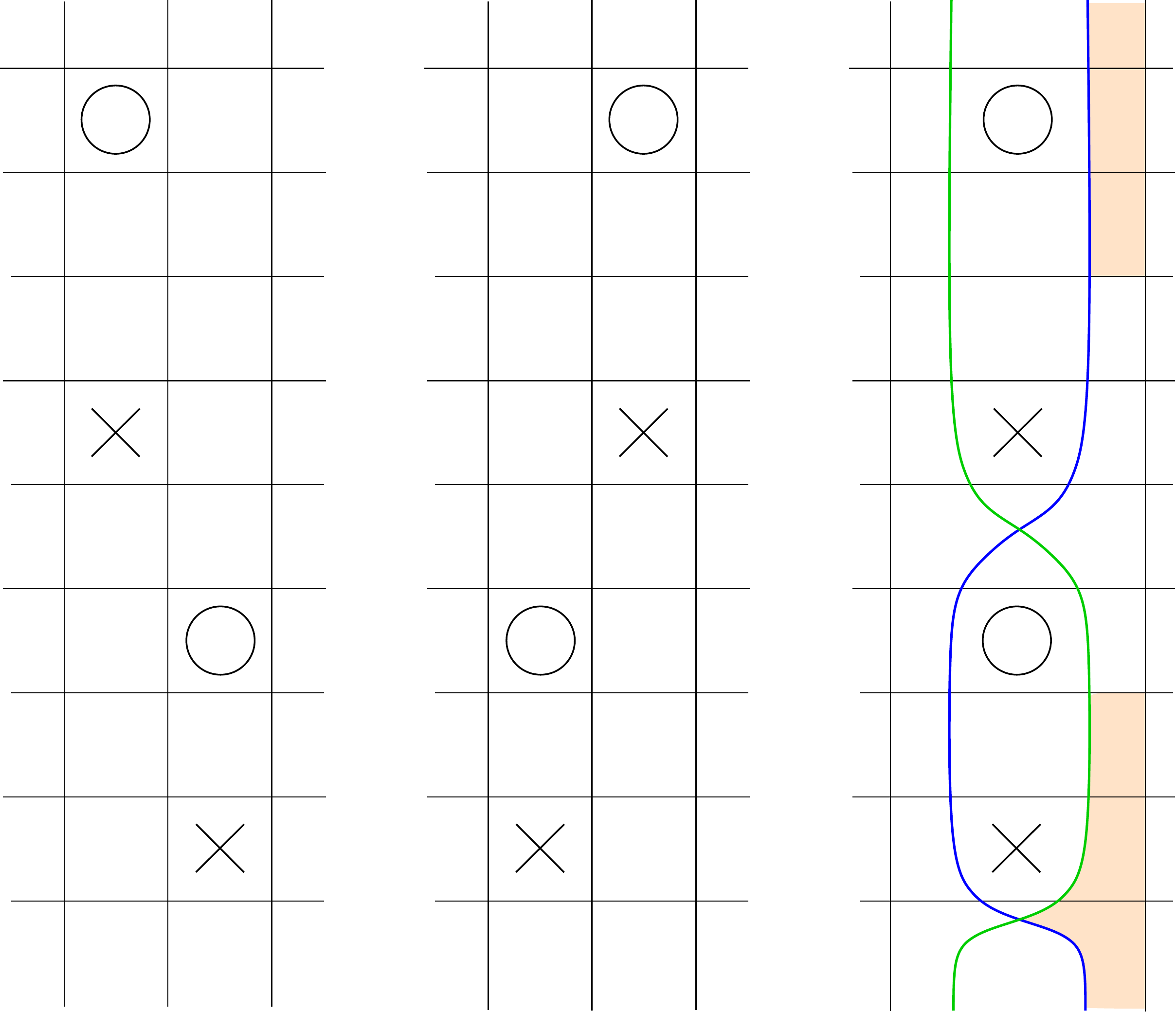}}%
    \put(0.91008515,0.06884313){\color[rgb]{0,0,0}\makebox(0,0)[lt]{\lineheight{1.25}\smash{\begin{tabular}[t]{l}$a$\end{tabular}}}}%
    \put(0.9048561,0.40104231){\color[rgb]{0,0,0}\makebox(0,0)[lt]{\lineheight{1.25}\smash{\begin{tabular}[t]{l}$b$\end{tabular}}}}%
    \put(0,0){\includegraphics[width=\unitlength,page=2]{grid-proof-invar-comm.pdf}}%
  \end{picture}%
\endgroup%

	\caption{Combining two grid diagrams related by a column commutation into one diagram}
	\label{fig:grid-proof-invar-comm}
\end{figure}
%  HAVE TO REDO FIGURE WITH SHADED PENTAGON  %
%We want to show that the chain complexes obtained from the two diagrams 
%$\mathbb{G}_1$ and $\mathbb{G}_2$ of the knot are chain homotopy equivalent.

  The idea now is to use this combined diagram to define a chain homotopy 
  equivalence $\pentmap_1 \colon \GCm (\G_1) \to \GCm (\G_2)$, which counts 
  empty pentagons. Instead of defining empty pentagons rigorously, we will just 
  refer to Figure~\ref{fig:grid-proof-invar-comm}, and say that for $\x \in 
  \genset (\G_1)$ and $\y \in \genset (\G_2)$, an element of $\ePent (\x, \y)$ 
  is an embedded pentagon whose sides lie on the curves in the combined 
  diagram, with vertices on $\x$, $\y$, and the distinguished point $a$, and 
  leave the details to the reader.  (Note that pentagons can exist to the left 
  or the right of the column in question.) Then $\pentmap_1$ is defined by
  \[
    \pentmap_1 (\x) = \sum_{\y \in \genset (\G_2)} \sum_{\substack{p \in \ePent 
        (\x, \y) \\ p \cap \XX = \emptyset}} U_1^{O_1 (p)} \dotsm U_n^{O_n (p)} 
    \y.
  \]

%We define a map $p\colon \GC^-(\mathbb{G}_1)\longrightarrow
%\GC^-(\mathbb{G}_2)$ which counts pentagons (see for example a shaded
%pentagon in Figure \ref{fig:grid-proof-invar-comm}) that has $a$ as a
%vertex.

  First, we need to show that $\pentmap_1$ is a chain map, \ie
  \[
    \bdym_{\GCm (\G_2)} \circ \pentmap_1 + \pentmap_1 \circ \bdym_{\GCm (\G_1)} 
    = 0.
  \]
  (Recall that we are in characteristic $2$.) The proof is similar to the proof 
  that $\bdym \circ \bdym = 0$, but we juxtapose a pentagon with a rectangle 
  (in either order), instead of two rectangles.

  Next, we need to show that $\pentmap_1$ has a homotopy inverse $\pentmap_2 
  \colon \GCm (\G_2) \to \GCm (\G_1)$. The definition of $\pentmap_2$ is 
  \emph{verbatim} identical to that of $\pentmap_1$, but this time we count 
  empty pentagons with a vertex on the distinguished point $b$ instead of $a$.  
  The homotopies $\hexmap_i \colon \GCm (\G_i) \to \GCm (\G_i)$ are defined by 
  counting empty hexagons that have both $a$ and $b$ as vertices.  The claim is 
  then that
  \begin{align*}
    \pentmap_2 \circ \pentmap_1 + \bdym_{\GCm (\G_1)} \circ \hexmap_1 + 
    \hexmap_1 \circ \bdym_{\GCm (\G_1)} & = \Id_{\GCm (\G_1)},\\
    \pentmap_1 \circ \pentmap_2 + \bdym_{\GCm (\G_2)} \circ \hexmap_2 + 
    \hexmap_2 \circ \bdym_{\GCm (\G_2)} & = \Id_{\GCm (\G_2)}.
  \end{align*}
  Indeed, most terms that appear on the left-hand side of these equations 
  cancel out in pairs as in the proof that $\bdym \circ \bdym = 0$, with the 
  exception of domains that are topologically annuli, as in 
  Figure~\ref{fig:grid-proof-invar-comm-id}, which cancel with the identity 
  maps on the right-hand side.
  
  \begin{figure}[!htbp]
  	\def\svgwidth{0.6\columnwidth}
  	\centering
  	%
	%% Creator: Inkscape 1.2.2 (732a01da63, 2022-12-09), www.inkscape.org
%% PDF/EPS/PS + LaTeX output extension by Johan Engelen, 2010
%% Accompanies image file 'grid-proof-invar-comm-id.pdf' (pdf, eps, ps)
%%
%% To include the image in your LaTeX document, write
%%   \input{<filename>.pdf_tex}
%%  instead of
%%   \includegraphics{<filename>.pdf}
%% To scale the image, write
%%   \def\svgwidth{<desired width>}
%%   \input{<filename>.pdf_tex}
%%  instead of
%%   \includegraphics[width=<desired width>]{<filename>.pdf}
%%
%% Images with a different path to the parent latex file can
%% be accessed with the `import' package (which may need to be
%% installed) using
%%   \usepackage{import}
%% in the preamble, and then including the image with
%%   \import{<path to file>}{<filename>.pdf_tex}
%% Alternatively, one can specify
%%   \graphicspath{{<path to file>/}}
%% 
%% For more information, please see info/svg-inkscape on CTAN:
%%   http://tug.ctan.org/tex-archive/info/svg-inkscape
%%
\begingroup%
  \makeatletter%
  \providecommand\color[2][]{%
    \errmessage{(Inkscape) Color is used for the text in Inkscape, but the package 'color.sty' is not loaded}%
    \renewcommand\color[2][]{}%
  }%
  \providecommand\transparent[1]{%
    \errmessage{(Inkscape) Transparency is used (non-zero) for the text in Inkscape, but the package 'transparent.sty' is not loaded}%
    \renewcommand\transparent[1]{}%
  }%
  \providecommand\rotatebox[2]{#2}%
  \newcommand*\fsize{\dimexpr\f@size pt\relax}%
  \newcommand*\lineheight[1]{\fontsize{\fsize}{#1\fsize}\selectfont}%
  \ifx\svgwidth\undefined%
    \setlength{\unitlength}{716.06562236bp}%
    \ifx\svgscale\undefined%
      \relax%
    \else%
      \setlength{\unitlength}{\unitlength * \real{\svgscale}}%
    \fi%
  \else%
    \setlength{\unitlength}{\svgwidth}%
  \fi%
  \global\let\svgwidth\undefined%
  \global\let\svgscale\undefined%
  \makeatother%
  \begin{picture}(1,1.39343099)%
    \lineheight{1}%
    \setlength\tabcolsep{0pt}%
    \put(0,0){\includegraphics[width=\unitlength,page=1]{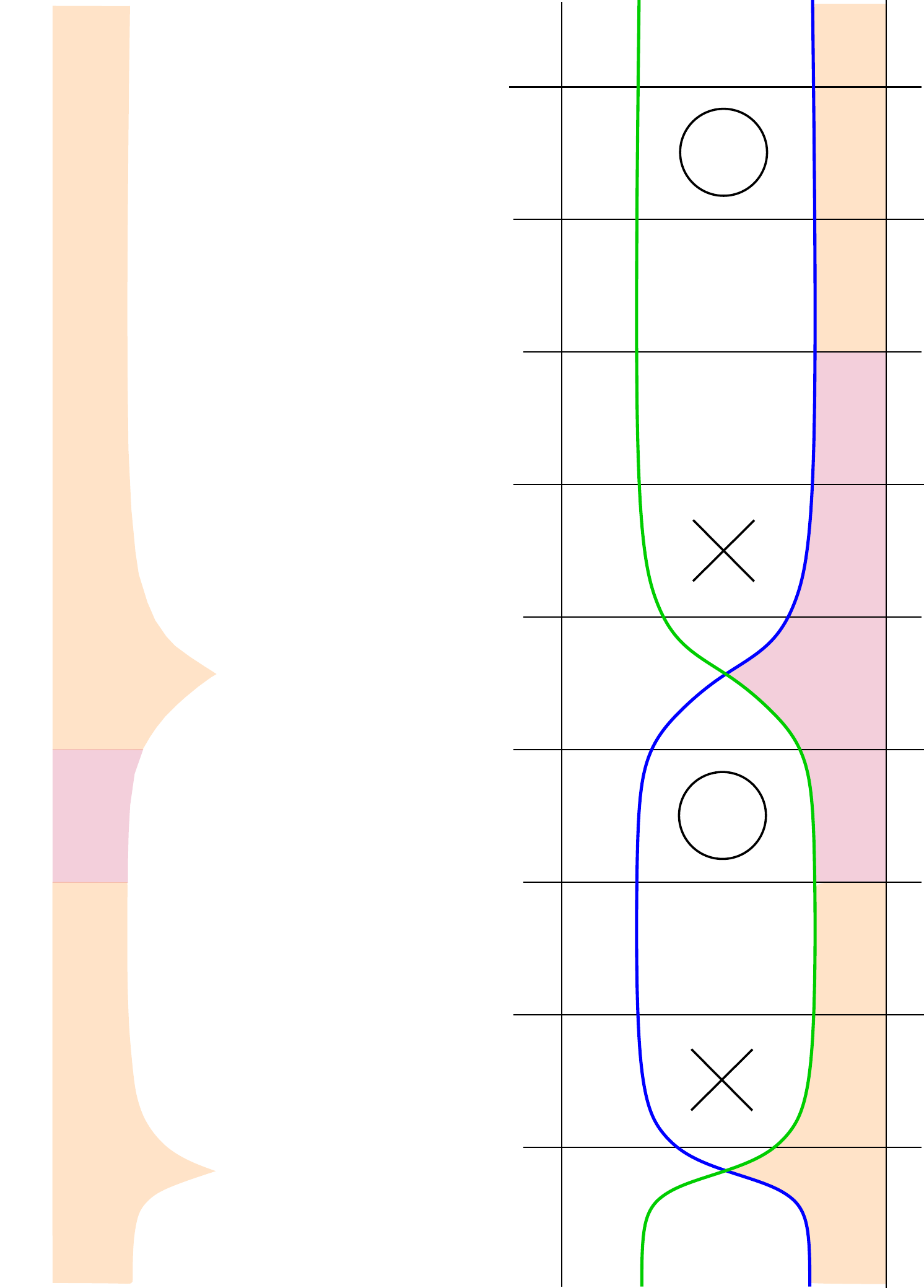}}%
    \put(0.85537575,0.11152058){\makebox(0,0)[lt]{\lineheight{1.25}\smash{\begin{tabular}[t]{l}$a$\end{tabular}}}}%
    \put(0.84690509,0.64965764){\makebox(0,0)[lt]{\lineheight{1.25}\smash{\begin{tabular}[t]{l}$b$\end{tabular}}}}%
    \put(0,0){\includegraphics[width=\unitlength,page=2]{grid-proof-invar-comm-id.pdf}}%
    \put(0.30433313,0.11152058){\makebox(0,0)[lt]{\lineheight{1.25}\smash{\begin{tabular}[t]{l}$a$\end{tabular}}}}%
    \put(0.29586247,0.64965764){\makebox(0,0)[lt]{\lineheight{1.25}\smash{\begin{tabular}[t]{l}$b$\end{tabular}}}}%
    \put(0,0){\includegraphics[width=\unitlength,page=3]{grid-proof-invar-comm-id.pdf}}%
  \end{picture}%
\endgroup%

    \caption{Two annuli, one representing a term in $\bdym_{\GCm (\G_1)} \circ 
      \hexmap_1 (\x)$ (left) and one representing a term in $\pentmap_2 \circ 
      \pentmap_1 (\x')$ (right), that cancel with the identity. Here $\x$ and 
      $\x'$ (both in red) are two distinct generators in $\genset (\G_1)$.  
      Note that there are other possibilities of such annular domains.  
      Depending on $\x$, there is always exactly one term that results in such 
      an annular domain}
  	\label{fig:grid-proof-invar-comm-id}
  \end{figure}

%To show that $p$ is a chain map, we need to show that $p\circ
%\partial^-=\partial^-\circ p$. The proof is similar to the proof that
%$\widetilde{\partial}$ is a boundary map, where instead of having a
%decomposition into two rectangles (see Figure
%\ref{fig:proof-tildemap-zeroB}), we consider decompositions into one
%rectangle and one pentagon, like in the figure
%\ref{fig:grid-proof-pentagon}. These decompositions also come in pairs,
%showing
%\[
%p\circ \partial ^- + \partial ^-\circ p=0 \pmod{2},
%\]
%verifying the claimed equation (over the field of two elements).
%\begin{figure}[h!]
%    \centering
%    \includegraphics[scale=0.3]{images/grid-proof-pentagon.pdf}
%    \caption{Example of a decomposition into a rectangle (green) and a pentagon (orange)}
%    \label{fig:grid-proof-pentagon}
%\end{figure}

%If we denote the reverse map by $p':\GC^-(\mathbb{G}_2)\longrightarrow
%\GC^-(\mathbb{G}_1)$ (\ie here we count pentagon with the point $b$ as
%a vertex), then we have the following equality:
%\[
%p'\circ p=\id_{\GC^-(\mathbb{G}_1)}+\partial^-\circ h+h\circ\partial^-,
%\]
%where $h$ counts hexagons with $a$ and $b$ as part of the vertices.
%This shows that $p\circ p'$ (and with a similar reasoning $p'\circ p$)
%is chain homotopic to the identity, so the maps
%$p_*$ and $p'_*$ induced by $p$ and $p'$ on homology are inverses
%of each other. This concludes the argument.
The paragraphs above show that $(\pentmap_1)_* \colon \GHm (\G_1) \to \GHm 
(\G_2)$ and $(\pentmap_2)_* \colon$ $\GHm (\G_2) \to \GHm (\G_1)$ are inverses of 
each other, and are thus isomorphisms. The proof for row commutations is 
similar, and the proof for $\GHt$ is obtained by modifying the proof above to 
block all $O$ markers in all maps.
\end{proof}

% invariance under stabilization: too tedious

The invariance of $\GHm$ under (de)stabilization is quite a bit more tedious, 
and so we refer the reader to \cite[Section~5.2]{OSS15:GH-book} for full details. We 
simply give a brief outline here:
\begin{itemize}
  \item First, we claim that the actions by the formal variables $U_i$ and 
    $U_j$ on $\GCm (\G)$ are chain homotopic for all $i, j \in \{ 1, \dotsc, n 
    \}$.  The proof proceeds by defining a homotopy $\phi \colon \GCm (\G) \to 
    \GCm (\G)$ that counts rectangles containing a single $X$ marker $X_i$ in 
    it.  One can show that
    \[
      \bdym \circ \phi + \phi \circ \bdym = U_i + U_{i+1},
    \]
    where the corresponding markers $O_i$ and $O_{i+1}$ are in the same row and 
    same column as $X_i$ respectively; this proof is again similar to the proof 
    that $\bdym \circ \bdym = 0$.
  \item The upshot is that $\GHm (\G)$ can be thought of as an $\bF [U]$-module, 
    where $U$ acts by multiplication by any of the $U_i$'s.
  \item One then proceeds to prove that the isomorphism type of $\GHm$ is 
    invariant under (de)stabilization, by considering certain mapping cones.  
    Combined with invariance under commutation, we see that the isomorphism 
    type of $\GHm (\G)$ is a knot invariant, which we denote by $\GHm (K)$.
  \item One may define the \emph{hat grid chain complex} by setting one of the 
    $U_i$'s to be zero, \ie
    \[
      \GCh (\G) = \GCm (\G) \otimes_{\bF [U_1, \dotsc, U_n]} \left( \faktor{\bF 
          [U_1, \dotsc, U_n]}{U_1} \right).
    \]
    This has the effect of setting the action of $U$ to be zero, and so we get 
    an $\bF$-module $\GCh (\G)$. It is clear from the above that the isomorphism 
    type of the $\bF$-module $\GHh (\G)$ is also an invariant, which we denote 
    by $\GHh (K)$.
  \item Finally, we may place $\GHt (\G)$ in this context also. Note that $\GCt 
    (\G)$ is obtained from $\GCh (\G)$ by also setting all the other $U_i$'s to 
    zero. In other words, we have set $n - 1$ ``too many'' $U_i$'s to be zero, 
    and some homological algebra shows that
    \[
      \GHt (\G) \cong \GHh (\G) \otimes_{\bF} \bF^{\otimes 2^{n-1}},
    \]
    and so the isomorphism type of $\GHt (\G)$ is ``almost'' a knot invariant 
    and can be denoted $\GHt (K, n)$.
\end{itemize}

%In $\GC^-$ one can also prove that $U_i$ and $U_j$ act
%homotopically.
%Indeed, the proof of this statement is again very similar to the
%verification of $\partial ^2=0$.
%We can thus think of $\GH^-$ as an $\bF[U]$-module, and
%this is an invariant under stabilization, hence providing a knot invariant.
%(The invariance under stabilization is somewhat more involved, and will
%not be presented here. Indeed, the $\widetilde{\GH}$ theory is not invariant
%under stabilization, while $\GH ^-$ is.)
%
%If we set one of the $U_i$'s to be zero, \ie
%\[
%\widehat{\GC}(\G):= \GC^-(\G)\otimes_{\bF[U_1,\dots,U_n]} \left(\faktor{\bF[U_1,\dots,U_n]}{U_1}\right) ,
%\]
%we obtain a new invariant $\widehat{\GH}(\G)$, which is an $\bF$-module (because the $U_i$'s act homotopically).
%By setting more $U_i$'s to be zero, we actually get
%\[
%\widetilde{\GH}(\G)\cong \widehat{\GH}(\G)\otimes\bF^{2n-1},
%\]
%so $\GHtilde(\G)$ is almost an invariant.

%%% ATTENTION: peut-être modifier l'exponent

%\begin{remark}
%    Invariance of $\GH^-$ in fact holds on of the chain level up to
%    homotopy equivalence.
%\end{remark}

%  détails du zéro homology (correspond à l'une des homologies, osef laquelle)

We may summarize the invariance of grid homology here.

\begin{theorem}[Manolescu, Ozsv\'ath, and Sarkar {\cite[Theorem~1.1]{MOS09}} 
  and Manolescu, Ozsv\'ath, Szab\'o, and Thurston {\cite[Theorem~1.2]{MOST07}}, 
  \cf {\cite[Theorem~5.3.1]{OSS15:GH-book}}]
  Let $\G$ be a grid diagram of $K \subset S^3$.  The graded isomorphism types 
  of $\GHm (\G)$ and $\GHh (\G)$ are invariants of the knot $K$.  Thus, we may write 
  $\GHm (K)$ and $\GHh (K)$ for these homologies.
\end{theorem}

\begin{remark}
  Instead of the homology, one may in fact state the invariance of the grid 
  complex (and, later, of the Heegaard Floer complex) on the chain level: The 
  \emph{graded chain homotopy type} of the chain complex is an invariant. This 
  is important, as many applications rely on invariants extracted from the 
  chain complex.
\end{remark}

In this lecture, we have chosen to exhibit in detail certain proofs while 
hiding others entirely. The point of showing the proof that $\bdym_{\GCm (\G)} 
\circ \bdym_{\GCm (\G)} = 0$ in detail is that it is structurally the same as 
the proof that $\bdym_{\CFm (\heeg)} \circ \bdym_{\CFm (\heeg)} = 0$ later, 
when we define the Heegaard Floer complex $\CFm (\heeg)$. There, the boundary 
homomorphism is defined by counting certain (pseudo)holomorphic bigons, and so 
the analysis becomes much more intricate; however, one will still consider a 
generator $\z$ that appears in $\bdym \circ \bdym (\x)$, juxtapose two 
holomorphic bigons (analogous to rectangles here), and prove that there is an 
even number of ways that the resulting holomorphic curve can be decomposed into 
two holomorphic bigons that are counted.

Similarly, the proof that $\GHm$ is invariant under commutation is structurally 
the same as the proof that $\wHFm$, the homology of $\CFm$, is invariant under 
handleslides. There, one also combines multiple Heegaard diagrams (analogous to 
grid digrams) into a single diagram, define chain maps by counting holomorphic 
triangles (analogous to pentagons here), and define chain homotopies by 
counting holomorphic quadrilaterials (analogous to hexagons here).

Limited by time, we will not be able to delve into the details of these proofs 
in later lectures. However, it is my hope that the analogies above will provide 
some insight into them.

%\section{Second lecture}
%Yesterday: $\widetilde{GH}(K,n)$, $GH^-()$, $\widehat{GH}$
\subsection{Maslov and Alexander gradings}

Next, we move on to the homological grading of grid homology, known as the 
\emph{Maslov grading}. This will allow us to discuss the graded isomorphism 
types of $\GHm (K)$. While the definitions of gradings on Heegaard Floer 
homology are somewhat different in nature, the hope is that this will help us 
gain some insight into those gradings also. We continue to follow \cite{OSS15:GH-book}.

%In fact, grid homology groups come with a bigrading, and these gradings
%play a central role in many applications. To set up the gradings, we need
%a few definitions.

\begin{definition}
  Let $A$ and $B$ be two finite sets in $\mathbb{R}^2$. We define the following quantities:
  \begin{gather*}
    \mathcal{I}(A,B) =\#  \{(a,b)\in A\times B\,\vert\, a=(a_1,a_2), \, 
    b=(b_1,b_2), \, a_1<b_1, \, a_2<b_2 \}, \\
    \mathcal{J}(A,B)=\frac{1}{2} (\mathcal{I}(A,B)+\mathcal{I}(B,A)).
  \end{gather*}
  In other words, $\mathcal{I} (A, B)$ counts the number of pairs of points 
  $(a, b) \in A \times B$ such that $a$ is strictly to the southwest of 
  $b$---such a pair is known as a \emph{northeast--southwest pair}---and 
  $\mathcal{J}$ symmetrizes $\mathcal{I}$.

  Given $\x \in \genset (\G)$, the \emph{Maslov grading} $M (\x) = M_{\OO} 
  (\x)$ of $\x$ is defined by
  \[
    M (\x) = M_{\OO} (\x) = \JJ (\x, \x) - 2 \JJ (\x, \OO) + \JJ (\OO, \OO) + 
    1.
  \]
%
%  The \emph{Maslov grading} on $\genset (\G)$ is the map $M$ assigning to
%  each element  $\bm{x}\in S(\G)$ the integer
%  $M(\bm{x})=M_\OO(\bm{x})=\mathcal{J}(\bm{x},\bm{x})
%  -2\mathcal{J}(\bm{x},\OO)+\mathcal{J}(\OO,\OO)+1$.
\end{definition}

\begin{proposition}
  If there is a rectangle $r$ from $\x$ to $\y$, then
  \[
    M(\x)-M(\y)=1-2\#(r\cap\OO).
  \]
\end{proposition}

\begin{proof}
  If there are no $O$'s inside the rectangle $r$, then
  $\mathcal{J}(\x,\OO)=\mathcal{J}(\y,\OO)$ and
  $\mathcal{J}(\x,\x)-\mathcal{J}(\y,\y)=1$,
  because $r$ ``destroys'' a northeast--southwest pair in $\x \times \x$ to 
  create a northwest--southeast pair in $\y$, and $\x$ and $\y$
  coincide everywhere else.

  If $r\cap\OO\neq \emptyset$, then for each $O$ inside the rectangle
  $r$, both $\mathcal{I}(\x,\OO)$ and
  $\mathcal{I}(\OO,\x)$ are increased by $1$; thus,
  $\mathcal{J}(\x,\OO)-\mathcal{J}(\y,\OO)$ counts the number of $O$'s
  inside $r$.
\end{proof}

This implies that $M$ induces a homological $\mathbb{Z}$-grading on
$\widetilde{\GC}(\G)$, since $r \cap \OO = \emptyset$ in all relevant 
rectangles $r$ in the boundary homomorphism $\bdyt$.

For $\GC^-(\G)$, note that $\# (r \cap \OO)$ may be a positive integer $m$.  
This would mean that $U^m \y$ appears as a term in $\bdym (\x)$, where by $U^m$ 
we mean a product of $m$ distinct $U_i$'s. Since $M (\x) - M (\y) = 1 - 2 m$, 
if we set $M (U_i) = - 2$ for all $U_i$'s, then
\[
  M (U^m \y) = M (U^m) + M (\y) = - 2 m + (M (\x) - 1 + 2 m) = M (\x) - 1,
\]
and we have a homological $\Z$-grading on $\GCm (\G)$.

%Recall that the
%boundary map allows rectangles from $\bm{x}$ to $\bm{y}$ to contain
%O's. In order to achieve that the boundary map drops grading by 1, we
%introduce a grading shifts of the $U_i$'s by $-2$.

\begin{remark}
  The reason that we have a homological $\Z$-grading on $\GHm (K)$ is because, 
  as we shall see, it is a version of knot Floer homology in the three-sphere, 
  $\HFKm (S^3, K)$.  For a general $3$-manifold $Y$, neither $\wHFm (Y)$ nor 
  $\HFKm (Y, K)$ is necessarily $\Z$-graded; in fact, one may get different $\Z 
  / m$ gradings for different $\Spinc$-structures on $Y$. These gradings are 
  also not necessarily \emph{absolute}, meaning that they may 
  be well-defined 
  only up to a constant shift. We will briefly discuss the 
  $\Spinc$-decomposition and gradings in the last lecture.
\end{remark}

In fact, $\GCt (\G)$ and $\GCm (\G)$ are equipped with an internal grading 
$A$ (meaning that the boundary homomorphism preserves $A$), known as the 
\emph{Alexander grading}.

\begin{definition}
  Given $\x \in \genset (\G)$, the \emph{Alexander grading} $A (\x)$ of $\x$ is 
  defined by
  \[
    A (\x) = \frac{1}{2} (M_{\OO} (\x) - M_{\XX} (\x)) - \frac{n - 1}{2},
  \]
  where $M_{\XX}$ is defined in the same way as $M_{\OO}$ but with $\OO$ 
  replaced by $\XX$ in its definition.
\end{definition}
  
The existence of $A$ is a feature of knot Floer invariants (\eg $\CFKm$) and 
can be understood in terms of a ``relative'' $\Spinc$-decomposition. We will 
not discuss it further until the final lecture, but will point out that $A$ 
plays an important role in many applications; for example, see 
Remark~\ref{rmk:tau}.

%\begin{remark}
%  In fact, $\GCt (\G)$ and $\GCm (\G)$ are equipped with an internal grading 
%  $A$ (meaning that the boundary homomorphism preserves $A$), known as the 
%  \emph{Alexander grading}, defined by the formula
%  \[
%    A (\x) = \frac{1}{2} (M_{\OO} (\x) - M_{\XX} (\x)) - \frac{n - 1}{2},
%  \]
%  where $M_{\XX}$ is defined in the same way as $M_{\OO}$ but with $\OO$ 
%  replaced by $\XX$ in its definition. The existence of $A$ is a feature of 
%  knot invariants (\eg $\CFKm$) and can be understood in terms of a 
%  ``relative'' $\Spinc$-decomposition. We will not pursue it further in this 
%  lecture series, but it must be pointed out that $A$ plays an important role 
%  in many applications; for example, see Remark~\ref{rmk:tau}.
%\end{remark}

\subsection{Graded isomorphism types}

%Let us now look at the algebraic structure of the various grid
%homologies.
From now until the final lecture, we will focus only on the Maslov grading and 
ignore the Alexander grading.

Let us now look at the possible graded isomorphism types that the various 
flavors of grid homology can have, starting with $\GHt (K, n)$ and $\GHh (K)$. 
Since they are finitely generated graded $\bF$-modules (that is to say, 
$\bF$--vector spaces), they must have graded isomorphism types
\[
  \bigoplus_{m\in\mathbb{Z}}\bF_{(m)}^{\oplus a_m} \quad \text{with} \quad \sum_m 
a_m <\infty,
\]
where $\bF_{(m)}$ denotes a summand in Maslov grading $m$.

%Because (graded) $\bF$-modules are $\bF$-vector spaces, the
%graded isomorphism type of $\widetilde{\GH}$ and $\widehat{\GH}$ can
%be
%
%\[
%\bigoplus_{m\in\mathbb{Z}}\bF_{(m)}^{\oplus a_m} \text{ with } \sum_m a_m <\infty,
%\]
%where $m$ denotes the grading.

As for $\GH^-$, it is an $\bF[U]$-module. Since $\bF[U]$ is a principal
ideal domain, the Fundamental Theorem for Finitely Generated Modules
over a Principal Ideal Domain implies that
$\GH^-(K)$ is a direct sum of cyclic modules, \ie it must have ungraded 
isomorphism type
\[
  \GH^-(K)\cong \bF[U]^{\oplus a}\oplus 
  \left(\faktor{\bF[U]}{(f_1(U))}\right)\oplus \dotsb \oplus 
  \left(\faktor{\bF[U]}{(f_k(U))}\right),
\]
where the $f_i$'s are polynomials.
Moreover, since $U$ strictly decreases grading,
\begin{comment}
we cannot have for example
\[
U^2x=x+\partial\circ H+H\circ\partial \iff (U^2+1)[x]=0.
\]
So
\end{comment}
the $f_i$'s must actually be monomials! Hence, we have
\[
  \GH^-(K)\cong \bF[U]^{\oplus a}\oplus 
  \left(\faktor{\bF[U]}{U^{b_1}}\right)\oplus \dotsb \oplus 
  \left(\faktor{\bF[U]}{U^{b_k}}\right) .
\]
%when $a+\sum b_i <\infty$.

\begin{remark}
  This last assertion is no longer true if the grading is a $\Z / m$-grading 
  instead of a $\Z$-grading, because $U^{m/2} \x$ can now have the same grading 
  as $\x$. This may occur for the Heegaard Floer homology of a general 
  $3$-manifold.
\end{remark}

One can, of course, write a graded version of this statement as well, which is 
left to the reader.

It turns out that, because $\GHm (K)$ is a version of $\HFKm (S^3, K)$, \ie the 
ambient $3$-manifold is $S^3$, there is only one $\bF [U]$-summand, meaning that 
$a = 1$.
%As it turns out, for a knot $K\subset S^3$, there is exactly one
%$\bF[U]$ summand in $\GH^-(K)$,
%hence
%\[
%\GH^-(K)\cong \bF[U]\oplus \left(\faktor{\bF[U]}{(f_1(U))}\right)\oplus \dots \oplus \left(\faktor{\bF[U]}{(f_k(U))}\right),
%\]

\begin{remark}
  \label{rmk:tau}
  The invariant $\tau (K)$ is defined as $- A$, where $A$ is the (top) 
  Alexander grading of the generator of the $\bF [U]$-summand of $\HFKm (K)$.  $\tau$ provides a very useful concordance homomorphism.
\end{remark}

One feature that is often overlooked is that the graded isomorphism type of 
$\GHh (K)$ is completely determined by that of $\GHm (K)$. Indeed, note that
\[
  \widehat{\GC}(\G)\cong \GC^-(\bG)\otimes_{\bF[U]} 
  \left(\faktor{\bF[U]}{U}\right).
\]
Thus, the Universal Coefficient Theorem implies that
\[
\widehat{\GH}(K)\cong \GH^-(K)\otimes_{\bF[U]} \left(\faktor{\bF[U]}{U}\right)\oplus \Tor_1^{\bF[U]}\left(\GH^-(K),\faktor{\bF[U]}{U}\right).
\]

%\begin{remark}
%The hat version of the grid chain complex amounts to setting  $U$ to zero, such that as a $\bF[U]$-module, we have
%\[
%\widehat{\GC}(\bG)\cong \GC^-(\bG)\otimes_{\bF[U]} \left(\faktor{\bF[U]}{U}\right).
%\]
%By the universal coefficient theorem, we obtain
%\[
%\widehat{\GH}(K)\cong \GH^-(K)\otimes_{\bF[U]} \left(\faktor{\bF[U]}{U}\right)\oplus \Tor_1^{\bF[U]}\left(\GH^-(K),\faktor{\bF[U]}{U}\right).
%\]
%\end{remark}

\begin{example}
    With some work, one can work out the \emph{minus} grid homology of the 
    left-handed trefoil to be
    \[
    \GH^-(\overline{T}_{2,3})\cong (\bF[U])_{(2)} \oplus \left( 
      \faktor{\bF[U]}{U} \right)_{(1)}, % \cong\bF[U]\oplus\bF \right)_{(0)},
    \]
    and so we can also read off the \emph{hat} grid homology to be
    \begin{align*}
    \widehat{\GH}(\overline{T}_{2,3}) &\cong 
    \left(\faktor{\bF[U]}{U}\right)_{(2)}
      \oplus\left(\faktor{\bF[U]}{U}\right)_{(1)}\oplus\left(\faktor{\bF[U]}{U}\right)_{(0)}\\
      & \cong \bF_{(2)} \oplus \bF_{(1)} \oplus \bF_{(0)}.
    \end{align*}
\end{example}

\begin{exercise}
Compute $\Tor_1$ here. Bonus: Compute it with gradings.
\end{exercise}

%because grid homology is some form of
%Heegaard Floer homology for knots in $S^3$.
\begin{comment}
\begin{definition}
    $\tau(K)=-\text{ grading of the unique } \bF[U]\text{ summand}$.
\end{definition}

\begin{exampl4}
    $\tau(T_{2,3})=\pm 1$.
\end{exampl4}
\end{comment}

%    bim    %

\subsection{Other flavors}

Finally, we mention some more flavors of grid homology. First, we have the 
\emph{infinity} and \emph{plus} flavors.

\begin{definition}
  The \emph{infinity} and \emph{plus grid chain complexes} are defined 
  respectively by
  \begin{align*}
    \GC^\infty(\G)&=\GC^-(\G)\otimes_{\bF[U]}\bF[U,U^{-1}],\\
    \GC^+(\G)&=\GC^\infty (\G) 
    \otimes_{\bF[U]}\left(\faktor{\bF[U,U^{-1}]}{U}\right),
  \end{align*}
  both thought of as (bi)graded chain complexes of $\bF[U]$-modules.
\end{definition}

One can recover $\GCm (\G)$ from $\GCi (\G)$ as a subcomplex, and $\GCp (\G)$ 
from $\GCi (\G)$ as a quotient complex, and so these three flavors contain 
roughly the same amount of information.

\begin{exercise}
  Work out some obvious mapping cones among the \emph{minus}, \emph{plus}, 
  \emph{infinity}, and \emph{hat} flavors of grid chain complexes.
\end{exercise}

All of these flavors are considered in the literature, as it is often easier to 
formulate a certain statement in one flavor versus another.  They also all 
appear in general Heegaard Floer homology, as $\CFm (\heeg)$, $\CFp (\heeg)$, 
$\CFi (\heeg)$, and $\CFh (\heeg)$. In each of these cases, the graded 
isomorphism type of the homology (or in fact, the graded chain homotopy type of 
the chain complex) is an invariant.

\begin{remark}
  Again by the Universal Coefficient Theorem, one can compute
  \[
    \GH^\infty(K)\cong \bF[U,U^{-1}],
  \]
  and so the \emph{infinity} homology is not very interesting.  However, the 
  \emph{infinity} chain complex $\GCi (K)$ enjoys a rich structure, and many 
  invariants, such as $\epsilon (K)$ \cite{Hom14} and $\Upsilon (K)$ 
  \cite{OSS17:Upsilon}, can be extracted from it.
\end{remark}

There are more flavors that are worth noting:
\begin{itemize}
  \item If we allow rectangles to contain $X$'s but do not record them, the 
    Alexander grading is no longer an internal grading; instead, it can be 
    lowered by the boundary homomorphism, meaning that it becomes only a 
    \emph{filtration}.  (The Maslov grading remains a homological grading.) The 
    filtered, graded homotopy type of this filtered chain complex, denoted 
    $\mathcal{GC}^- (\G)$, is a knot invariant. To recover $\GCm (K)$, one 
    takes the \emph{associated graded object} of $\mathcal{GC}^- (K)$, which 
    exactly blocks all rectangles that contain $X$'s.
    The object in general Heegaard Floer theory analogous to $\mathcal{GC}^-(K)$ is $\CFKm (Y, K)$, which is a filtered, graded chain complex, and     $\HFKm (Y, K)$, which is analogous to $\GHm (K)$, is the homology of the  associated graded object of $\CFKm (Y, K)$.
  \item An approach alternative to filtered complexes is to introduce formal 
    variables $V_i$ for $X_i$, and define
    \[
      \bdym_{\OO\XX} (\x) = \sum_{\y \in \genset (\G)} \sum_{r \in \eRect (\x, 
        \y)} U_1^{O_1 (r)} \dotsm U_n^{O_n (r)} V_1^{X_1 (r)} \dotsm V_n^{X_n 
        (r)} \y,
    \]
    This is the approach taken in \cite{Zem19:HFK-func, Zem19:HFK-gr} and some 
    other recent work.  However, one must grapple with the fact that 
    $\bdym_{\OO\XX} \circ \bdym_{\OO\XX} \equiv c \neq 0$, where $c \in \wring$ 
    is the \emph{curvature} of a \emph{curved chain complex}.
    %$\mathcal{R}=\bF[U_1,\dots,U_n,V_1,\dots,V_n]$ (see \cite{Ze16},
    %\cite{Zemke_2018} \cite{zemke2019link});
  \item We can also set $U_1=\dots=U_n=V_1=\dots=V_n$ and consider a chain 
    complex with
    \[
      \partial^{'} (\x)=\sum_{\y\in \genset(\G)} \sum_{r\in \eRect(\x,\y)} 
      U^{O(r)+X(r)} \y.
    \]
    Since $O$'s and $X$'s are treated equally in the boundary homomorphism, we 
    lose the information of the orientation of the knot.
    %(by the definition of the grid diagram).
    This is called
    \textit{unoriented grid homology} \cite{OSS17:HF'}.
\end{itemize}
%These all have subtle consequences on the gradings.
%\end{remark}

\section{``Naive'' Heegaard Floer homology}
\subsection{Definition}

Having caught a glimpse of some of the features of Heegaard Floer homology as 
exhibited in grid homology, the rest of this lecture will be devoted to 
discussing a ``naive'' definition of Heegaard Floer homology, where the 
boundary homomorphism counts domains on a Heegaard diagram, just as the grid 
boundary homomorphism counts rectangles on a grid diagram.  For someone 
unacquainted with Heegaard Floer homology attending a research talk on the 
subject, this may seem like what is going on when the speaker shades certain 
domains on a Heegaard diagram. However, this is not the actual definition of 
Heegaard Floer homology, as it may not even produce a legitimate chain complex. 
%in fact, it does not give us an invariant.
And yet, I think there is some merit in investigating this ``naive'' 
definition.

This exercise serves several purposes.  First, not yet burdened with the 
analysis involved in Morse or Floer theory, we can transfer some of the 
intuition we gained from grid homology to help us make sense of some 
definitions.  Second, by looking at some of the simplest examples, for which 
the ``naive'' definition seems to make sense, we can get a taste of some of the 
issues that must be addressed in the actual theory. Third, by looking at an 
example where $\bdy^2 \not\equiv 0$, we will get to see why the ``naive'' 
definition \emph{cannot} make sense.

%We will first give a simplified version of the theory.
We assume familiarity with Heegaard diagrams of $3$-manifolds. In Heegaard 
Floer homology, we consider \emph{pointed $3$-manifolds} $(Y, z)$, where $z \in 
Y$. Correspondingly, we consider \emph{pointed Heegaard diagrams} $\heeg = 
(\heegsurf, \alphas, \betas, z)$.  As usual, if $\heegsurf$ has genus $g$, then 
$\alphas$ is a set of $g$ mutually disjoint circles $\{ \alpha_1, \dotsc, 
\alpha_g \}$ on $\heegsurf$, and $\betas$ is a set of $g$ mutually disjoint 
circles $\{ \beta_1, \dotsc, \beta_g \}$. We often abuse notation and use 
$\alphas$ to denote $\alpha_1 \cup \dotsb \cup \alpha_g$ and $\betas$ to denote 
$\beta_1 \cup \dotsb \cup \beta_g$. We assume that $z$ lies on $\heegsurf$ away 
from all $\alpha$- and $\beta$-curves.
%Consider the pointed diagram $\mathcal{H}=\left(\Sigma,\alpha,\beta,z\right)$ 
%for the pointed three-manifold $(Y,z)$.

\begin{example}\label{examplenaiveHF}
  Figure~\ref{fig:two-HD-of-3sphere} shows two pointed Heegaard diagrams of 
  $(S^3, z)$, while Figure~\ref{fig:two-HD-of-s1xs2} shows two diagrams of 
  $(S^1 \times S^2, z)$. As is customary, $\alpha$-curves are colored red, 
  while $\beta$-curves are colored blue.
%%%%%%%%%%%%%%%%%%%%%%%%%%%%%%%%%%%%%%%%%%%%
\begin{figure}[!htbp]
     \centering
     \begin{subfigure}[b]{0.4\textwidth}
     	\def\svgwidth{1\columnwidth}
     	\centering
     	%
	%% Creator: Inkscape 1.2.2 (732a01da63, 2022-12-09), www.inkscape.org
%% PDF/EPS/PS + LaTeX output extension by Johan Engelen, 2010
%% Accompanies image file 'HD-S3.pdf' (pdf, eps, ps)
%%
%% To include the image in your LaTeX document, write
%%   \input{<filename>.pdf_tex}
%%  instead of
%%   \includegraphics{<filename>.pdf}
%% To scale the image, write
%%   \def\svgwidth{<desired width>}
%%   \input{<filename>.pdf_tex}
%%  instead of
%%   \includegraphics[width=<desired width>]{<filename>.pdf}
%%
%% Images with a different path to the parent latex file can
%% be accessed with the `import' package (which may need to be
%% installed) using
%%   \usepackage{import}
%% in the preamble, and then including the image with
%%   \import{<path to file>}{<filename>.pdf_tex}
%% Alternatively, one can specify
%%   \graphicspath{{<path to file>/}}
%% 
%% For more information, please see info/svg-inkscape on CTAN:
%%   http://tug.ctan.org/tex-archive/info/svg-inkscape
%%
\begingroup%
  \makeatletter%
  \providecommand\color[2][]{%
    \errmessage{(Inkscape) Color is used for the text in Inkscape, but the package 'color.sty' is not loaded}%
    \renewcommand\color[2][]{}%
  }%
  \providecommand\transparent[1]{%
    \errmessage{(Inkscape) Transparency is used (non-zero) for the text in Inkscape, but the package 'transparent.sty' is not loaded}%
    \renewcommand\transparent[1]{}%
  }%
  \providecommand\rotatebox[2]{#2}%
  \newcommand*\fsize{\dimexpr\f@size pt\relax}%
  \newcommand*\lineheight[1]{\fontsize{\fsize}{#1\fsize}\selectfont}%
  \ifx\svgwidth\undefined%
    \setlength{\unitlength}{294.84135111bp}%
    \ifx\svgscale\undefined%
      \relax%
    \else%
      \setlength{\unitlength}{\unitlength * \real{\svgscale}}%
    \fi%
  \else%
    \setlength{\unitlength}{\svgwidth}%
  \fi%
  \global\let\svgwidth\undefined%
  \global\let\svgscale\undefined%
  \makeatother%
  \begin{picture}(1,0.63401607)%
    \lineheight{1}%
    \setlength\tabcolsep{0pt}%
    \put(0,0){\includegraphics[width=\unitlength,page=1]{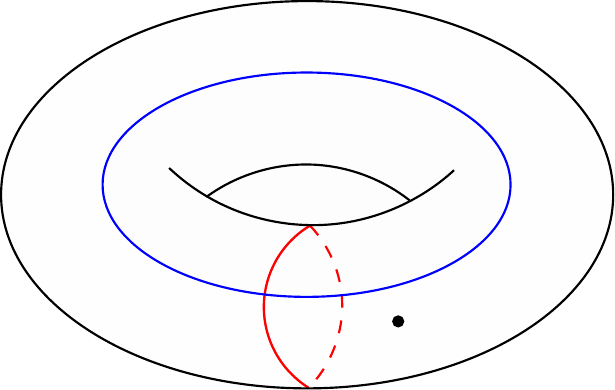}}%
    \put(0.66918602,0.0887051){\makebox(0,0)[lt]{\lineheight{1.25}\smash{\begin{tabular}[t]{l}$z$\end{tabular}}}}%
  \end{picture}%
\endgroup%

     	\caption{$\mathcal{H}_1$}
     	\label{fig:3sphere-a}
     \end{subfigure}
     \hfill
     \begin{subfigure}[b]{0.4\textwidth}
     	\def\svgwidth{1\columnwidth}
     	\centering
     	%
	%% Creator: Inkscape 1.2.2 (732a01da63, 2022-12-09), www.inkscape.org
%% PDF/EPS/PS + LaTeX output extension by Johan Engelen, 2010
%% Accompanies image file 'HD-S3-admissible.pdf' (pdf, eps, ps)
%%
%% To include the image in your LaTeX document, write
%%   \input{<filename>.pdf_tex}
%%  instead of
%%   \includegraphics{<filename>.pdf}
%% To scale the image, write
%%   \def\svgwidth{<desired width>}
%%   \input{<filename>.pdf_tex}
%%  instead of
%%   \includegraphics[width=<desired width>]{<filename>.pdf}
%%
%% Images with a different path to the parent latex file can
%% be accessed with the `import' package (which may need to be
%% installed) using
%%   \usepackage{import}
%% in the preamble, and then including the image with
%%   \import{<path to file>}{<filename>.pdf_tex}
%% Alternatively, one can specify
%%   \graphicspath{{<path to file>/}}
%% 
%% For more information, please see info/svg-inkscape on CTAN:
%%   http://tug.ctan.org/tex-archive/info/svg-inkscape
%%
\begingroup%
  \makeatletter%
  \providecommand\color[2][]{%
    \errmessage{(Inkscape) Color is used for the text in Inkscape, but the package 'color.sty' is not loaded}%
    \renewcommand\color[2][]{}%
  }%
  \providecommand\transparent[1]{%
    \errmessage{(Inkscape) Transparency is used (non-zero) for the text in Inkscape, but the package 'transparent.sty' is not loaded}%
    \renewcommand\transparent[1]{}%
  }%
  \providecommand\rotatebox[2]{#2}%
  \newcommand*\fsize{\dimexpr\f@size pt\relax}%
  \newcommand*\lineheight[1]{\fontsize{\fsize}{#1\fsize}\selectfont}%
  \ifx\svgwidth\undefined%
    \setlength{\unitlength}{294.84135111bp}%
    \ifx\svgscale\undefined%
      \relax%
    \else%
      \setlength{\unitlength}{\unitlength * \real{\svgscale}}%
    \fi%
  \else%
    \setlength{\unitlength}{\svgwidth}%
  \fi%
  \global\let\svgwidth\undefined%
  \global\let\svgscale\undefined%
  \makeatother%
  \begin{picture}(1,0.63401607)%
    \lineheight{1}%
    \setlength\tabcolsep{0pt}%
    \put(0,0){\includegraphics[width=\unitlength,page=1]{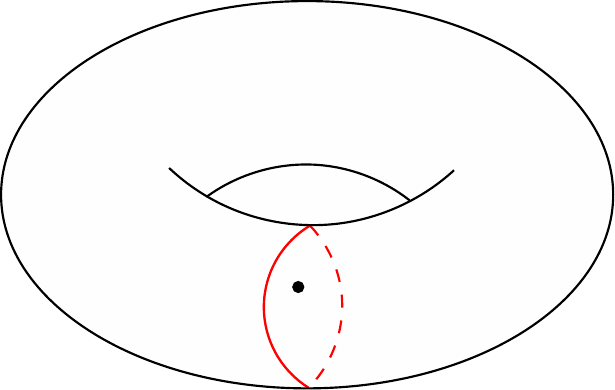}}%
    \put(0.50648304,0.14464655){\makebox(0,0)[lt]{\lineheight{1.25}\smash{\begin{tabular}[t]{l}$z$\end{tabular}}}}%
    \put(0,0){\includegraphics[width=\unitlength,page=2]{HD-S3-admissible.pdf}}%
    \put(0.42428382,0.22382654){\makebox(0,0)[lt]{\lineheight{1.25}\smash{\begin{tabular}[t]{l}$a$\end{tabular}}}}%
    \put(0.4014217,0.14250106){\makebox(0,0)[lt]{\lineheight{1.25}\smash{\begin{tabular}[t]{l}$b$\end{tabular}}}}%
    \put(0.4354265,0.01434384){\makebox(0,0)[lt]{\lineheight{1.25}\smash{\begin{tabular}[t]{l}$c$\end{tabular}}}}%
  \end{picture}%
\endgroup%

     	\caption{$\mathcal{H}_2$}
     	\label{fig:3sphere-b}
     \end{subfigure}
        \caption{Two Heegaard diagrams of $(S^3,z)$}
        \label{fig:two-HD-of-3sphere}
\end{figure}

\begin{figure}[!htbp]
     \centering
     \begin{subfigure}[b]{0.4\textwidth}
     	\def\svgwidth{1\columnwidth}
     	\centering
     	%
	%% Creator: Inkscape 1.2.2 (732a01da63, 2022-12-09), www.inkscape.org
%% PDF/EPS/PS + LaTeX output extension by Johan Engelen, 2010
%% Accompanies image file 'HD-S1xS2-a.pdf' (pdf, eps, ps)
%%
%% To include the image in your LaTeX document, write
%%   \input{<filename>.pdf_tex}
%%  instead of
%%   \includegraphics{<filename>.pdf}
%% To scale the image, write
%%   \def\svgwidth{<desired width>}
%%   \input{<filename>.pdf_tex}
%%  instead of
%%   \includegraphics[width=<desired width>]{<filename>.pdf}
%%
%% Images with a different path to the parent latex file can
%% be accessed with the `import' package (which may need to be
%% installed) using
%%   \usepackage{import}
%% in the preamble, and then including the image with
%%   \import{<path to file>}{<filename>.pdf_tex}
%% Alternatively, one can specify
%%   \graphicspath{{<path to file>/}}
%% 
%% For more information, please see info/svg-inkscape on CTAN:
%%   http://tug.ctan.org/tex-archive/info/svg-inkscape
%%
\begingroup%
  \makeatletter%
  \providecommand\color[2][]{%
    \errmessage{(Inkscape) Color is used for the text in Inkscape, but the package 'color.sty' is not loaded}%
    \renewcommand\color[2][]{}%
  }%
  \providecommand\transparent[1]{%
    \errmessage{(Inkscape) Transparency is used (non-zero) for the text in Inkscape, but the package 'transparent.sty' is not loaded}%
    \renewcommand\transparent[1]{}%
  }%
  \providecommand\rotatebox[2]{#2}%
  \newcommand*\fsize{\dimexpr\f@size pt\relax}%
  \newcommand*\lineheight[1]{\fontsize{\fsize}{#1\fsize}\selectfont}%
  \ifx\svgwidth\undefined%
    \setlength{\unitlength}{294.84135111bp}%
    \ifx\svgscale\undefined%
      \relax%
    \else%
      \setlength{\unitlength}{\unitlength * \real{\svgscale}}%
    \fi%
  \else%
    \setlength{\unitlength}{\svgwidth}%
  \fi%
  \global\let\svgwidth\undefined%
  \global\let\svgscale\undefined%
  \makeatother%
  \begin{picture}(1,0.63493239)%
    \lineheight{1}%
    \setlength\tabcolsep{0pt}%
    \put(0,0){\includegraphics[width=\unitlength,page=1]{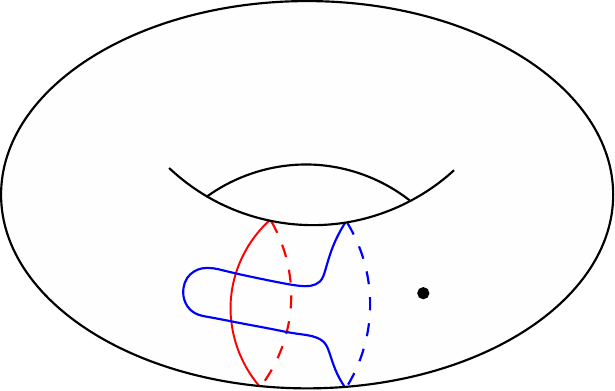}}%
    \put(0.71004713,0.13552624){\makebox(0,0)[lt]{\lineheight{1.25}\smash{\begin{tabular}[t]{l}$z$\end{tabular}}}}%
    \put(0,0){\includegraphics[width=\unitlength,page=2]{HD-S1xS2-a.pdf}}%
    \put(0.36214979,0.20355086){\makebox(0,0)[lt]{\lineheight{1.25}\smash{\begin{tabular}[t]{l}$a$\\\end{tabular}}}}%
    \put(0.34536238,0.0724779){\makebox(0,0)[lt]{\lineheight{1.25}\smash{\begin{tabular}[t]{l}$b$\end{tabular}}}}%
  \end{picture}%
\endgroup%

     	\caption{$\mathcal{H}_3$}
     	\label{fig:s1xs2-a}
     \end{subfigure}
     \hfill
     \begin{subfigure}[b]{0.4\textwidth}
     	\def\svgwidth{1\columnwidth}
     	\centering
     	%
	%% Creator: Inkscape 1.2.2 (732a01da63, 2022-12-09), www.inkscape.org
%% PDF/EPS/PS + LaTeX output extension by Johan Engelen, 2010
%% Accompanies image file 'HD-S1xS2-b.pdf' (pdf, eps, ps)
%%
%% To include the image in your LaTeX document, write
%%   \input{<filename>.pdf_tex}
%%  instead of
%%   \includegraphics{<filename>.pdf}
%% To scale the image, write
%%   \def\svgwidth{<desired width>}
%%   \input{<filename>.pdf_tex}
%%  instead of
%%   \includegraphics[width=<desired width>]{<filename>.pdf}
%%
%% Images with a different path to the parent latex file can
%% be accessed with the `import' package (which may need to be
%% installed) using
%%   \usepackage{import}
%% in the preamble, and then including the image with
%%   \import{<path to file>}{<filename>.pdf_tex}
%% Alternatively, one can specify
%%   \graphicspath{{<path to file>/}}
%% 
%% For more information, please see info/svg-inkscape on CTAN:
%%   http://tug.ctan.org/tex-archive/info/svg-inkscape
%%
\begingroup%
  \makeatletter%
  \providecommand\color[2][]{%
    \errmessage{(Inkscape) Color is used for the text in Inkscape, but the package 'color.sty' is not loaded}%
    \renewcommand\color[2][]{}%
  }%
  \providecommand\transparent[1]{%
    \errmessage{(Inkscape) Transparency is used (non-zero) for the text in Inkscape, but the package 'transparent.sty' is not loaded}%
    \renewcommand\transparent[1]{}%
  }%
  \providecommand\rotatebox[2]{#2}%
  \newcommand*\fsize{\dimexpr\f@size pt\relax}%
  \newcommand*\lineheight[1]{\fontsize{\fsize}{#1\fsize}\selectfont}%
  \ifx\svgwidth\undefined%
    \setlength{\unitlength}{294.84135111bp}%
    \ifx\svgscale\undefined%
      \relax%
    \else%
      \setlength{\unitlength}{\unitlength * \real{\svgscale}}%
    \fi%
  \else%
    \setlength{\unitlength}{\svgwidth}%
  \fi%
  \global\let\svgwidth\undefined%
  \global\let\svgscale\undefined%
  \makeatother%
  \begin{picture}(1,0.63493239)%
    \lineheight{1}%
    \setlength\tabcolsep{0pt}%
    \put(0,0){\includegraphics[width=\unitlength,page=1]{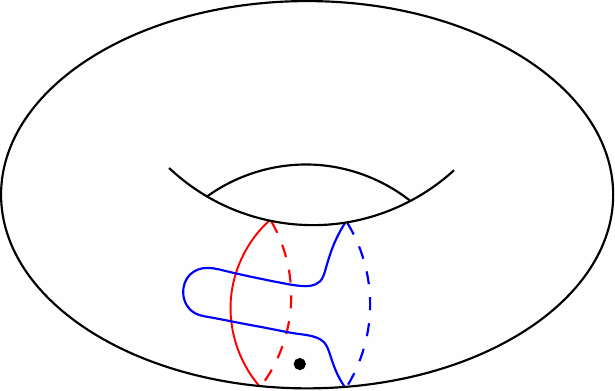}}%
    \put(0.50892149,0.02019926){\makebox(0,0)[lt]{\lineheight{1.25}\smash{\begin{tabular}[t]{l}$z$\end{tabular}}}}%
    \put(0,0){\includegraphics[width=\unitlength,page=2]{HD-S1xS2-b.pdf}}%
    \put(0.36214979,0.20355086){\makebox(0,0)[lt]{\lineheight{1.25}\smash{\begin{tabular}[t]{l}$a$\end{tabular}}}}%
    \put(0.34536238,0.0724779){\makebox(0,0)[lt]{\lineheight{1.25}\smash{\begin{tabular}[t]{l}$b$\end{tabular}}}}%
  \end{picture}%
\endgroup%

     	\caption{$\mathcal{H}_4$}
     	\label{fig:s1xs2-b}
     \end{subfigure}
        \caption{Two Heegaard diagrams of $(S^1\times S^2,z)$}
        \label{fig:two-HD-of-s1xs2}
\end{figure}
\end{example}

%%%%%%%%%%%%%%%%%%%%%%%%%%%%%%%%%%%%%%%%%%%%

%%% TORUS DRAWING %%%

Inspired by the grid chain complexes, here is a ``naive'' definition of the 
Heegaard Floer chain complex associated to a pointed Heegaard diagram.
%of genus $1$.

\begin{definition}
  Given a pointed Heegaard diagram $\heeg = (\heegsurf, \alphas, \betas, z)$, we 
  define the generating set
  \[
    \genset (\heeg)=\{\text{matchings between $\alphas$ and $\betas$ in } 
    \heeg\},
  \]
  where a matching is a one-to-one correspondence between the $\alpha$- and 
  $\beta$-circles.

  %Suppose that $g (\heegsurf) = 1$; then $\alphas = \{\alpha\}$ and $\betas = 
  %\{\beta\}$.
  We define the \emph{hat} and \emph{minus} \emph{``naive'' Heegaard Floer 
    chain complexes} respectively as the $\bF$-module and the $\bF [U]$-module
  \[
    \CFh (\heeg) = \bF \langle \genset (\heeg) \rangle, \qquad
    \CFm (\heeg) = \bF[U] \langle \genset (\heeg) \rangle,
  \]
  with boundary homomorphisms $\bdyh \colon \CFh (\heeg) \to \CFh (\heeg)$ and 
  $\bdym \colon \CFm (\heeg) \to \CFm (\heeg)$ given by
  \[
    \bdyh (\x) = \sum_{\y \in \genset (\heeg)} \sum_{\substack{p \in \pi_2 (\x, 
        \y) \\ p \cap \{z\} = \emptyset}} \y, \qquad
    \bdym (\x) = \sum_{\y \in \genset (\heeg)} \sum_{p \in \pi_2 (\x, \y)} 
    U^{n_z (p)} \y,
  \]
  where $n_z (p)$ is the multiplicity of $z$ in $p$, and
  %where
  $\pi_2 (\x, \y)$ denotes the space of domains from $\x$ to $\y$.  Here, a 
  \emph{region} is a connected component of $\heegsurf \setminus (\alphas \cup 
  \betas)$, and a \emph{domain $p$ from $\x$ to $\y$} is a $2$-chain on 
  $\heegsurf$ that is a linear combination of regions with non-negative 
  coefficients, such that
  \[
    \bdy ((\bdy p) \cap \alphas) = \y - \x
  \]
  in the induced orientation on $\bdy p$.
\end{definition}
  %$\bdy p$ consists exactly of an arc on $\alpha$ from $\x$ to $\y$ and an arc 
  %on $\beta$ from $\y$ to $\x$, in its induced orientation.
  
%\begin{definition}
%    In a similar fashion as for grid homology, we define the set
%    \[
%    S(\mathcal{H})=\{\text{matchings between $\alpha$ and $\beta$}\}
%    \]
%    of specified intersection points.
%    Just like for the grid case, the hat and minus chain complexes are generated (over $\bF$ and $\bF[U]$ respectively) by the set of matchings. The differentials are now defined as
%    \[
%    \widehat{\partial}\bm{x}
%    =\sum_{\bm{y}\in S(\mathcal{H})} \sum\limits_{\substack{p\in\pi_2(\bm{x},\bm{y}) \\ p\cap\{z\}=\emptyset}} \bm{y}, \qquad \partial^- \bm{x}=\sum_{\bm{y}\in S(\mathcal{H})} \sum_{p\in\pi_2(\bm{x},\bm{y})} U^{z(p)} \bm{y}
%    \]
%    where again $z(p)$ is the multiplicity of the domain $p$ at $z$.
%    In the sum, $\pi_2(\bm{x},\bm{y})$ is the set
%      of ``domains''. We do not define it rigorously here, but in the
%      above examples those domains are disks from $\bm{x}$ to $\bm{y}$
%      (with orientation).
%\end{definition}

Something seems immediately suspicious in this definition. Indeed, if we 
concatenate a domain from $\x$ to $\y$ and a domain from $\y$ to $\z$, we will 
get a domain from $\x$ to $\z$; but clearly we do not want a scenario where 
$\y$ appears as a term in $\bdy (\x)$, and $\z$ appears as a term in both $\bdy 
(\y)$ and $\bdy (\x)$. In grid homology, concatenating two rectangles always 
results in a domain that is not a rectangle; here, we need to instead rely on a 
\emph{Maslov index} on elements of $\pi_2 (\x, \y)$. We will not go into 
details here, except that such an index exists and can be computed 
\cite[Corollary~4.10]{Lip06}, and gives rise to the homological \emph{Maslov 
  grading} on our chain complex. The idea, then, is that we will only count 
domains that are of index $1$; concatenating two such domains will give a 
domain of index $2$, which is not counted. Applying this idea back to the grid 
chain complex, this would eliminate
\begin{itemize}
  \item domains that have reflex angles in them; and
  \item domains that contain points in $\x$ or $\y$ in them.
\end{itemize}
This is why only empty rectangles are counted in grid homology. In any case, this 
represents the first modification we would like to make.

When $g (\heegsurf) = 1$, the condition that a domain must be of Maslov index 
$1$ will imply that it must in fact be an immersed bigon with non-reflex angles 
at its corners.  (A bigon is a disk with two corners.)

\begin{example}
  We can compute the chain complexes associated to the two diagrams of $(S^3, 
  z)$ in Example~\ref{examplenaiveHF}. For $\heeg_1$, let $x$ be the unique 
  intersection point of the $\alpha$ and $\beta$ curves. Since there are no 
  domains, it follows that
  \begin{align*}
    \widehat{\CF}(\mathcal{H}_1) &\cong\bF\langle x\rangle, & \CF^-(\mathcal{H}_1) 
    &\cong\bF[U]\langle x\rangle, &
    \widehat{\partial} &\equiv 0 \equiv \partial^-,\\
    \HFh (\heeg_1) & \cong \bF, & \wHFm (\heeg_1) & \cong \bF [U].
  \end{align*}
  For $\heeg_2$, labeling the intersection points by $a$, $b$, and $c$ as in 
  Figure~\ref{fig:3sphere-b}, there are two domains, one from $a$ to $b$ 
  containing $z$, and one from $c$ to $b$ not containing $z$. Thus,
  \begin{align*}
    & \widehat{\CF}(\mathcal{H}_2) \cong \bF \langle a, b, c \rangle, & &
    \CF^-(\mathcal{H}_2) \cong \bF [U] \langle a, b, c \rangle,\\
    & \bdyh (a) = \bdyh (b) = 0, \quad \bdyh (c) = b, \quad & &
    \bdym (a) = U b, \quad \bdym (b) = 0, \quad \bdym (c) = b,\\
    & \HFh (\heeg_2) \cong \bF \langle a, b \rangle / (b) \cong \bF, & &
    \wHFm (\heeg_2) \cong \bF [U] \langle a + U c, b \rangle / (b) \cong \bF [U].
  \end{align*}
\end{example}

\subsection{Admissibility}

So far, we have not run into any major problems. Let us look an example that 
seems more mysterious:

\begin{example}
  \label{exp:s1-s2}
  We now look at the two diagrams of $(S^1\times S^2, z)$ in 
  Example~\ref{examplenaiveHF}.  One can compute that
  \[
    \widehat{\wHF}(\mathcal{H}_3)\cong \bF\oplus\bF, \qquad  
    \widehat{\wHF}(\mathcal{H}_4)\cong 0,
  \]
  and
  \[
    \wHF^-(\mathcal{H}_3)\cong \bF[U]\oplus\bF[U], \qquad
    \wHF^-(\mathcal{H}_4)\cong \faktor{\bF[U]}{(1+U)}.
  \]
  This seems to suggest that $\HFh$ and $\wHFm$ are not invariants of 
  $3$-manifolds! One may naturally suspect that the culprit is the ``naive'' 
  definition, but in fact, it is not---we are indeed computing the correct 
  chain complexes here.

  So what gives? It turns out that this has to do with a requirement on the 
  Heegaard diagrams called \emph{admissibility}, a condition that guarantees 
  the diagram to be suitable for computation.  In other words, we cannot simply 
  choose any Heegaard diagram and expect it to give us the correct answer. This 
  is closely related to the $\Spinc$-decomposition of Heegaard Floer homology, 
  and we will briefly touch upon this issue again in the last lecture.
%
  %So we see that the ``naive'' approach to Heegaard Floer
  %homology does not quite work; to set up
  %the actual theory one needs more subtle arguments.
  %Indeed, when $b_1(Y)>0$ (so for example for $S^1\times S^2$),
  %not every Heegaard diagram will be suitable for computing
  %Heegaard Floer homology --- the diagram should satisfy a certain
  %condition, which has been called \emph{admissibility}, and will be
  %mentioned later.
\end{example}

\subsection{$\bdy^2 \equiv 0$}

There is, in fact, a more serious problem that has been lurking in our 
``naive'' approach to the Heegaard Floer complex: We have not proved $\bdy^2 
\equiv 0$ for either the \emph{hat} or the \emph{minus} flavor; in other words, 
we do not know that we have defined chain complexes.

This turns out to be a real problem. See Figure~\ref{fig:d-squared-not-zero};  
the shaded domain has three decompositions into domains, 
each of which has the correct index. This means that $\z$ (in green) appears as 
a term in $\bdyh \circ \bdyh (\x)$ that is not canceled by any other term.

\begin{figure}[!htbp]
	\def\svgwidth{0.4\columnwidth}
	\centering
	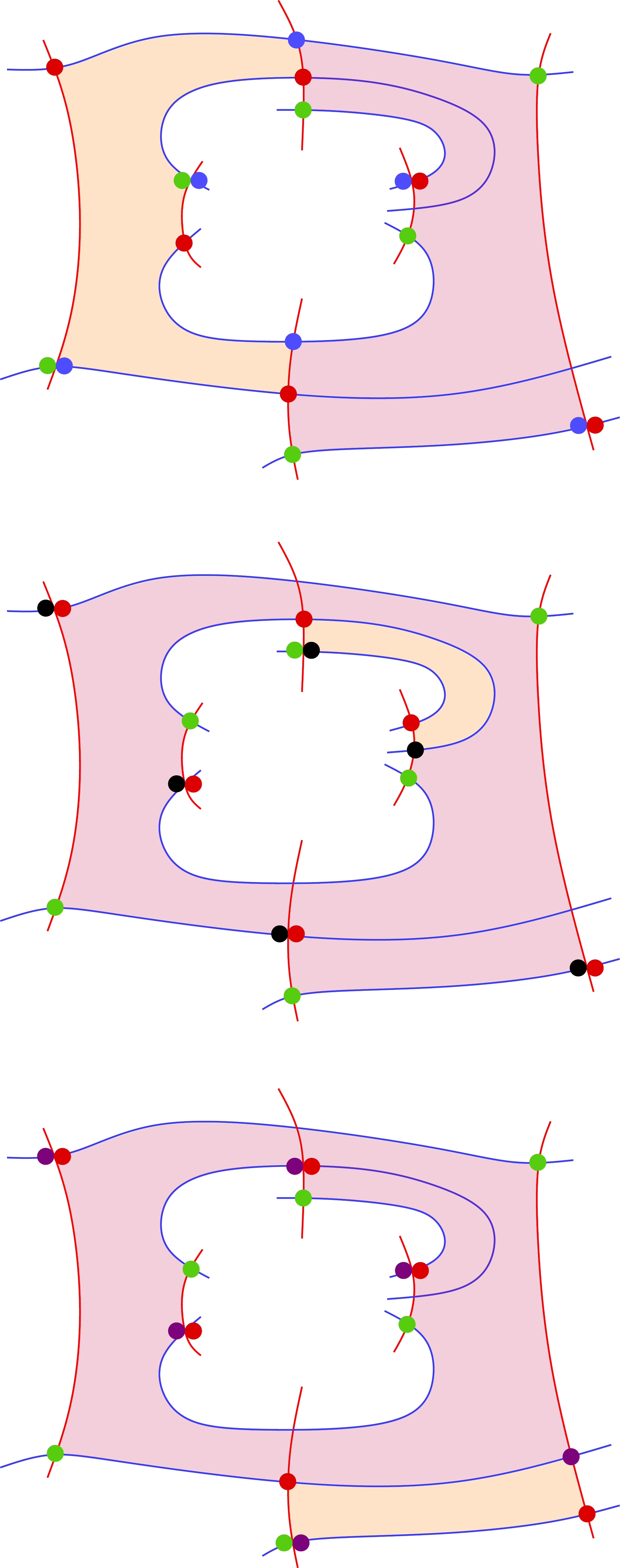

  \caption{A composite domain that has three decompositions into domains that 
    are counted in the differential}
	\label{fig:d-squared-not-zero}
\end{figure}

This illustrates a key point, upon understanding which, we graduate from the 
``naive'' approach: The Heegaard Floer chain complex is not supposed to be 
defined by counting polygons on a Heegaard diagram. Instead, it is defined by 
counting (pseudo)holomorphic curves on a space \emph{associated} to the 
Heegaard diagram, distinct from the Heegaard diagram itself unless $g 
(\heegsurf) = 1$.  One can think of the domains we have been counting as 
``shadows'' of such holomorphic curves; given a domain of the right index, 
there may be $0$, $1$, or any number of holomorphic curves with that shadow. It 
is the count of holomorphic curves, rather than the count of domains, that 
satisfies $\bdy^2 \equiv 0$.

%\section{Third lecture}
\section{Lagrangian Floer homology}
\subsection{Setup}

So now we understand that the boundary homomorphism in the Heegaard Floer chain 
complex counts certain holomorphic curves in some yet mysterious space. One of 
the questions most often asked is ``why'': What is the motivation behind the 
Heegaard Floer complex? In this lecture, we attempt to address this by 
providing an introduction to Lagrangian Floer homology; in short, Heegaard 
Floer homology can be defined via Lagrangian Floer homology, which in turn can 
be understood as a kind of infinite-dimensional Morse homology.

\emph{Lagrangian intersection Floer homology} (or simply \emph{Lagrangian Floer 
  homology}) is an invariant of $\left(M,\omega,
L_0,L_1\right)$, where $(M,\omega)$ is a symplectic manifold, and
$L_0$, $L_1$ are two transversely intersecting Lagrangian submanifolds of $(M, 
\omega )$.
%It is the underlying
%theory of Heegaard Floer homology, and it connects the latter to Morse
%homology.
The following is a quick survey that mostly follows the exposition in
\cite{Ped18}.  We will make whichever assumptions we need to simplify the
exposition, regardless of whether they hold in Heegaard
Floer theory. First, we will assume that $\pi_2(M)=0$---or less restrictively,
\begin{equation}
  \label{eqn:pi-2}
  \int_{S^2} f^* \omega=0
\end{equation}
for all $\left[f\right]\in\pi_2(M)$---and
$\pi_1(L_i)=0$ for each $i=0,1$.

The idea is to perform ``Morse homology'' on an infinite-dimensional
manifold associated to $(M, \omega, L_0, L_1)$ defined as follows.
Consider first
\[
\mathcal{P}=\mathcal{P}(L_0,L_1)=\{\gamma \colon  [0,1] \to M\,\vert\, 
\gamma\text{ is smooth}, \, \gamma(0)\in L_0, \, \gamma(1)\in L_1\},
\]
the set of smooth paths from $L_0$ to $L_1$, with the $\mathcal{C}^\infty$ 
topology. Note that $\mathcal{P}$ is not necessarily connected.
Choose an element $\widehat{\gamma}\in\mathcal{P}$ (and by doing so, a
connected component of $\mathcal{P}$), and consider the universal
cover $\widetilde{\mathcal{P}}$ based at $\widehat{\gamma}$.

This means that we get a set of paths of paths:
\[
  \begin{split}
    \widetilde{\mathcal{P}} = \widetilde{\mathcal{P}} 
    (L_0,L_1;\widehat{\gamma}) & =\{(\gamma,w)\,\vert\,
    \gamma\in\mathcal{P}, w \text{ is a path from } \widehat{\gamma}
    \text{ to } \gamma, \textit{ i.e.}\\
    & \qquad w\colon [0,1]\times [0,1]\to M
    \text{ with } w(\blank,r)\in\mathcal{P} \text{ for all } r\in
    [0,1], \\
    & \qquad w(s,0)=\widehat{\gamma}(s), \, w(s,1)=\gamma(s) \text{ for all } s 
    \in [0, 1]\}/\sim,
  \end{split}
  \]
where $(\gamma,w)\sim(\gamma,w')$ if $w,w'$ are homotopic relative to $(\widehat{\gamma},\gamma)$.

\begin{figure}[!htbp]
	\def\svgwidth{0.4\columnwidth}
	\centering
	%
	%% Creator: Inkscape 1.2.2 (732a01da63, 2022-12-09), www.inkscape.org
%% PDF/EPS/PS + LaTeX output extension by Johan Engelen, 2010
%% Accompanies image file 'L0intL1-homotopy.pdf' (pdf, eps, ps)
%%
%% To include the image in your LaTeX document, write
%%   \input{<filename>.pdf_tex}
%%  instead of
%%   \includegraphics{<filename>.pdf}
%% To scale the image, write
%%   \def\svgwidth{<desired width>}
%%   \input{<filename>.pdf_tex}
%%  instead of
%%   \includegraphics[width=<desired width>]{<filename>.pdf}
%%
%% Images with a different path to the parent latex file can
%% be accessed with the `import' package (which may need to be
%% installed) using
%%   \usepackage{import}
%% in the preamble, and then including the image with
%%   \import{<path to file>}{<filename>.pdf_tex}
%% Alternatively, one can specify
%%   \graphicspath{{<path to file>/}}
%% 
%% For more information, please see info/svg-inkscape on CTAN:
%%   http://tug.ctan.org/tex-archive/info/svg-inkscape
%%
\begingroup%
  \makeatletter%
  \providecommand\color[2][]{%
    \errmessage{(Inkscape) Color is used for the text in Inkscape, but the package 'color.sty' is not loaded}%
    \renewcommand\color[2][]{}%
  }%
  \providecommand\transparent[1]{%
    \errmessage{(Inkscape) Transparency is used (non-zero) for the text in Inkscape, but the package 'transparent.sty' is not loaded}%
    \renewcommand\transparent[1]{}%
  }%
  \providecommand\rotatebox[2]{#2}%
  \newcommand*\fsize{\dimexpr\f@size pt\relax}%
  \newcommand*\lineheight[1]{\fontsize{\fsize}{#1\fsize}\selectfont}%
  \ifx\svgwidth\undefined%
    \setlength{\unitlength}{137.40475423bp}%
    \ifx\svgscale\undefined%
      \relax%
    \else%
      \setlength{\unitlength}{\unitlength * \real{\svgscale}}%
    \fi%
  \else%
    \setlength{\unitlength}{\svgwidth}%
  \fi%
  \global\let\svgwidth\undefined%
  \global\let\svgscale\undefined%
  \makeatother%
  \begin{picture}(1,0.85276774)%
    \lineheight{1}%
    \setlength\tabcolsep{0pt}%
    \put(0,0){\includegraphics[width=\unitlength,page=1]{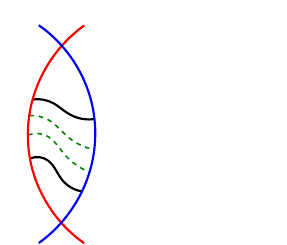}}%
    \put(0.30729825,0.78620599){\color[rgb]{1,0,0}\makebox(0,0)[lt]{\lineheight{1.25}\smash{\begin{tabular}[t]{l}$L_0$\end{tabular}}}}%
    \put(0.02719775,0.78756638){\color[rgb]{0,0,1}\makebox(0,0)[lt]{\lineheight{1.25}\smash{\begin{tabular}[t]{l}$L_1$\end{tabular}}}}%
    \put(0,0){\includegraphics[width=\unitlength,page=2]{L0intL1-homotopy.pdf}}%
    \put(0.18619835,0.51618954){\color[rgb]{0,0,0}\makebox(0,0)[lt]{\lineheight{1.25}\smash{\begin{tabular}[t]{l}$\widehat{\gamma}$\end{tabular}}}}%
    \put(0.17949716,0.15156702){\color[rgb]{0,0,0}\makebox(0,0)[lt]{\lineheight{1.25}\smash{\begin{tabular}[t]{l}$\gamma$\end{tabular}}}}%
    \put(-0.00398003,0.38110159){\color[rgb]{0,0.50196078,0}\makebox(0,0)[lt]{\lineheight{1.25}\smash{\begin{tabular}[t]{l}$r$\end{tabular}}}}%
    \put(0.34088857,0.31176302){\color[rgb]{0,0.50196078,0}\makebox(0,0)[lt]{\lineheight{1.25}\smash{\begin{tabular}[t]{l}$w$\end{tabular}}}}%
  \end{picture}%
\endgroup%

	\caption{A path $w$ from $\widehat{\gamma}$ to $\gamma$, where each value 
		of $r$ gives a path $w (\blank, r)$ from $L_0$ to $L_1$}
	\label{fig:intersection-L0-L1}
\end{figure}

It is on this space that we will apply ``Morse homology''.

\subsection{Morse homology: An interlude}

We briefly recall here the construction of Morse homology. Let $M$ be a closed, 
smooth manifold. Recall that a smooth function $f \colon M \to \mathbb{R}$ is 
\emph{Morse} if at each of its critical points (where in local coordinates all 
first derivatives $\partial f / \partial x_i$ vanish), the Hessian matrix $H = 
(\partial^2 f / \partial x_i \partial x_j)$ is not degenerate (\ie has non-zero 
determinant). By the Morse Lemma, all such critical points must be isolated.  
The \emph{Morse index} $\Ind (x)$ of a non-degenerate critical point $x$ is the 
dimension of the largest subspace of the tangent space that is negative 
definite.  Classical Morse theory says that Morse functions are generic, and 
that, given a Morse function $f$, $M$ is homotopy equivalent to a CW complex 
with an $n$-cell for each critical point of $f$ of index $n$.

In fact, one can even reconstruct a CW chain complex of $M$ using $f$ and 
auxiliary data. Since the critical points of $f$ correspond to cells in a CW 
complex, let $\CM (M; f)$ be the $\bF$-module generated by the critical points 
of $f$. To define a boundary homomorphism
%$\bdy \colon \CM (M; f) \to \CM (M; f)$,
the idea is to count the number of negative gradient flow lines from a critical 
point of index $n$ to one of index $n - 1$. The notion of the gradient requires 
the choice of a Riemannian metric $g$ on $M$, meaning that we will have a chain 
complex $\CM (M; f, g) = \CM (M; f)$, with boundary homomorphism $\bdy \colon 
\CM (M; f, g) \to \CM (M; f, g)$.  Concretely, we consider the moduli space 
$\moduli_g (x, y)$ of curves $u \colon \mathbb{R} \to M$ satisfying
\begin{itemize}
  \item $\displaystyle \frac{du}{dt} = - \operatorname{grad} f$;
  \item $\lim_{t \to - \infty} u (t) = x$; and
  \item $\lim_{t \to \infty} u (t) = y$.
\end{itemize}
To ensure that $\moduli_g (x, y)$ is a manifold of dimension $\Ind (x) - 
\Ind (y)$, the pair $(f, g)$ needs to satisfy the \emph{Morse--Smale} condition, 
which is luckily also generic. (The desired property that $\moduli_g (x, y)$ is 
be a manifold of the right dimension, is known as \emph{transversality}.) There 
is a natural action of $\R$ on $\moduli_g (x, y)$ by shifting $t \in \R$; to 
consider flow lines rather than flows, one takes the quotient $\moduli_g (x, y) 
/ \R$, which has dimension $\Ind (x) - \Ind (y) - 1$.  The manifold $\moduli_g 
(x, y) / \R$ can be compactified to a compact manifold $\widehat{\moduli}_g (x, 
y)$ with boundary by including the limits of sequences of elements, which are 
broken gradient flow lines.  (The desired property that 
$\moduli_g (x, y)$ can be compactified to a compact manifold with boundary is 
known as \emph{compactness}.)
When the index difference between $x$ and $y$ is $1$, the moduli space 
$\widehat{\moduli}_g (x, y)$ is a compact manifold of dimension $0$.  Thus,
one can define the matrix element $\langle \bdy x, y \rangle$ as follows: If 
$\Ind (x) - \Ind (y) = 1$, it is the number of elements in $\widehat{\moduli}_g 
(x, y)$ modulo $2$; otherwise, it is $0$.

To show that $\bdy^2 \equiv 0$, we consider $\widehat{\moduli}_g (x, z)$ where 
$\Ind (x) - \Ind (z) = 2$. This is a compact $1$-dimensional manifold with 
boundary, and so there is an even number of boundary points; since these 
correspond to broken flow lines from $x$ to $z$ that appear in $\bdy \circ \bdy 
(x)$, this means that all terms cancel in pairs.

The main result is then that the \emph{Morse homology}
\[
  \HM (M; f, g) = H_* (\CM (M; f, g), \bdy)
\]
is isomorphic to the CW or singular homology $H_* (M; \Z / 2)$ of $M$. One 
could deduce from this that $\HM (M; f, g)$ is invariant under a different 
choice of $f$ or $g$. Absent the isomorphism with singular homology, one would 
need to define chain homotopy equivalences between $\CM (M; f, g)$ and $\CM (M; 
f', g')$. This is important, since in the context of Floer homology, there is 
often no classical invariant to rely on, and this is the only way to prove 
invariance.

\subsection{The differential}

Turning back to Lagrangian Floer homology, we would like to follow the
steps above in this setting. The role of the manifold $M$ will be played
by the space $\widetilde{\mathcal{P}}$; what we need now is a
``Morse function''.

\begin{definition}
The \textit{action functional} $\mathcal{A} \colon \widetilde{\mathcal{P}} \to 
\mathbb{R}$ is defined by
\[
\mathcal{A}\left([\gamma,w]\right)= \int_{[0,1]^2} w^*\omega.
\]
\end{definition}
(Here and below, we abbreviate $[(\gamma, w)]$ to $[\gamma, w]$.) We need to 
check that it does not depend on the choice of $w$, \ie that 
%(assuming $\pi _2 (M)=0$)
we have
\begin{equation}
  \label{eqn:choice-of-w}
  (\gamma,w)\sim(\gamma,w') \implies 
  \mathcal{A}([\gamma,w])=\mathcal{A}([\gamma,w']) .
\end{equation}
This follows from Equation \eqref{eqn:pi-2}.
%The verification of this statement  is left to the reader as an exercise. 

\begin{exercise}
  Verify Equation \eqref{eqn:choice-of-w}.
\end{exercise}

In local coordinates,
we have
\[
  w^* \omega \left(\frac{\partial}{\partial s}, \frac{\partial}{\partial 
      r}\right) = \omega \left(\frac{\partial w}{\partial s}, \frac{\partial 
      w}{\partial r}\right),
\]
and so
\[
\mathcal{A}\left([\gamma,w]\right)=\int_0^1 \int_0^1 \omega 
\left(\frac{\partial w}{\partial s}, \frac{\partial w}{\partial r}\right) \, dr 
\, ds.
\]
Now, we want to run ``Morse homology'' on
$(\widetilde{\mathcal{P}},\mathcal{A})$, which is structurally similar
to the finite-dimensional case but has a lot of technicalities that
come with infinite-dimensional spaces.

Let us compute $\grad \mathcal{A}$. We need a Riemannian metric on
$\widetilde{\mathcal{P}}$, which we will get by integrating a Riemannian
metric on $M$.
%We start by understanding the tangent space of $\widetilde{\mathcal{P}}$.
This is possible because
a tangent vector $\xi\in T_{[\gamma,w]}\widetilde{\mathcal{P}}$ can be 
regarded,
%as something that ``pushes'' $\gamma (s)$ a little
roughly, as an infinitesimal change to $\gamma (s)$
for every $s \in [0, 1]$; in other words, it can be viewed as a collection of 
tangent vectors
\(
\xi(s)\in T_{\gamma(s)} M.
\)

%So
We will find a natural metric on $M$ using $\omega$ and a
further choice. %For this choice,
Recall that the following structures are related.
\[\begin{tikzcd}
    \text{Symplectic structure } \omega && \begin{tabular}{c} Almost complex\\ structure $J$ \end{tabular} \\
	& \begin{tabular}{c} Riemannian\\ structure $g$\end{tabular}
	\arrow[no head, from=1-1, to=1-3]
	\arrow[no head, from=1-3, to=2-2]
	\arrow[no head, from=2-2, to=1-1]
\end{tikzcd}\]
%where
Here, 
a choice of any two of these structures that are compatible determines the 
third one.  This is
related to the structure groups of these
structures, which are subgroups of $GL_{2n} (\mathbb{R})$.
\[\begin{tikzcd}
	{Sp(2n)} && {GL_n (\mathbb{C})} \\
	& {SO(2n)}
	\arrow[no head, from=1-1, to=1-3]
	\arrow[no head, from=1-3, to=2-2]
	\arrow[no head, from=2-2, to=1-1]
\end{tikzcd}\]
The intersection of any two of these is $U(n)$.
%In order two of these structures to determine the third, we need some kind of
%compatibility, leading to the following definition.
The notion of compatibility mentioned above is specified in the following 
definition.

\begin{definition}
    An almost complex structure $J$ is \emph{compatible} with $\omega$ if
    \begin{align}
      \label{eqn:(1)}
      \omega (J v,J {v'})& =\omega(v,v') & &\text{for all } v,v'\in T_p M,\\
      \label{eqn:(2)}
      \omega (v,Jv)& >0 & &\text{for all } v\neq 0 \in T_p M.
    \end{align}
\end{definition}

It can be shown that compatible almost complex structures
on symplectic manifolds are abundant.

Given $\omega$ and a compatible $J$, we can now define
the metric $g_J$ by the formula
\begin{equation}
  \label{eqn:g_J}
  g_J(v,v') = \omega(v,Jv').
\end{equation}

\begin{exercise}
Check that $g_J$ is a Riemannian structure.
\end{exercise}

With a metric $g_J$ on $M$, we can define a metric on
$\widetilde{\mathcal{P}}$ as follows:
Given $\xi_1,\xi_2 \in T_{ [\gamma,w]} \widetilde{\mathcal{P}}$,
we can define
\[
  \langle\xi_1,\xi_2 \rangle=\int_0^1 g_{J}(\xi_1(s),\xi_2(s)) \, 
  ds.
\]
But in fact, to allow sufficient flexibility to achieve genericity, we will 
allow the almost complex structure $J$ to vary for different values of $s$ in 
the integral. This means that we will make a choice of a family
\[
  \acsfamily = \{ J_s \}_{0 \leq s \leq 1}
\]
of $\omega$-compatible almost complex structures on $M$, and define
\begin{equation}
  \label{eqn:g_J-int}
  \langle\xi_1,\xi_2 \rangle=\int_0^1 g_{J_s}(\xi_1(s),\xi_2(s)) 
  \, ds.
\end{equation}
%In fact, we can allow the almost complex structure $J$ to depend on
%the parameter $s$ in the integral, meaning we need a choice of a family
%$\mathcal{J}=\{J_s\}_{0\leq s\leq 1}$ of $\omega$-compatible almost
%complex structures on $M$ rather than a single almost complex
%structure. (This more general choice will be important in a later step
%of the theory.)

We now determine $\grad_{\langle\cdot,\cdot\rangle}
\mathcal{A}([\gamma,w])$, by first computing
\begin{align*}
    d\mathcal{A}_{[\gamma,w]} (\xi)&=\int_0^1 
    \omega\left(\frac{\partial\gamma}{\partial s},\xi(s)\right)\, ds\\
    &=\int_0^1 \omega\left(J_s\frac{\partial\gamma}{\partial s},J_s 
      \xi(s)\right)\, ds\\
    &=\int_0^1 g_{J_s}\left(J_s\frac{\partial\gamma}{\partial 
        s},\xi(s)\right)\, ds\\
    &=\left\langle J_s \frac{\partial\gamma}{\partial s},\xi \right\rangle,
\end{align*}
where the second, third, and fourth equalities follow from Equations \eqref{eqn:(1)}, 
\eqref{eqn:g_J}, and \eqref{eqn:g_J-int} respectively.
This means that
\[
  \grad_{\langle\cdot,\cdot\rangle} \mathcal{A}([\gamma,w])(s)=J_s 
  \frac{\partial\gamma}{\partial s},
\]
which is depicted in Figure~\ref{fig:intersection-L0-L1-grad}.
\begin{figure}[!htbp]
	\def\svgwidth{0.4\columnwidth}
	\centering
	%
	%% Creator: Inkscape 1.2.2 (732a01da63, 2022-12-09), www.inkscape.org
%% PDF/EPS/PS + LaTeX output extension by Johan Engelen, 2010
%% Accompanies image file 'L0intL1-homotopy2.pdf' (pdf, eps, ps)
%%
%% To include the image in your LaTeX document, write
%%   \input{<filename>.pdf_tex}
%%  instead of
%%   \includegraphics{<filename>.pdf}
%% To scale the image, write
%%   \def\svgwidth{<desired width>}
%%   \input{<filename>.pdf_tex}
%%  instead of
%%   \includegraphics[width=<desired width>]{<filename>.pdf}
%%
%% Images with a different path to the parent latex file can
%% be accessed with the `import' package (which may need to be
%% installed) using
%%   \usepackage{import}
%% in the preamble, and then including the image with
%%   \import{<path to file>}{<filename>.pdf_tex}
%% Alternatively, one can specify
%%   \graphicspath{{<path to file>/}}
%% 
%% For more information, please see info/svg-inkscape on CTAN:
%%   http://tug.ctan.org/tex-archive/info/svg-inkscape
%%
\begingroup%
  \makeatletter%
  \providecommand\color[2][]{%
    \errmessage{(Inkscape) Color is used for the text in Inkscape, but the package 'color.sty' is not loaded}%
    \renewcommand\color[2][]{}%
  }%
  \providecommand\transparent[1]{%
    \errmessage{(Inkscape) Transparency is used (non-zero) for the text in Inkscape, but the package 'transparent.sty' is not loaded}%
    \renewcommand\transparent[1]{}%
  }%
  \providecommand\rotatebox[2]{#2}%
  \newcommand*\fsize{\dimexpr\f@size pt\relax}%
  \newcommand*\lineheight[1]{\fontsize{\fsize}{#1\fsize}\selectfont}%
  \ifx\svgwidth\undefined%
    \setlength{\unitlength}{139.84717933bp}%
    \ifx\svgscale\undefined%
      \relax%
    \else%
      \setlength{\unitlength}{\unitlength * \real{\svgscale}}%
    \fi%
  \else%
    \setlength{\unitlength}{\svgwidth}%
  \fi%
  \global\let\svgwidth\undefined%
  \global\let\svgscale\undefined%
  \makeatother%
  \begin{picture}(1,0.83787419)%
    \lineheight{1}%
    \setlength\tabcolsep{0pt}%
    \put(0,0){\includegraphics[width=\unitlength,page=1]{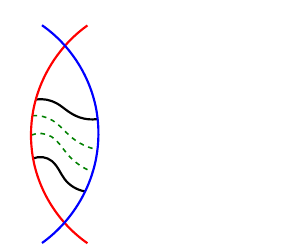}}%
    \put(0.3128268,0.77247493){\color[rgb]{1,0,0}\makebox(0,0)[lt]{\lineheight{1.25}\smash{\begin{tabular}[t]{l}$L_0$\end{tabular}}}}%
    \put(0.03761824,0.77381156){\color[rgb]{0,0,1}\makebox(0,0)[lt]{\lineheight{1.25}\smash{\begin{tabular}[t]{l}$L_1$\end{tabular}}}}%
    \put(0,0){\includegraphics[width=\unitlength,page=2]{L0intL1-homotopy2.pdf}}%
    \put(-0.00391052,0.37679573){\color[rgb]{0,0.50196078,0}\makebox(0,0)[lt]{\lineheight{1.25}\smash{\begin{tabular}[t]{l}$r$\end{tabular}}}}%
    \put(0.35399678,0.29706028){\color[rgb]{0,0.50196078,0}\makebox(0,0)[lt]{\lineheight{1.25}\smash{\begin{tabular}[t]{l}$w$\end{tabular}}}}%
    \put(0.20041136,0.50802686){\color[rgb]{0,0,0}\makebox(0,0)[lt]{\lineheight{1.25}\smash{\begin{tabular}[t]{l}$\widehat{\gamma}$\end{tabular}}}}%
    \put(0.1795753,0.16084564){\color[rgb]{0,0,0}\makebox(0,0)[lt]{\lineheight{1.25}\smash{\begin{tabular}[t]{l}$\gamma$\end{tabular}}}}%
  \end{picture}%
\endgroup%

  \caption{The gradient of $\mathcal{A}$ represented (in brown) as a collection 
    of tangent vectors ``rotated'' counterclockwise by $J_s$}
	%\caption{Different values of $r$ give paths from $L_0$ to $L_1$}
	\label{fig:intersection-L0-L1-grad}
\end{figure}

We will now determine the critical points of $\mathcal{A}$. Since the
almost complex structure $J_s$ is an automorphism of
$T_{\gamma(s)}M$, we have
\[
  \grad \mathcal{A}([\gamma,w])=0 \iff
  \frac{\partial\gamma}{\partial s}=0 \quad \text{for all } s\in[0,1],
\]
\ie $\gamma$ is constant. As $\gamma$ is a path from
$L_0$ to $L_1$, it must be a constant path with $\gamma(0)=\gamma(1)\in \Lint$.  
This means
that there is a bijection between the set of critical points of
$\mathcal{A}$ and the intersection points $\Lint$.

Next, we move on to gradient flows, which are maps $u\colon \mathbb{R}\to 
\widetilde{\mathcal{P}}$ such that

\begin{minipage}{.6\linewidth}
\begin{itemize}
    \item $\displaystyle \frac{du}{dt}=-\grad\mathcal{A}\implies 
      \frac{du}{dt}\bigg\rvert_{t=t_0}=-J_s \frac{\partial u(t_0)}{\partial 
        s}$,
    \item $\lim_{t\rightarrow -\infty} u(t)=x$,
    \item $\lim_{t\rightarrow +\infty} u(t)=y$.
\end{itemize}
\end{minipage}
\hfill
\begin{minipage}{.35\linewidth}
\centering
\includegraphics[width=.4\linewidth]{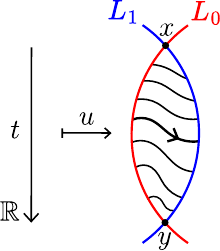}
%\captionof{figure}{This is a nice image}
\label{fig:flow-lines-Lagrangian}
\end{minipage}

\begin{comment}
\begin{figure}[!htbp]
    \centering
    \includegraphics[scale=1]{images/L0intL1-family-paths.pdf}
    \label{fig:flow-lines-Lagrangian}
\end{figure}
\begin{wrapfigure}{r}{0.3\textwidth}
    \includegraphics[scale=1]{images/L0intL1-family-paths.pdf}
    \label{fig:flow-lines-Lagrangian}
\end{wrapfigure}
\end{comment}

Equivalently, we can regard $u$ as a function $u\colon [0,1]\times\mathbb{R} 
\to M$, with parameters $s \in [0, 1]$ and $t \in \R$, and translate the first 
condition above to the condition
\[
\frac{\partial u}{\partial t}+J_s \frac{\partial u}{\partial s}=0.
\]
Of course, this is nothing but the Cauchy--Riemann equation; in other words, 
$u$ is a \emph{pseudoholomorphic curve}.
It is also called a \emph{(pseudo)holomorphic disk} or a 
\emph{(psuedo)holomorphic bigon},
as $[0,1]\times \mathbb{R}$ is conformally equivalent
to $\mathbb{D}^2 \setminus \{i, -i \} \subset\mathbb{C}$, a disk with two 
corners (punctures) on its boundary.
%, with punctures on $i$ and $-i$.

Therefore, to do ``Morse homology'', we define a chain complex
\[
  \CF(M,\omega,L_0,L_1;\widehat{\gamma},\mathcal{J})=\bF\langle \Lint\rangle.
\]
Given $x,y\in \Lint$, consider the moduli spaces
\begin{multline*}
  \mathcal{M}_\mathcal{J}(x,y)=\left\{u\colon [0,1]\times \mathbb{R}\to M\, 
    \middle\vert\, \frac{\partial u}{\partial t}+J_s \frac{\partial u}{\partial 
      s}=0, \, u(\blank,t)=\gamma(\blank)\in\mathcal{P} \;\; \forall 
    t,\right.\\
  \left.\lim_{t\to +\infty}u(s,t)=x, \, \lim_{t\to -\infty}u(s,t)=y\;\; \forall 
    s \right\}
\end{multline*}
%
%%% mauvaaaiiiiissss %%%
%
Often, we want one more ``finite energy'' condition:
\[
\int_{[0,1]\times\mathbb{R}} u^* \omega<\infty.
\]
Hopefully, we will have
\begin{enumerate}
  \item (Transversality) $\mathcal{M}_\mathcal{J}(x,y)$ is a manifold with an 
    $\R$-action, of dimension $\Ind (u)$ for some notion of \emph{index} $\Ind$ 
    of $u$; and
  \item (Compactness) $\mathcal{M}_\mathcal{J}(x,y) / \R$ can be compactified 
    to a manifold $\widehat{\moduli}_{\acsfamily} (x, y)$ with boundary.
\end{enumerate}
%manifold of dimension $Ind(u)$, the \emph{index} of $u$,  with an
%$\mathbb{R}$-action.
Then we can define the boundary homomorphism $\bdy \colon \CF (M, \omega, L_0, 
L_1; \widehat{\gamma}, \acsfamily) \to \CF (M, \omega, L_0, L_1; 
\widehat{\gamma}, \acsfamily)$ by
\[
\partial (x) =\sum\limits_{y\in\Lint} 
\#\widehat{\mathcal{M}}_\mathcal{J}^{\Ind(u)=1}(x,y) \cdot y.
\]
Transversality is necessary to ensure that the moduli space is $0$-dimensional 
and can be counted. Compactness is necessary to prove that $\bdy (x)$ is a 
finite sum. Both are necessary for proving $\bdy^2 \equiv 0$.
%In order to have a finite sum, we also need that the quotient
%space $\faktor{\mathcal{M}_\mathcal{J}^{\Ind(u)=1}(x,y)}{\mathbb{R}}$
%is compact --- this fact will be a consequence of a more general
%compactness result, called Gromov compactness.

\subsection{Technicalities and additional structures}

Below, we discuss some technicalities that arise in the definitions above.

First, $\Ind(u)$ in fact depends on the class $B=[u]\in\pi_2(M,L_0\cup L_1)$.
Therefore, we need to write
\[
  \moduli_{\acsfamily} (x, y) = \bigsqcup_{B \in \pi_2 (M, \Lint)} 
  \moduli_{\acsfamily}^B (x, y).
\]
In other words, different components of $\moduli_{\acsfamily} (x, y)$ 
might be of 
different dimensions. We then really have
\[
\partial (x)=\sum\limits_{y\in\Lint} \sum\limits_{\substack{B\in\pi_2(M,L_0\cup 
    L_1) \\ \Ind(B)=1}} \#\widehat{\mathcal{M}}_\mathcal{J}^{B}(x,y) \cdot y
\]

Second, the issue of transversality is way more technical than in 
finite-dimensional
Morse homology, requiring one to do analysis on Sobolev spaces.
This is the step that requires us to consider
families of almost complex structures on $M$, rather than a single almost
complex structure, in our earlier definitions.
%One also needs to relax the smooth condition to the appropriate regularity.
The lack of appropriate regularity and transversality results is the reason 
that some versions of Floer homology are not yet well defined.

Third, the desired compactness result is provided by the following, known as 
\emph{Gromov compactness}.

\begin{theorem}[Floer {\cite{Flo88}}, Oh {\cite{Oh93}}, \cf Gromov 
  {\cite{Gro85}}, \etal]
For a dense subset $\{\mathcal{J}\}$ of $1$-parameter families of almost 
complex structures on $M$, $\mathcal{M}_\mathcal{J}^B (x,y)/\mathbb{R}$ can be 
compactified to a compact manifold $\widehat{\mathcal{M}}_\mathcal{J}^B (x,y)$ 
with boundary, of dimension $\Ind(B)$.
\end{theorem}

Here, $\partial \widehat{\mathcal{M}}_\mathcal{J}^B
(x,y)$ consists of holomorphic buildings (\ie broken flow lines and
bubbles). Extra work then goes into dealing with bubbling.

When $\Ind (B) = 1$, $\widehat{\moduli}_{\acsfamily}^B (x, y)$ is a compact 
$0$-dimensional manifold, which guarantees that $\bdy (x)$ is a finite sum.  
When $\Ind (B) = 2$, under favorable conditions (when bubbles are absent or 
appear in a controlled manner), $\widehat{\moduli}_{\acsfamily}^B (x, y)$ is a 
compact $1$-dimensional manifold, whose boundary points correspond to broken 
flow lines. As in Morse homology, this allows one to prove $\bdy \circ \bdy 
\equiv 0$.  The pictures in this proof are reminiscent of the pictures of 
$\bdy^2 \equiv 0$ in grid homology. Fortunately, bubbling is in general not an 
issue for Heegaard Floer homology.

%As it
%turns out, when $\Ind(B)=1$, then this approach concludes that the
%corresponding moduli space is compact. When $Ind(B)=2$, then (after
%taking the quotient with the ${\mathbb {R}}$-action) we have a
%1-dimensional space; and the compactifying points will come from
%trajectories counted in $\partial \circ\partial$.  In favourable
%cases this then allows
%one to prove $\partial\circ\partial=0$, which has pictures similar
%to the proof of $\partial\circ\partial=0$ in grid homology. The presence of
%bubbles, however, obstruct the usual proof to pass through, and in
%some cases we do not have a boundary operator --- in the case, of
%Heegaard Floer homology, fortunately bubbles will not cause
%such a problem.

Finally, as mentioned before, unlike Morse homology on a finite-dimensional
manifold, there is no homology
we could relate our Lagrangian Floer homology to, that could provide a proof of 
invariance. In the present context, it turns out that the chain homotopy type 
of
\[
  \CF (M, \omega, L_0, L_1; \widehat{\gamma}, \acsfamily)
\]
is an invariant of the quadruple $(M, \omega, L_0, L_1)$. To prove this, one 
must define chain homotopy equivalences between complexes defined using 
different auxiliary data $(\widehat{\gamma}, \acsfamily)$. In particular, given 
two families of almost complex structures $\acsfamily_1$ and $\acsfamily_2$, we 
choose a family of almost complex structures $\acsfamily$ that ``connects'' the 
two, and use $\acsfamily$ to define a chain homotopy equivalence between the 
chain complexes.

We end this discussion with a couple of features in this theory:

\begin{itemize}
  \item Given a divisor $D^{2n-2} \subset M^{2n}$, we can modify the definition 
    of $\moduli_{\acsfamily}^B (x, y)$ to block all holomorphic curves $u$ that 
    intersect $D$. This would result in a chain complex
    \[
      \CF (M, \omega, L_0, L_1, D; \widehat{\gamma}, \acsfamily).
    \]

  \item For brevity, we will write $\CF (L_0, L_1)$ to mean $\CF (M, \omega, 
    L_0, L_1)$. Given three mutually transversely intersecting Lagrangian 
    submanifolds $L_0$, $L_1$, and $L_2$, there is a map
    \[
      \mu_2 \colon \CF (L_0, L_1) \otimes \CF (L_1, L_2) \to \CF (L_0, L_2),
    \]
    defined by counting (pseudo)holomorphic triangles. A holomorphic triangle 
    is much like a holomorphic bigon, but it has three instead of two punctures 
    on the boundary of its domain (\eg $\mathbb{D}^2 \setminus \{1, \zeta, 
    \zeta^2\} \subset \mathbb{C}$); see Figure~\ref{fig:holo-triangle}.
    \begin{figure}[!htbp]
    	\def\svgwidth{0.4\columnwidth}
    	\centering
    	%
	%% Creator: Inkscape 1.2.2 (732a01da63, 2022-12-09), www.inkscape.org
%% PDF/EPS/PS + LaTeX output extension by Johan Engelen, 2010
%% Accompanies image file 'holo-triangle.pdf' (pdf, eps, ps)
%%
%% To include the image in your LaTeX document, write
%%   \input{<filename>.pdf_tex}
%%  instead of
%%   \includegraphics{<filename>.pdf}
%% To scale the image, write
%%   \def\svgwidth{<desired width>}
%%   \input{<filename>.pdf_tex}
%%  instead of
%%   \includegraphics[width=<desired width>]{<filename>.pdf}
%%
%% Images with a different path to the parent latex file can
%% be accessed with the `import' package (which may need to be
%% installed) using
%%   \usepackage{import}
%% in the preamble, and then including the image with
%%   \import{<path to file>}{<filename>.pdf_tex}
%% Alternatively, one can specify
%%   \graphicspath{{<path to file>/}}
%% 
%% For more information, please see info/svg-inkscape on CTAN:
%%   http://tug.ctan.org/tex-archive/info/svg-inkscape
%%
\begingroup%
  \makeatletter%
  \providecommand\color[2][]{%
    \errmessage{(Inkscape) Color is used for the text in Inkscape, but the package 'color.sty' is not loaded}%
    \renewcommand\color[2][]{}%
  }%
  \providecommand\transparent[1]{%
    \errmessage{(Inkscape) Transparency is used (non-zero) for the text in Inkscape, but the package 'transparent.sty' is not loaded}%
    \renewcommand\transparent[1]{}%
  }%
  \providecommand\rotatebox[2]{#2}%
  \newcommand*\fsize{\dimexpr\f@size pt\relax}%
  \newcommand*\lineheight[1]{\fontsize{\fsize}{#1\fsize}\selectfont}%
  \ifx\svgwidth\undefined%
    \setlength{\unitlength}{275.55475486bp}%
    \ifx\svgscale\undefined%
      \relax%
    \else%
      \setlength{\unitlength}{\unitlength * \real{\svgscale}}%
    \fi%
  \else%
    \setlength{\unitlength}{\svgwidth}%
  \fi%
  \global\let\svgwidth\undefined%
  \global\let\svgscale\undefined%
  \makeatother%
  \begin{picture}(1,0.5882667)%
    \lineheight{1}%
    \setlength\tabcolsep{0pt}%
    \put(0,0){\includegraphics[width=\unitlength,page=1]{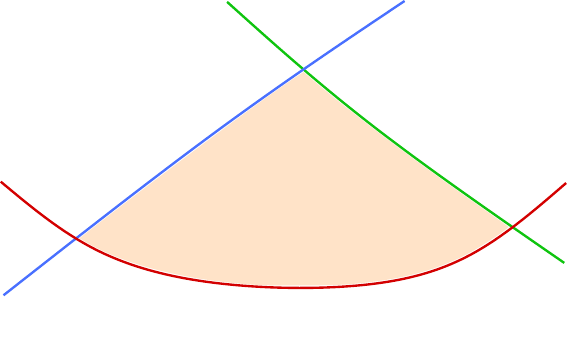}}%
    \put(0.75867487,0.38283521){\color[rgb]{0.05882353,0.76862745,0.05882353}\makebox(0,0)[lt]{\lineheight{1.25}\smash{\begin{tabular}[t]{l}$L_2$\end{tabular}}}}%
    \put(0.21068931,0.39801019){\color[rgb]{0.28627451,0.43921569,1}\makebox(0,0)[lt]{\lineheight{1.25}\smash{\begin{tabular}[t]{l}$L_1$\end{tabular}}}}%
    \put(0.43325575,0.0068328){\color[rgb]{0.82745098,0,0}\makebox(0,0)[lt]{\lineheight{1.25}\smash{\begin{tabular}[t]{l}$L_0$\end{tabular}}}}%
    \put(0,0){\includegraphics[width=\unitlength,page=2]{holo-triangle.pdf}}%
  \end{picture}%
\endgroup%

    	\caption{A holomorphic triangle with two inputs and one output}
    	\label{fig:holo-triangle}
    \end{figure}
    
    The map $\mu_2$ is like a multiplication, in that it satisfies the Leibniz 
    rule with the boundary homomorphisms $\bdy$. However, it is not necessarily 
    associative; instead, it is associative only up to homotopy, given by a map
    \[
      \mu_3 \colon \CF (L_0, L_1) \otimes \CF (L_1, L_2) \otimes \CF (L_2, L_3) 
      \to \CF (L_0, L_3),
    \]
    defined by counting holomorphic quadrilaterals. Writing $\mu_1$ for $\bdy$, 
    there is in fact a collection $\{ \mu_n \}_{n = 1}^\infty$ of maps, each 
    defined by counting holomorphic $(n+1)$-gons, which satisfy a compatibility 
    condition.  This is known as an \emph{$A_\infty$-structure}. Given $(M, 
    \omega)$, the collection of Lagrangian submanifolds, the Lagrangian Floer 
    complexes---assuming they are well defined---and the $A_\infty$-structure 
    on them, together form the \emph{Fukaya category of $(M, \omega)$}.
\end{itemize}

\section{Heegaard Floer homology}
\subsection{Definition}

Finally, we are ready to give a definition of the Heegaard Floer chain complex.
Let $\mathcal{H}=(\Sigma_g,\alphas,\betas,z)$ be a pointed Heegaard diagram, 
and
consider $M=\Sym^g(\Sigma)=(\Sigma\times\dotsb\times\Sigma)/S_g$, where we
quotient the product of $g$ copies of $\Sigma$ by the group $S_g$ of 
permutations on $g$ letters. One can prove that $M$ is a manifold of dimension 
$2g$, by a standard (but non-trivial) argument.

Next, choose a symplectic structure $\omega_\Sigma$ on $\Sigma$;
this induces a symplectic structure $\omega_{\Sigma^{\times g}}$ on 
$\Sigma^{\times g} = \Sigma \times \dotsb \times \Sigma$.
%Note that
%Consider two submanifolds of $\Sym^g (\Sigma)$,
Since the $\alpha_i$'s and $\beta_i$'s are Lagrangian submanifolds of $\Sigma$, 
it follows that
$\mathbb{T}_{\alphas} = \alpha_1 \times \dotsb \times \alpha_g$
and
$\mathbb{T}_{\betas} = \beta \times \dotsb \times \beta_g$
are Lagrangian submanifolds of $\Sigma^{\times g}$ also.
By a result of Perutz \cite{Per08:unpublished}, there is a symplectic form 
$\omega_{\Sym^g (\Sigma)}$ on $\Sym^g (\Sigma)$ that coincides, outside of a 
neighborhood of the diagonal, with the pushforward of $\omega_{\Sigma^{\times 
    g}}$ under the quotient map. Since $\mathbb{T}_{\alphas}$ and 
$\mathbb{T}_{\betas}$ are disjoint from the diagonal, it follows that their 
images, which we denote by the same symbols, are also Lagrangian submanifolds 
of $(\Sym^g (\Sigma), \omega_{\Sym^g (\Sigma)})$.

\begin{definition}[Ozsv\'{a}th and Szab\'{o} \cite{OS04:HF}, reformulation by 
  Perutz \cite{Per08:unpublished, Per08}]\label{def:CF}
  Given $\heeg = (\heegsurf, \alphas, \betas, z)$, where $\heegsurf$ is of 
  genus $g$, we define the \emph{hat Heegaard Floer chain complex} of 
  $\mathcal{H}$ as
  the Lagrangian Floer complex of $M=\Sym ^g (\Sigma )$ with the two
  submanifolds $L_0 = {\mathbb {T}}_{\alphas}$
%  =\alpha_1\times\dots\times\alpha_g$
  and $L_1 = {\mathbb {T}}_{\betas}$,%= \beta_1\times\dots\times\beta_g$ as
  \[
    \widehat{\CF}(\mathcal{H})=\CF\left(\Sym^g(\Sigma), 
      \omega_{\Sym^g(\Sigma)}, \mathbb{T}_{\alphas}, \mathbb{T}_{\betas}, \{z\} 
      \times \Sym^{g-1} (\Sigma) \right),
  \]
  blocking all holomorphic bigons that intersect the divisor $\{z\} \times 
  \Sym^{g-1} (\Sigma) \subset \Sym^g (\Sigma)$. (Note that the divisor is also 
  symmetrized.)
    %
    %By a result of Perutz the symplectic form $\omega_{Sym^g(\Sigma)}$
    %on $M$ can be changed so that it stays the same outside of a fixed
    %neighbourhood of the diagonal in $M$, and the $g$-dimensiomal tori
    %${\mathbb {T}}_{\alpha}$ and ${\mathbb {T}}_{\beta}$ are
    %Lagrangian submanifolds.
\end{definition}

%The last element will be seen in the next lecture.

%\section{Fourth lecture}
\subsection{Unraveling the definition}

%Given a divisor $D^{2n-2}\subset M^{2n}$ (\ie a submanifold), we can
%``ban'' pseudo-holomorphic curves $u\colon
%[0,1]\times\mathbb{R}\longrightarrow M$ whose image intersects $D$. We
%then get another version of Floer homology, which we denote by
%$\CF(M,\omega,L_0,L_1,D)$.
%
%Given a Heegaard diagram $\mathcal{H}=(\Sigma_g,\alpha,\beta,z)$, we define
%\[
%\widehat{\CF}(\mathcal{H}):=\CF(\Sym^g(\Sigma_g),\omega_{\Sym^g(\Sigma_g)},\mathbb{T}_\alpha,\mathbb{T}_\beta,\{z\}\times \Sym^{g-1}(\Sigma_g));
%\]
%here $\{z\}\times \Sym^{g-1}(\Sigma_g)$ plays the role
%  of the divisor.

% At the end of the last lecture, we defined $\CFh (\heeg)$ as a Lagrangian Floer 
% complex
% \[
%   \widehat{\CF}(\mathcal{H})=\CF\left(\Sym^g(\Sigma), 
%     \omega_{\Sym^g(\Sigma)}, \mathbb{T}_{\alphas}, \mathbb{T}_{\betas}, \{z\} 
%     \times \Sym^{g-1} (\Sigma) \right).
% \]
%The original definition by
Ozsv\'ath and Szab\'o \cite{OS04:HF} gave a more direct definition, however, 
given Perutz's result \cite{Per08:unpublished},
it differs from the standard Lagrangian Floer homology framework only in how it 
handles energy bounds. In any case, we unravel the definition here to make it 
more explicit.

First, an intersection point of $L_0 = \mathbb{T}_{\alphas}$ and $L_1 = 
\mathbb{T}_{\betas}$ corresponds exactly to a matching between the $\alpha$- 
and $\beta$-circles. Second, given two generators $\x$ and $\y$, recall that 
the index of a holomorphic bigon $u \in \moduli_{\acsfamily} (\x, \y)$ is 
determined by the class $B = [u] \in \pi_2 (\Sym^g (\Sigma), 
\mathbb{T}_{\alphas} \cup \mathbb{T}_{\betas})$. It turns out that $B$ 
specifies and is specified by $\x$, $\y$, and a \emph{domain} on $\Sigma$, 
which is a formal linear combination of regions of $\Sigma$, \ie connected 
components of $\Sigma \setminus (\alphas \cup \betas)$. We may thus consider 
$B$ to be an element of the space $\pi_2 (\x, \y)$ of domains from $\x$ to 
$\y$.  Finally, the fact that we block holomorphic bigons that intersect the 
divisor $\{z\} \times \Sym^{g-1} (\Sigma)$ translates to blocking domains that 
contain $z$ in its interior, a condition denoted by $n_z (B) = 0$, where $n_z$ 
is the multiplicity of $z$.

\begin{definition}[Ozsv\'{a}th and Szab\'{o} \cite{OS04:HF}]
  Given a pointed Heegaard diagram $\heeg = (\heegsurf, \alphas, \betas, z)$ 
  where $\heegsurf$ is of genus $g$, and a choice of a family of almost complex 
  structures $\acsfamily = \{ J_s \}_{0 \leq s \leq 1}$, we define the \emph{hat 
    Heegaard Floer chain complex} of $(\mathcal{H}, \acsfamily)$ as the 
  $\bF$-module
  \[
    \CFh (\heeg; \acsfamily) = \bF \langle \genset (\heeg) \rangle
  \]
  with boundary homomorphism $\bdyh \colon \CFh (\heeg; \acsfamily) \to \CFh 
  (\heeg; \acsfamily)$ given by
  \[
    \bdyh (\x) = \sum_{\y \in \genset (\heeg)} \sum_{\substack{B \in \pi_2 (\x, 
        \y) \\ \Ind (B) = 1 \\ n_z (B) = 0}} \# 
    \widehat{\moduli}_{\acsfamily}^B (\x, \y) \cdot \y.
  \]

  Similarly, we define the \emph{minus Heegaard Floer chain complex} of 
  $(\mathcal{H}, \acsfamily)$ as the $\bF [U]$-module
  \[
    \CFm (\heeg; \acsfamily) = \bF[U] \langle \genset (\heeg) \rangle
  \]
  with boundary homomorphism $\bdym \colon \CFm (\heeg; \acsfamily) \to \CFm 
  (\heeg; \acsfamily)$ given by
  \[
    \bdym (\x) = \sum_{\y \in \genset (\heeg)} \sum_{\substack{B \in \pi_2 (\x, 
        \y) \\ \Ind (B) = 1}} \# \widehat{\moduli}_{\acsfamily}^B (\x, \y) 
    \cdot U^{n_z (B)} \cdot \y.
  \]

  Finally, we define the \emph{infinity} and \emph{plus Heegaard Floer chain 
    complexes} of $(\mathcal{H}, \acsfamily)$ respectively by
  \begin{align*}
    \CFi (\heeg; \acsfamily) &= \CFm (\heeg; \acsfamily) \otimes_{\bF [U]} \bF 
    [U, U^{-1}],\\
    \CFp (\heeg; \acsfamily) &= \CFi (\heeg; \acsfamily) \otimes_{\bF [U]} 
    \left( \faktor{\bF [U, U^{-1}]}{U} \right),
  \end{align*}
  viewed as chain complexes of $\bF[U]$-modules. Note that, alternatively, these 
  complexes can also be defined without reference to $\CFm (\heeg; 
  \acsfamily)$, with the boundary homomorphism explicitly defined.\footnote{As 
    we shall see later, when $\heeg$ is weakly but not strongly admissible, it 
    is possible that we cannot guarantee $\CFm (\heeg; \acsfamily)$ to be well 
    defined, while we can guarantee $\CFp (\heeg; \acsfamily)$ to be well 
    defined.}
\end{definition}

%In more details, $\widehat{\CF}(\mathcal{H})=\bF\langle
%S(\mathcal{H})\rangle$ is generated by the generators of the
%Heegaard diagram (the intersection points), and the boundary map is
%given by
%\[
%\widehat{\partial}x:=\sum_{y\in S(\mathcal{H})} \sum\limits_{\substack{B\in\pi_2(x,y) \\ Ind(B)=1\\ n_z(B)=0}} \# \widehat{\mathcal{M}}_\mathcal{J}^B (x,y) \cdot y.
%\]
%
%As for the minus version, we still have the chain complex generated by
%the intersection points, but now over the
%polynomial ring $\bF [U]$ with one variable.
%Therefore $\CF^-(\mathcal{H})=\bF[U]\langle
%S(\mathcal{H})\rangle$, and in the boundary map
%we do not require $n_z(B)=0$ anymore;
%instead, we use the formal variable
%  $U$ to record $n_z(B)$:
%\[
%\partial^- x:=\sum_{y\in S(\mathcal{H})} \sum\limits_{\substack{B\in\pi_2(x,y) \\ Ind(B)=1}} \# \widehat{\mathcal{M}}_\mathcal{J}^B (x,y) \cdot U^{n_z(B)}\cdot y.
%\]
%(Note that one would have to verify that the sum in the boundary map is
%finite, so that the map is well-defined.)
%
%There are also the plus and infinity versions (which can be defined from
%the above chain complex just as it was done in the grid homology context).

A few comments are in order:
\begin{itemize}
  \item As discussed before, both the proof that the boundary homomorphisms are 
    finite sums and the proof that $\bdy^2 \equiv 0$ require Gromov 
    compactness.
  \item When $\Ind (\x) - \Ind (\y) = 1$, Gromov compactness guarantees that 
    $\widehat{\moduli}_{\acsfamily}^B (\x, \y)$ consists of a finite number of 
    points. But to get that the boundary homomorphism is a finite sum, one 
    would also need to know that there are only finitely many elements in $B 
    \in \pi_2 (\x, \y)$ that contribute a non-zero term. We will need to 
    address this later, when we discuss the admissibility of diagrams.
  \item As in general Lagrangian Floer homology, one can prove that the chain 
    complex is independent of the choice of $\acsfamily$. Namely, given 
    $\acsfamily_1$ and $\acsfamily_2$, one uses a ``connecting'' family of 
    almost complex structures $\acsfamily$ to define a chain homotopy 
    equivalence. This justifies writing $\CFc (\heeg)$.
  \item However, the goal of Heegaard Floer homology is not to create 
    invariants of pointed Heegaard diagrams $\heeg$. Instead, it is to create 
    invariants of pointed $3$-manifolds $(Y, z)$. This is indeed the case, but 
    we defer the statement of invariance until after the discussion of 
    $\Spinc$-decompositions.
  \item When the genus $g (\Sigma) = 1$, we have that $\Sym^g (\Sigma) = 
    \Sigma$. This means that we are really counting holomorphic bigons in 
    $\Sigma$.  Given $B \in \pi_2 (\x, \y)$ with the right index, by the 
    Riemann Mapping Theorem, there is a unique holomorphic disk (up to 
    translation by $\mathbb{R}$) in the moduli space 
    $\widehat{\moduli}_{\acsfamily}^B (\x, \y)$, which means that the count is 
    necessarily $1$ for each such $B$. This means that, when $g (\Sigma) = 1$, 
    the complex $\CFc (\heeg)$ can be computed combinatorially rather than 
    holomorphically; in other words, our ``naive'' definition of the Heegaard 
    Floer chain complex, while problematic in general, is mostly fine when $g 
    (\Sigma) = 1$.
  \item Like the grid chain complexes, the Heegaard Floer chain complexes can 
    also be defined over $\Z$ or $\Z [U]$. However, even more work is required 
    here, as one needs to orient the relevant moduli spaces consistently.
\end{itemize}

%\begin{remark}
%\begin{enumerate}
%    \item $\partial^2=0$ follows again from Gromov compactness;
%    \item The homology of the chain
%      complex is independent on $\mathcal{J}$:
%      this follows from the general arguments for Lagrangian Floer homology;
%    \item It is also independent of the Heegaard diagram $\mathcal{H}$
%      representing a given three-manifold $Y$. This step rests on the
%      following result: $\mathcal{H}_1$ and $\mathcal{H}_2$ represent
%      the same three-manifold if and only if they are related by a sequence
%      of moves: isotopies, handleslides, and 
%      (de)stabilizations.
%
%    Indeed, an isotopy corresponds to a change of $\mathcal{J}$, while
%    handleslide invariance proceeds very similar to the invariance of
%    grid homology under commutation: draw the two sets of curves on
%    the same diagram,     and count holomorphic triangles.

\begin{comment}
\begin{subfigure}[b]{0.4\textwidth}
         \centering
         \includegraphics[scale=0.6]{images/HD-S3.pdf}
         \caption{Heegaard triple diagram of $S^3$}
         \label{fig:triple-diagram-S3}
     \end{subfigure}
     \hfill
     \begin{subfigure}[b]{0.4\textwidth}
         \centering
         \includegraphics[scale=0.6]{images/HD-S3-admissible.pdf}
         \caption{The handleslide is drawn on the same diagram, and
           corresponds to considering the $\beta$-curves to be either
           the blue curves or the green curves.  }
         \label{fig:handleslides}
     \end{subfigure}
\end{comment}

\subsection{$\Spinc$-decomposition, gradings, admissibility, and invariance}

We now explain, without proof or details, the $\Spinc$-decomposition and the homological Maslov grading on Heegaard Floer homology. Let
$\mathord{\circ} \in \{ \raisebox{-1ex}{$\widehat{\phantom{m}}$}, -, +, \infty 
\}$ denote any of the four flavors.
\begin{itemize}
  \item First, there is a decomposition of $\CFc (\heeg)$ into a direct sum
    \[
      \CFc (\mathcal{H})=\bigoplus_{\s\in \Spinc(Y)} \CFc (\mathcal{H},\s),
    \]
    with each summand corresponding to a $\Spinc$-structure of $Y$.
    %For space reasons, we do not explain here what $\Spinc$-structures are.  
    We give a terse summary below, and encourage the reader to consult 
    references on this topic. For example, other than the description contained 
    in \cite[Section~2.6]{OS04:HF}, see also \cite[Section~1]{Tur97} for the 
    equivalence between $\Spinc$-structures and Euler structures. For our 
    purposes, it suffices to know that
    \begin{itemize}
      \item The set $\Spinc (Y)$ of $\Spinc$-structures on $Y$ is an affine 
        copy of $H^2 (Y; \Z)$, with an action by $H^2 (Y; \Z)$;
      \item There is a one-to-one correspondence between $\Spinc$-structures 
        and \emph{Euler structures}, which are homology classes of 
        nowhere-vanishing vector fields.  In other words, each $\s \in \Spinc 
        (Y)$ is represented non-uniquely by a nowhere vanishing 
        vector field $v$ on $Y$.  
        Replacing $v$ by $-v$ corresponds to replacing $\s$ by its 
        \emph{conjugate $\Spinc$-structure} $\overline{\s}$.
      \item For $\s \in \Spinc (Y)$, one can define its \emph{first Chern 
          class} $c_1 (\s) \in H^2 (Y; \Z)$ as $\s - \overline{\s}$, 
        interpreted as the element in $H^2 (Y; \Z)$ that sends $\overline{\s}$ 
        to $\s$ by its action. Alternatively, it can be defined as the first 
        Chern class of the orthogonal complement of $v$.  Intuitively, $c_1$ 
        acts like ``multiplication by $2$'' on $\Spinc (Y)$ (which is of course 
        not well defined on an affine space).
    \end{itemize}
    The key point is then the following. Given $\x \in \genset (\heeg)$, 
    observe that $\x$ pairs up the index-$1$ and index-$2$ critical points of a 
    Morse function $f$ corresponding to $\heeg$; in fact, it specifies a 
    gradient flow line between each pair. Likewise, $z$ specifies a gradient 
    flow line between the index-$0$ and index-$3$ critical points. The gradient 
    vector field of $f$ does not vanish outside a neighborhood of these flow 
    lines, and can be completed to a vector field that is nowhere vanishing on 
    $Y$.  In short, $\x$ and $z$ determine some $\s_z (\x) \in \Spinc (Y)$. It 
    can be proved that $\pi_2 (\x, \y) \neq \emptyset$ if and only if $\s_z 
    (\x) = \s_z (\y)$, which gives the direct sum decomposition.
  \item Next, each summand $\CFc (\heeg, \s)$ is relatively $\Z / m$-graded by 
    $\Ind (B)$, where
    \[
      m = \gcd_{\xi \in H_2 (Y; \Z)} \langle c_1 (\s), \xi \rangle.
    \]
    Here, since $\Ind (B)$ only provides the grading difference between $\x$ 
    and $\y$, the grading is \emph{relative}, meaning that it is defined only 
    up to a shift.

    Note that if $c_1 (\s) \in H^2 (Y; \Z)$ is torsion, then the evaluation 
    against $\xi$ is always $0$. In this case, $m = 0$, and $\CFc (\heeg, \s)$ 
    is relatively $\Z$-graded. (With further work \cite{OS06:HF-4, OS03:HF-gr}, it can 
    be lifted to an absolute $\mathbb{Q}$-grading.) In particular, this applies 
    when $H^2 (Y; \Z) \cong H_1 (Y; \Z)$ is torsion, \ie when $b_1 (Y) = 0$.

  \item Given $\s \in \Spinc (Y)$, to compute the summand of the Heegaard 
    Floer complex associated to $\s$, we need to use a diagram $\heeg$ such 
    that $\s_z (\x) = \s$ for some $\x \in \genset (\heeg)$. This condition is 
    called \emph{$\s$-realizability}.
    
    Furthermore, for a fixed $\s$, we need to ensure that whenever $\s_z (\x) = 
    \s_z (\y) = \s$, there are only finitely many $B \in \pi_2 (\x, \y)$ that 
    contribute a non-zero term to the differential. It turns out that the 
    different flavors requires admissibility of different strengths. For a 
    given $\s$, a Heegaard diagram $\heeg$ may be \emph{strongly 
      $\s$-admissible}, \emph{weakly $\s$-admissible}, or neither. We do not 
    give the definitions here; see \cite[Section~4.2.2]{OS04:HF}. %Then
    \begin{itemize}
      \item If $\heeg$ is strongly $\s$-admissible, then $\bdyh$, $\bdym$, 
        $\bdy^+$, and $\bdy^\infty$ all give finite sums; while
      \item If $\heeg$ is weakly $\s$-admissible, then $\bdyh$ and $\bdy^+$ 
        give finite sums.
    \end{itemize}
    Notably, if $\heeg$ is strongly $\s_0$-admissible for a torsion $\s_0$, 
    then it is weakly $\s$-admissible for all $\s$.
\end{itemize}

We are now ready to state the invariance of Heegaard Floer homology.

\begin{theorem}[Ozsv\'ath and Szab\'o {\cite[Theorem~11.1]{OS04:HF}}]
  \label{thm:hf-inv} Let $\heeg$ be a pointed Heegaard diagram of $(Y, z)$.  If $\heeg$ is 
  strongly (resp.\ weakly) $\s$-admissible, then the graded isomorphism type of 
  $\HFc (\heeg, \s)$ is an invariant of $(Y, \s)$ for $\mathord{\circ} \in \{ 
  \raisebox{-1ex}{$\widehat{\phantom{m}}$}, -, +, \infty \}$ (resp.\ for 
  $\mathord{\circ} \in \{ \raisebox{-1ex}{$\widehat{\phantom{m}}$}, + \}$).  
  Thus, we may write $\HFc (Y, \s)$ for these homologies.
\end{theorem}

\begin{remark}
  As in grid homology, in fact, the graded chain homotopy type of $\CFc (\heeg, 
  \s)$ is an invariant of $(Y, \s)$. Also, it may seem that the point $z$ here 
  does not play a role. However, if we wish to get concrete homology modules 
  $\HFc (Y, \s)$ rather than isomorphism types, we would need to know that the 
  isomorphisms relating $\HFc (\heeg; \acsfamily)$ for different choices of 
  $(\heeg, \acsfamily)$ are canonical. This \emph{naturality} property is 
  proved in \cite{JTZ21}, and is required for functoriality as well as the 
  definition of a variant called involutive Heegaard Floer homology 
  \cite{HM17:HFI}. It turns out that, taking naturality into account, Heegaard 
  Floer theory is really a \emph{pointed} theory, with concrete homology 
  modules $\HFc (Y, z, \s)$.
\end{remark}

The proof of Theorem~\ref{thm:hf-inv} proceeds as follows. Analogous to 
commutation and (de)stabilization for grid diagrams, two weakly or strongly 
$\s$-admmissible pointed Heegaard diagrams are related by a finite sequence of 
isotopies, handleslides, and (de)stabilizations.

Here, isotopies refer to isotopies of the sets of $\alpha$- and $\beta$-curves, 
independent of each other, and relative to $z$ (\ie the curves are not allowed 
to cross $z$).  If the $\alpha$- and $\beta$-curves remain transverse to each 
other throughout the isotopy, this is equivalent to choosing a path of 
different $\acsfamily$'s, which has been dealt with.  Otherwise, one defines 
the chain homotopy equivalence by counting holomorphic bigons with a (time 
$t$)--dependent constraint.

Handleslides refer to handlesliding $\alpha$-circles over $\alpha$-circles, and 
sliding $\beta$-circles over $\beta$-circles. Let us focus on the latter. We 
get two sets of $\beta$-curves, $\betas_1$ and $\betas_2$, corresponding to 
$\heeg_1$ and $\heeg_2$ respectively.  As in grid homology, we may combine them 
with $\alphas$ into a single \emph{Heegaard triple diagram}
\[
  (\Sigma, \alphas, \betas_1, \betas_2, z),
\]
as in Figure~\ref{fig:triple-diagram-S3}.
\begin{figure}[!htbp]
	\def\svgwidth{1\columnwidth}
	\centering
	%
	%% Creator: Inkscape 1.2.2 (732a01da63, 2022-12-09), www.inkscape.org
%% PDF/EPS/PS + LaTeX output extension by Johan Engelen, 2010
%% Accompanies image file 'HD-S3-triple.pdf' (pdf, eps, ps)
%%
%% To include the image in your LaTeX document, write
%%   \input{<filename>.pdf_tex}
%%  instead of
%%   \includegraphics{<filename>.pdf}
%% To scale the image, write
%%   \def\svgwidth{<desired width>}
%%   \input{<filename>.pdf_tex}
%%  instead of
%%   \includegraphics[width=<desired width>]{<filename>.pdf}
%%
%% Images with a different path to the parent latex file can
%% be accessed with the `import' package (which may need to be
%% installed) using
%%   \usepackage{import}
%% in the preamble, and then including the image with
%%   \import{<path to file>}{<filename>.pdf_tex}
%% Alternatively, one can specify
%%   \graphicspath{{<path to file>/}}
%% 
%% For more information, please see info/svg-inkscape on CTAN:
%%   http://tug.ctan.org/tex-archive/info/svg-inkscape
%%
\begingroup%
  \makeatletter%
  \providecommand\color[2][]{%
    \errmessage{(Inkscape) Color is used for the text in Inkscape, but the package 'color.sty' is not loaded}%
    \renewcommand\color[2][]{}%
  }%
  \providecommand\transparent[1]{%
    \errmessage{(Inkscape) Transparency is used (non-zero) for the text in Inkscape, but the package 'transparent.sty' is not loaded}%
    \renewcommand\transparent[1]{}%
  }%
  \providecommand\rotatebox[2]{#2}%
  \newcommand*\fsize{\dimexpr\f@size pt\relax}%
  \newcommand*\lineheight[1]{\fontsize{\fsize}{#1\fsize}\selectfont}%
  \ifx\svgwidth\undefined%
    \setlength{\unitlength}{528.76463528bp}%
    \ifx\svgscale\undefined%
      \relax%
    \else%
      \setlength{\unitlength}{\unitlength * \real{\svgscale}}%
    \fi%
  \else%
    \setlength{\unitlength}{\svgwidth}%
  \fi%
  \global\let\svgwidth\undefined%
  \global\let\svgscale\undefined%
  \makeatother%
  \begin{picture}(1,0.35601525)%
    \lineheight{1}%
    \setlength\tabcolsep{0pt}%
    \put(0,0){\includegraphics[width=\unitlength,page=1]{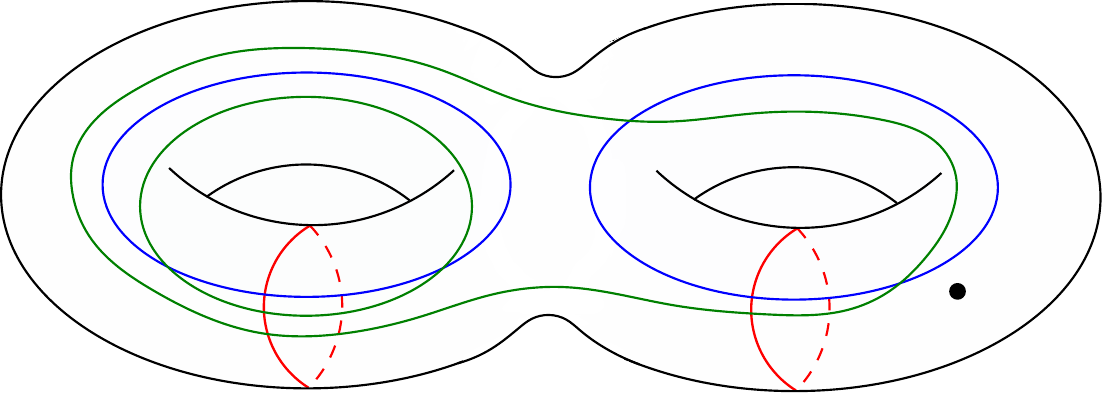}}%
    \put(0.88376523,0.07172025){\makebox(0,0)[lt]{\lineheight{1.25}\smash{\begin{tabular}[t]{l}$z$\end{tabular}}}}%
  \end{picture}%
\endgroup%

	\caption{A Heegaard triple diagram. Red circles represent $\alphas$, blue 
		circles represent $\betas_1$, and green circles represent $\betas_2$.  
		Both $(\Sigma, \alphas, \betas_1, z)$ and $(\Sigma, \alphas, \betas_2, 
		z)$ are Heegaard diagrams of $(S^3, z)$}
	\label{fig:triple-diagram-S3}
\end{figure}

One may now define a map $\Phi \colon \CFc (\heeg_1) \to \CFc (\heeg_2)$ using 
the $\mu_2$ map in the $A_\infty$-structure on the Fukaya category,
\[
  \mu_2 \colon \CF (\mathbb{T}_{\alphas}, \mathbb{T}_{\betas_1})
  \otimes
  \CF (\mathbb{T}_{\betas_1}, \mathbb{T}_{\betas_2})
  \to
  \CF (\mathbb{T}_{\alphas}, \mathbb{T}_{\betas_2}),
\]
defined by counting holomorphic triangles. To define $\Phi$, we choose a 
generator $\Theta$ in
\(
  \CF (\mathbb{T}_{\betas_1}, \mathbb{T}_{\betas_2})
\)
and define
\[
  \Phi (\x) = \mu_2 (\x \otimes \Theta).
\]
One can then use a compatibility condition of the $A_\infty$-structure to check 
that $\Phi$ is a chain map if $\Theta$ is a cycle. Then to define chain 
homotopies, one uses $\mu_3$ and count holomorphic quadrilaterals.

Finally, stabilization refers to taking the connected sum with a toroidal 
Heegaard diagram $(T^2, \alpha, \beta)$ of $S^3$ where $\alpha$ and $\beta$ 
meet exactly once, and destabilization is the inverse process. If $\heeg'$ is a 
stabilization of $\heeg$, there is a tautological identification of $\genset 
(\heeg')$ with $\genset (\heeg)$; however, to identify the boundary 
homomorphisms, some work is required to identify the corresponding moduli 
spaces by gluing holomorphic curves. Combining all these proofs together would 
then conclude the proof of invariance.

\begin{example}
  In Example~\ref{exp:s1-s2}, we looked at two Heegaard diagrams $\heeg_3$ and 
  $\heeg_4$ of $(S^1 \times S^2, z)$. There, we used a ``naive'' definition of 
  the Heegaard Floer chain complexes---which turns out to be fine in this 
  case---to compute
  \begin{align*}
    \HFh (\heeg_3) & \cong \bF \oplus \bF, & \HFh (\heeg_4) & \cong 0\\
    \wHFm (\heeg_3) & \cong \bF [U] \oplus \bF [U], & \wHFm (\heeg_4) & \cong 
    \faktor{\bF [U]}{(1 + U)}.
  \end{align*}

  We can now explain this phenomenon. Observe that $H^2 (S^1 \times S^2; \Z) 
  \cong \Z$, and so $\Spinc (S^1 \times S^2)$ is an affine copy of $\Z$.  The 
  word ``affine'' here is perhaps not necessary in this case, since there is in 
  fact a unique self-conjugate, torsion $\Spinc$-structure $\s_0$, satisfying 
  $c_1 (\s_0) = 0$. Fixing an isomorphism $H^2 (S^1 \times S^2; \Z) \cong \Z$, 
  we let $\s_i$ denote the $\Spinc$-structure satisfying $\s_i = \s_0 + i$.  As 
  it turns out, $\heeg_3$ is strongly $\s_0$-admissible, implying that it is 
  weakly $\s_i$-admissible for all $i \in \Z$. On the other hand, $\heeg_4$ is 
  strongly $\s_1$-admissible; since $\s_1$ is not torsion, one cannot conclude 
  that $\heeg_4$ is weakly $\s_i$-admissible for all $i \in \Z$. In fact, it is 
  weakly $\s_i$-admissible for all $i \neq 0$.

  From the above, we may conclude that
  \begin{align*}
    \HFh (S^1 \times S^2, \s_0) & \cong \bF \oplus \bF, & \HFh (S^1 \times S^2, 
    \s_1) & \cong 0\\
    \wHFm (S^1 \times S^2, \s_0) & \cong \bF [U] \oplus \bF [U], & \wHFm (S^1 
    \times S^2, \s_1) & \cong \faktor{\bF [U]}{(1 + U)}.
  \end{align*}
\end{example}

\begin{exercise}
  Determine the relative gradings on these homologies.
\end{exercise}

\subsection{Some computational tools}

One of the biggest strengths of Heegaard Floer homology, compared to other 
Floer-theoretic invariants, is its computability. We now present a brief survey 
of some key computational tools that have been developed in the past several 
decades, for the reader to pursue further:

\begin{itemize}
    \item \textbf{Cylindrical reformulation (Lipshitz \cite{Lip06})}\\
      Instead of counting holomorphic disks in $\Sym^g (\Sigma)$, one can count 
      holomorphic curves in $\Sigma \times [0, 1] \times \R$ instead.  The 
      price is that one must now count holomorphic maps $u \colon S \to \Sigma 
      \times [0, 1] \times \R$, whose domain $S$ is any Riemann surface 
      (possibly with multiple components). Lipshitz also provided an index formula to 
      compute $\Ind (B)$.  This is the formulation on which the following is 
      built.
    \item \textbf{Bordered Heegaard Floer homology (Lipshitz, Ozsv\'ath, and 
        Thurston \cite{LOT18})}\\
      This theory introduces cut-and-paste methods to Heegaard Floer homology.  
      It assigns a dg or $A_\infty$-module to a $3$-manifold $Y$ with 
      parametrized boundary, over an algebra associated with that boundary 
      surface. If $Y$ is obtained from gluing $Y_1$ and $Y_2$ along their 
      common parametrized boundary, then $\CFh (Y)$ can be reconstructed as a 
      derived tensor product of the modules associated to $Y_1$ and $Y_2$.

    \item \textbf{Plumbed $3$-manifolds and lattice homology (Ozsv\'{a}th and 
        Szab\'{o} \cite{OS03:plumbed}, N\'emethi \cite{Nem05, Nem08}, 
        Ozsv\'ath, Stipsicz, and Szab\'o \cite{OSS14:plumbed}, Zemke 
        \cite{Zem23:plumbed})}\\
      Ozsv\'ath and Szab\'o \cite{OS03:plumbed} first computed $\wHFm (Y)$ of certain 
      plumbed $3$-manifolds with negative-definite plumbing graphs. N\'emethi 
      \cite{Nem08} defined a combinatorial \emph{lattice (co)homology} 
      $\mathbb{H}$ for any negative-definite plumbing graph, connecting 
      Heegaard Floer homology to singularity theory. For almost rational 
      plumbing graphs (which includes the graphs that Ozsv\'ath and Szab\'o 
      considered), N\'emethi \cite{Nem05} proved that $\mathbb{H}^- (Y)$ 
      coincides with $\wHFm (Y)$ \cite{Nem05} (but $\mathbb{H}^- (Y)$ enjoys 
      more structure); this is later extended to a larger class of graphs by 
      Ozsv\'ath, Stipsicz, and Szab\'o \cite{OSS14:plumbed}, and to all 
      negative-definite plumbing trees by Zemke \cite{Zem23:plumbed}.
      %For almost rational
      %graphs, the Heegaard Floer homology group $\HF^-(Y)$ is isomorphic
      %to the lattice homology $\mathbb{H}(Y)$ of the plumbing graph
      %(regarded as the resolution graph of an isolated surface
      %singularity) (Némethi \cite{Ne05}, \cite{Ne08}).
      %Recently,
      %this result has been extended by Zemke to 
      %all negative plumbing trees of spheres in \cite{Ze21}.

    \item \textbf{Nice diagrams (Sarkar and Wang \cite{SW10})}\\
      If every region, \ie every component of $\Sigma\setminus 
      (\alphas\cup\betas)$, either is blocked by $z$ (in the \emph{hat} 
      flavor), or is a bigon or a rectangle, then the count of holomorphic 
      curves is in fact entirely combinatorial: Every domain that is a bigon or 
      a rectangle has a unique holomorphic representative. This is of great 
      importance, since this means that, by judiciously choosing a Heegaard 
      diagram, the Heegaard Floer chain complex can be combinatorially computed 
      without any reference to almost complex structures or solving any partial 
      differential equations. Sarkar and Wang showed that any closed,
      oriented three-manifold admits such Heegaard diagram, which they called \emph{nice}.

    \item \textbf{Surgery exact triangle (Ozsv\'{a}th and Szab\'{o} 
        \cite{OS04:HF-prop})}\\
      There is a long exact sequence
      \[
        \dotsb \to \HFhat (Y)\to \HFhat (Y_r(K))\to \HFhat 
        (Y_{r+1}(K))\to\dotsb
      \]
      where $Y_r(K)$ is the $r$-surgery on $Y$ along the
      knot $K$. Knowing two of these terms often allows one to compute the 
      third.

    \item \textbf{Large, integer, and rational surgery formulas
      (Ozsv\'{a}th and Szab\'{o} \cite{OS04:HFK, OS08:HFK-Z-surg, 
        OS11:HFK-Q-surg}, Rasmussen \cite{Ras03:HFK}, Manolescu and Ozsv\'ath 
      \cite{MO22:HFK-Z-surg})}\\
      For a $\Z$-homology sphere $Y$, the large surgery formula \cite{OS04:HFK, 
        Ras03:HFK} allows one to
      extract $\CFc(Y_p(K))$,% and $\CF ^-(S^3_r(K))$
      when $p \in \Z$ is large,
      %by understanding the knot
      %version $\CFK^\infty(S^3,K)$,
      from the \emph{infinity knot Floer chain complex} $\CFKi (Y, K)$ as a 
      subquotient. More complicated formulas allow one to extract $\CFc (Y_p 
      (K))$ for any $p \in \Z$ \cite{OS08:HFK-Z-surg}, extended to $\Z$-surgery 
      on links \cite{MO22:HFK-Z-surg}, and $\CFc (Y_r (K))$ for $r \in 
      \mathbb{Q}$ \cite{OS11:HFK-Q-surg}.
\end{itemize}

\subsection{Knot Floer homology: Definition and features}

The knot Floer complex mentioned above is a variant of the Heegaard Floer 
complex, and may be viewed as a \emph{relative} version of the theory.

We start by considering \emph{doubly pointed Heegaard diagrams} for a pair $(Y, 
K)$, where $K \subset Y$ is an oriented knot. (For ease of exposition, it is 
often assumed that $Y$ is a $\Z$- or $\mathbb{Q}$-homology sphere, and $K$ is 
null-homotopic.) Such diagrams are similar to the pointed Heegaard diagrams we 
have been working with so far, only with one additional point $w$,
\[
  \heeg = (\heegsurf, \alphas, \betas, z, w).
\]
Here $(\heegsurf, \alphas, \betas)$ is a Heegaard diagram of $Y$. The diagram 
$\heeg$ encodes a knot $K$ in the following way: Draw an oriented path from $z$ 
to $w$ on $\heegsurf$, avoiding $\alpha$-curves, and push it into the 
$\alpha$-handlebody; then draw an oriented path from $w$ to $z$ on $\heegsurf$, 
avoiding $\beta$-curves, and push it into the $\beta$-handlebody; $K$ is the 
union of these two oriented paths.  Formulated in an equivalent way, $w$ and 
$z$ each determine a gradient flow line from the index-$3$ critical point to 
the index-$0$ critical point, and $K$ is the union of these two gradient flow 
lines.

\begin{example}
Figure~\ref{fig:two-HD-of-3sphere-LHT} shows how one recovers a knot from a 
doubly pointed Heegaard diagram. The depicted Heegaard diagram encodes the 
left-handed trefoil.
%the left-handed trefoil from the doubly pointed Heegaard diagram of $S^3$.  We 
%connect $z$ to $w$ by avoiding the $\alpha$-curves, and connect $w$ to $z$ by 
%avoiding the $\beta$-curves.

    \begin{figure}[!htbp]
     \centering
     \begin{subfigure}[c]{0.4\textwidth}
     	\def\svgwidth{1\columnwidth}
     	\centering
     	%
	%% Creator: Inkscape 1.2.2 (732a01da63, 2022-12-09), www.inkscape.org
%% PDF/EPS/PS + LaTeX output extension by Johan Engelen, 2010
%% Accompanies image file 'HD-S3-doubly-pointed.pdf' (pdf, eps, ps)
%%
%% To include the image in your LaTeX document, write
%%   \input{<filename>.pdf_tex}
%%  instead of
%%   \includegraphics{<filename>.pdf}
%% To scale the image, write
%%   \def\svgwidth{<desired width>}
%%   \input{<filename>.pdf_tex}
%%  instead of
%%   \includegraphics[width=<desired width>]{<filename>.pdf}
%%
%% Images with a different path to the parent latex file can
%% be accessed with the `import' package (which may need to be
%% installed) using
%%   \usepackage{import}
%% in the preamble, and then including the image with
%%   \import{<path to file>}{<filename>.pdf_tex}
%% Alternatively, one can specify
%%   \graphicspath{{<path to file>/}}
%% 
%% For more information, please see info/svg-inkscape on CTAN:
%%   http://tug.ctan.org/tex-archive/info/svg-inkscape
%%
\begingroup%
  \makeatletter%
  \providecommand\color[2][]{%
    \errmessage{(Inkscape) Color is used for the text in Inkscape, but the package 'color.sty' is not loaded}%
    \renewcommand\color[2][]{}%
  }%
  \providecommand\transparent[1]{%
    \errmessage{(Inkscape) Transparency is used (non-zero) for the text in Inkscape, but the package 'transparent.sty' is not loaded}%
    \renewcommand\transparent[1]{}%
  }%
  \providecommand\rotatebox[2]{#2}%
  \newcommand*\fsize{\dimexpr\f@size pt\relax}%
  \newcommand*\lineheight[1]{\fontsize{\fsize}{#1\fsize}\selectfont}%
  \ifx\svgwidth\undefined%
    \setlength{\unitlength}{294.84135111bp}%
    \ifx\svgscale\undefined%
      \relax%
    \else%
      \setlength{\unitlength}{\unitlength * \real{\svgscale}}%
    \fi%
  \else%
    \setlength{\unitlength}{\svgwidth}%
  \fi%
  \global\let\svgwidth\undefined%
  \global\let\svgscale\undefined%
  \makeatother%
  \begin{picture}(1,0.63401607)%
    \lineheight{1}%
    \setlength\tabcolsep{0pt}%
    \put(0,0){\includegraphics[width=\unitlength,page=1]{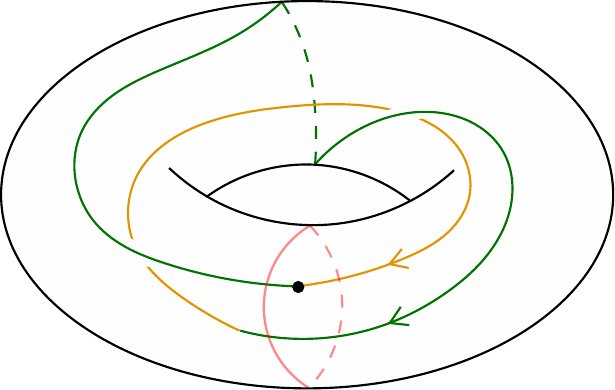}}%
    \put(0.49975384,0.14464655){\makebox(0,0)[lt]{\lineheight{1.25}\smash{\begin{tabular}[t]{l}$w$\end{tabular}}}}%
    \put(0,0){\includegraphics[width=\unitlength,page=2]{HD-S3-doubly-pointed.pdf}}%
    \put(0.4071824,0.06237536){\makebox(0,0)[lt]{\lineheight{1.25}\smash{\begin{tabular}[t]{l}$z$\end{tabular}}}}%
    \put(0,0){\includegraphics[width=\unitlength,page=3]{HD-S3-doubly-pointed.pdf}}%
  \end{picture}%
\endgroup%

     	\caption{A doubly pointed Heegaard diagram}
     	\label{fig:3sphere-2points}         
     \end{subfigure}
     \hfill
     \begin{subfigure}[c]{0.4\textwidth}
         \centering
         \includegraphics[scale=0.6]{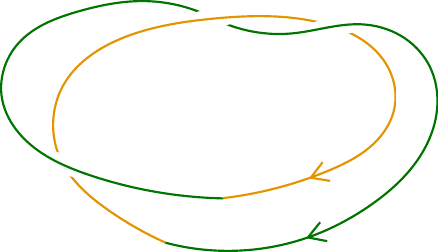}
         \caption{The left-handed trefoil}
         \label{fig:3sphere-LHT}
     \end{subfigure}
        \caption{Recovering a knot from a doubly pointed Heegaard diagram}
        \label{fig:two-HD-of-3sphere-LHT}
\end{figure}
\end{example}

\begin{definition}[\cf Ozsv\'ath and Szab\'o {\cite{OS04:HFK}}, Rasmussen 
  {\cite{Ras03:HFK}}]
  Suppose that $Y$ is a $\Z$-homology sphere, which implies that there is only 
  one $\Spinc$-structure on $Y$, which is torsion. Given a doubly pointed 
  Heegaard diagram $\heeg = (\heegsurf, \alphas, \betas, z, w)$ of $(Y, K)$, 
  denote by $\heeg'$ the pointed Heegaard diagram $\heeg = (\heegsurf, \alphas, 
  \betas, z)$ of $(Y, z)$. We define the \emph{hat} and \emph{minus knot Floer 
    chain complexes} of $\mathcal{H}$,
  \[
    \CFKh (\heeg), \qquad \CFKm (\heeg),
  \]
  respectively as the chain complexes
  \[
    \CFh (\heeg'), \qquad \CFm (\heeg'),
  \]
  equipped with a relative $\Z$-filtration $A$, called the \emph{Alexander 
    filtration}, characterized by
  \[
    A (\x) - A (\y) = n_z (B) - n_w (B)
  \]
  for every $B \in \pi_2 (\x, \y)$, where $n_z$ and $n_w$ denote the 
  multiplicities of $z$ and $w$ in $B$ respectively. To ensure that $\CFKm 
  (\heeg)$ is filtered, we set the filtration level of $U$ to be $-1$.
  
  One can also define the \emph{infinity} and \emph{plus knot Floer chain 
    complexes} $\CFK^\infty (\heeg)$ and $\CFK^+ (\heeg)$ similarly.

  Note that $H_* (\CFKc (Y, K))$ is again $\HFc (Y)$, since we must ignore the 
  filtration when taking the homology. Given a $\Z$-filtered chain complex 
  $\mathcal{C} = (C, \wfilt)$, recall that its \emph{associated graded object} 
  is the chain complex
  \[
    \mathrm{g}\mathcal{C} = \bigoplus_{i \in \Z} \wfilt_i C / \wfilt_{i-1} C,
  \]
   with an internal relative $\Z$-grading given by $i$. (Colloquially, in this 
   context, we ``erase'' any term from the differential that strictly lowers 
   the filtration.) In our setting, the internal $\Z$-grading on $\gCFKc 
   (\heeg)$ is called its \emph{Alexander grading}. (It continues to have a 
   homological relative \emph{Maslov grading}.) We define the \emph{knot Floer 
     homology} of $\heeg$ to be
  \[
    \HFKc (\heeg) = H_* (\gCFKc (\heeg)).
  \]

  When $Y = S^3$, we can fix an \emph{absolute Maslov grading} $M$ on $\genset 
  (\heeg')$ by requiring that
  \[
    %\iota_* \colon H_* (\wfilt_i \CFKh (\heeg)) \to H_* (\CFKh (\heeg)) \cong 
    %\HFh (S^3) \cong \bF
    H_* (\CFKh (\heeg)) \cong \HFh (S^3) \cong \bF
  \]
  be concentrated in grading $0$. In this case, one can, with work, prove that 
  the difference between the top and bottom Alexander gradings supported in 
  $\HFKh (\heeg)$ is always $2 g (K)$, where $g$ denotes the Seifert genus 
  \cite{OS04:HFK-genus}; there is an \emph{absolute Alexander grading} $A$ on 
  $\genset (\heeg')$ characterized by the condition that the top Alexander 
  grading be $g (K)$ and the bottom $- g (K)$.
\end{definition}

%    The chain complex $\CFK^\bullet(Y,K)$ is $\CF^\bullet(Y)$, but the extra 
%    point gives a filtration on it called the \textit{Alexander filtration}, 
%    by looking at $n_w(B)$.
%We can then define
%\[
%\HFK^\bullet(Y,K):=H_*(g\CFK^\bullet(Y,K)),
%\]
%where the extra ``$g$'' means the associated graded object. Note that to take the associated graded object we will also ban $w$ in the boundary map.
%\end{definition}

Of course, the point is that these are knot invariants.

\begin{theorem}[Ozsv\'ath and Szab\'o {\cite[Theorem~3.1]{OS04:HFK}}, Rasmussen {\cite[Theorem~1]{Ras03:HFK}}]
  \label{thm:hfk-inv} Let $\heeg$ be a doubly pointed Heegaard diagram of $(Y, K)$, where $Y$ is a 
  $\Z$-homology sphere and $K$ is null-homotopic.  Then, the filtered, graded chain homotopy 
  type of $\CFKc (\heeg)$ is an invariant of $(Y, K)$ for $\mathord{\circ} \in 
  \{ \raisebox{-1ex}{$\widehat{\phantom{m}}$}, -, +, \infty \}$ (resp.\ for 
  $\mathord{\circ} \in \{ \raisebox{-1ex}{$\widehat{\phantom{m}}$}, + \}$).  
  Thus, we may write $\CFKc (Y, K)$ for these chain complexes.
\end{theorem}

%%% DIAGRAMMMMMMMMMM x2 %%%%%%%%%%

\begin{example}
  Using the doubly pointed Heegaard diagram in 
  Figure~\ref{fig:3sphere-2points}, we can calculate the \emph{hat} and 
  \emph{minus} knot Floer homologies of the left-handed trefoil to be
  %If $K$ denotes the left-handed trefoil, then
  %$\HFK^-(K)\cong\bF[U]\oplus\bF$, and $\widehat{\HFK}(K)\cong\bF^{\oplus 3}$.
  \[
    \HFKh (S^3, \overline{T}_{2,3}) \cong \bF^{\oplus 3}, \qquad \HFKm (S^3, 
    \overline{T}_{2,3}) \cong \bF[U]\oplus \faktor{\bF [U]}{U}.
  \]
\end{example}

%\subsection{Knot Floer homology: Some properties}

We conclude this discussion with some basic features of knot Floer homology:

\begin{itemize}
  \item The \emph{graded Euler characteristic} of $\HFKh (S^3, K)$ is defined 
    to be
    \[
      \chi (\HFKh (S^3, K)) = \sum_{M, A} (-1)^M t^A \cdot \rk_{\bF} \HFKh_M 
      (S^3, K, A),
    \]
    where $M$ and $A$ denote the Maslov and Alexander gradings respectively.  
    A key feature of knot Floer homology is that this turns out to be the 
    symmetrized Alexander polynomial \cite{OS04:HFK}:
    \[
      \chi (\HFKh (S^3, K)) = \Delta_K (t).
    \]
    (Recall that the Alexander polynomial is defined up to multiplication by 
    a monomial; it can be symmetrized by requiring that $\Delta_K (t^{-1}) = 
    \Delta_K (t)$.)

  \item From $\CFK^\infty (S^3, K)$, we can extract many other invariants of 
    knots in $S^3$, \eg $\tau (K)$ \cite{OS03:HFK-g4}, $\epsilon (K)$ 
    \cite{Hom14}, $\Upsilon (K)$ \cite{OSS17:Upsilon}, the torsion order of 
    $\HFKh (S^3, K)$ \cite{AE20:TFH, AE20:HFK-unknotting},
    \textit{etc.}  Many of these are powerful invariants, and some of them turn 
    out to be concordance invariants.

  \item Heegaard Floer homology can be used to study contact $3$-manifolds (via 
    open books), while knot Floer homology can be used for Legandrian and 
    transverse knots. Namely, if a $3$-manifold $Y$ is equipped with a contact 
    structure $\xi$, then there is an invariant $c (Y, \xi) \in \HFh (-Y)$ 
    \cite{OS05:HF-contact, HKM09:HF-contact}; likewise, given a Legendrian knot 
    $\Lambda \subset (Y, \xi)$, there are also invariants $\mathcal{L} (Y, \xi, 
    \Lambda) \in \HFKm (-Y, \Lambda)$ and $\widehat{\mathcal{L}} (Y, \xi, 
    \Lambda) \in \HFKh (-Y, \Lambda)$ \cite{LOSS09, OST08}.
\end{itemize}

Knot Floer homology has a vast amount of applications, and can itself be the 
subject of an entire lecture series; the above is but a condensed introduction 
that inevitably misses many other important features.

\subsection{Multi-pointed Heegaard and knot Floer homologies: Definition and relationship with grid homology}

Previously, Heegaard diagrams $\heeg = (\Sigma, \alphas, 
\betas)$ have been assumed to have exactly $g (\Sigma)$ pairs of $\alpha$- and 
$\beta$-circles, and $\{z\}$ and $\{w\}$ have been singleton sets of points. In 
fact, we can work with more general \emph{multi-pointed Heegaard diagrams} of 
the form
\[
  \mathcal{H}=(\Sigma_g,\{\alpha_1,\dotsc,\alpha_{g+n-1}\},\{\beta_1,\dotsc,\beta_{g+n-1}\},\{z_1,\dotsc,z_n\}),
\]
possibly also with $\{w_1, \dotsc, w_n\}$, for some $n \geq 1$. Such diagrams 
must satisfy the condition that each component of
\(
  \Sigma_g \setminus \alphas
\)
contain exactly one $z_i$ (and also exactly one $w_j$ if $w_j$'s exist), and 
likewise for
\(
  \Sigma_g \setminus \betas.
\)
For these multi-pointed Heegaard diagrams, we may define $\CFc (\heeg)$ or 
$\CFKc (\heeg)$ as before, except that we count holomorphic bigons in 
$\Sym^{g+n-1} (\heegsurf)$ rather than in $\Sym^g (\heegsurf)$.

The (filtered,) graded chain homotopy types of the resulting chain complexes 
turn out to coincide with those of the chain complexes defined with singly and 
doubly pointed Heegaard diagrams.\footnote{However, for knot Floer homology, 
  the issue of naturality becomes more subtle; see \cite{Sar15:HFK-basept, 
    Juh16:HFK-func, Zem19:HFK-func}.}
In the proof, one considers an additional 
Heegaard move, \emph{quasi-(de)stabilization}, where one adds (or subtracts) a 
pair of $\alpha$- and $\beta$-curves and a point $z_i$ (and $w_j$ if relevant).

For the story to come full circle, we note that grid diagrams are exactly 
multi-pointed Heegaard diagrams of knots $K \subset S^3$ that are nice in the 
sense of Sarkar and Wang \cite{SW10}.

\begin{proposition}
  Given a toroidal grid diagram $\bG= (n, \OO, \XX)$ of $K \subset S^3$, the 
  quintuple
  \[
    \heeg = (T^2, \{ \text{horizontal circles} \}, \{ \text{vertical circles} 
    \}, \OO, \XX)
  \]
  is a multi-pointed Heegaard diagram of $(S^3, K)$.
  \footnote{There exist 
    different conventions in the literature of this identification. Often, 
    $\OO$ and $\mathbf{w}$ are associated with the formal variable $U$ while 
    $\XX$ and $\mathbf{z}$ are associated with the Alexander filtration. (The 
    roles of $\mathbf{w}$ and $\mathbf{z}$ are swapped in our discussion of 
    $\CFKm$ above.) One way to resolve any discrepancy is to choose between 
    $\heeg$ and $\heeg' = (-T^2, \{ \text{vertical circles} \}, \{ 
    \text{horizontal circles} \}, \XX, \OO)$, which is also a multi-pointed 
    Heegaard diagram of $(S^3, K)$ (rather than, say, of $(-S^3, K)$ or of $(S^3, 
    -K)$).}
\end{proposition}

Thus, the knot Floer chain complex can be computed by combinatorially counting 
domains of the correct index, which are empty rectangles.

\begin{corollary}[Manolescu, Ozsv\'ath, and Sarkar {\cite{MOS09}}]
  Let $\bG= (n, \OO, \XX)$ be a grid diagram of $K \subset S^3$. Then there is 
  a chain homotopy equivalence
  \[
    \mathcal{GC}^- (\G) \simeq \CFKm (S^3, K)
  \]
  of filtered, graded chain complexes over $\bF [U]$.
\end{corollary}

\newpage
\nocite{*}
\bibliographystyle{plain}
\bibliography{references}

\bigskip
\footnotesize

C.-M.~Michael Wong, \textsc{University of Ottawa, Ottawa, Canada}\par\nopagebreak
\textit{E-mail address}: \texttt{mike.wong@uottawa.ca}

\medskip

Sarah~Zampa, \textsc{Budapest University of Technology and Economics, Budapest, Hungary}\par\nopagebreak
\textit{E-mail address}: \texttt{zampa.sarah@renyi.hu}

\end{document}